\documentclass[reqno]{amsart}

\usepackage{amsmath}
\usepackage{amsthm}
\usepackage{amssymb}
\usepackage{amsfonts,yfonts}
\usepackage{mathtools}
\usepackage[colorlinks=true,linkcolor=blue,citecolor=blue,urlcolor=blue,breaklinks]{hyperref}
\usepackage{color} % to define own color dgreen
\usepackage{cases}
\usepackage{enumitem}
\usepackage{empheq,graphicx}
\usepackage{multimedia}
\usepackage{longtable}
\usepackage{verbatim}
\usepackage[font=small,skip=2pt]{caption}
\usepackage{esint}
\usepackage{geometry}
\usepackage[normalem]{ulem} %strike out text

\graphicspath{{./}{Figures/}}

\title[Nonlinear Diffusion Equations: Full characterization of Entropies]{Nonlinear Diffusion Equations: Full characterization of Entropies}

\author{Anton Arnold}
\address{Institute for Analysis and Scientific Computing, TU Wien, Wiedner Hauptstra\ss e 8-10, 1040 Vienna, Austria}
\email{anton.arnold@tuwien.ac.at}
\author{Jose A. Carrillo}
\address{Mathematical Institute, University of Oxford, Woodstock Road,
Oxford OX2 6GG, UK}
\email{jose.carrillo@maths.ox.ac.uk}
\author{Daniel Matthes}
\address{Department of Mathematics, School of Computation, Information and Technology, TU M\"unchen, Boltzmannstra\ss e 3, 85748 Garching bei M\"unchen, Germany}
\email{matthes@ma.tum.de}

\newtheorem{lemma}{Lemma}[section]
\newtheorem{theorem}[lemma]{Theorem}
\newtheorem{corollary}[lemma]{Corollary}
\newtheorem{remark}[lemma]{Remark}
\newtheorem{prop}[lemma]{Proposition}
\newtheorem{definition}[lemma]{Definition}
\newtheorem{example}[lemma]{Example}

\newcommand{\tr}{\operatorname{tr}}

\newcommand{\setR}{\mathbb{R}}
\newcommand{\R}{\mathbb{R}}
\newcommand{\Rnn}{\mathbb{R}_{\ge0}}

\newcommand{\N}{\mathbb{N}}
\newcommand{\SSS}{\mathbb{S}}

\newcommand{\Rd}{{\R^d}}

\newcommand{\dv}{\nabla\cdot}

\newcommand{\eps}{{\varepsilon}}
\renewcommand{\H}{{\mathcal{H}}}
\newcommand{\Q}{{\mathcal{Q}}}

\newcommand{\ximin}{\xi_{\text{min}}}

\newcommand{\intRd}{{\int_\Rd}}

\newcommand{\uinf}{u_\infty}

\DeclareMathOperator{\supp}{supp}

\newcommand{\be}{\begin{equation}}
\newcommand{\ee}{\end{equation}}

\definecolor{dgreen}{rgb}{0,.8,0.2}
\definecolor{brown}{rgb}{0.6,0.5,0.2}

\begin{document}

\date{\today}
\begin{abstract}
This paper is concerned with the large-time behavior of quasilinear Fokker-Planck equations with confinement on the whole space $\R^d$. It aims at characterizing all relative entropy functionals such that the entropy method \`a la Bakry-\'Emery yields exponential convergence of all solutions towards the unique steady state (with the same mass as the initial condition). 
We call such entropies admissible. 
The convergence rate is determined by the uniform convexity parameter of the confinement potential. As such, this program extends the analogous study of linear Fokker-Planck equations \cite{BaEm85, AMTU} to the nonlinear case, and it derives additional functionals for the nonlinear case --- beyond the Ralston-Newman entropies used in \cite{CJMTU}. 

Two key results are the characterization of those nonlinear Fokker-Planck equations which admit all entropy functionals that are admissible for the corresponding linear Fokker-Planck equation, and vice versa, the characterization of all admissible entropies for a given nonlinearity. The latter quest for power-law nonlinearities yields a large family of entropies for the porous-medium equations, but only the Ralston-Newman entropy for the fast-diffusion equations. Additional results include the derivation of new generalized Csisz\'ar-Kullback and generalized Log-Sobolev inequalities for our entropy functionals as well as moment-weighted $L^1$--convergence estimates for the Fokker-Planck solutions.
\end{abstract}

\maketitle

\section{Introduction}

In this article, we give a comprehensive analysis of exponentially-in-time decaying Lyapunov functionals
for nonlinear diffusion equations of Fokker-Planck type on $\R^d$, in arbitrary space dimensions $d\ge1$,
\begin{align}
  \label{eq:PDE0}
  \partial_t u = \Delta P(u) + \dv(u\,\nabla V),\quad t>0,
\end{align}
with $u(x,t)\ge0$, a strictly monotone nonlinearity $P$, and a uniformly convex confinement potential $V:\R^d\to\R$. It is well-known \cite{McCann,otto,CJMTU} that under mild additional hypotheses on $P$, there exists a unique weak stationary solution $u_\infty$ of any prescribed mass $M>0$, satisfying
\begin{align*}
  \phi\big(u_\infty(x)\big) + V(x) \ge \bar C,
  \quad
  \phi(r):=\int_1^r\frac{P'(s)}{s}\,ds,
\end{align*}
with a suitable constant $\bar C$ depending on $M$, for each $x\in\R^d$, with equality for each $x$ in the support of $u_\infty$, see \S\ref{sec:equilibria} for a recap. We shall provide a detailed analysis of the decay behaviour for a class of functionals $\mathcal H_g(u(t)|u_\infty)$ that quantify the proximity of the solution $u(t)$ to $u_\infty$. It is well-know that, in the linear case, i.e.\ \eqref{eq:PDE0}  with $P(u)=u$, there exists a large family of such functionals (or ``relative entropies'') \cite{BaEm84, AMTU}. 
Our main goal is now to find the largest possible family of relative entropy-like functionals for the nonlinear case. 

Specifically, we consider in \S\ref{sec:gen-entropies} functionals of the form
\begin{align}
  \label{eq:H}
  \mathcal H_g(u|u_\infty) := \int_{\R^d} G\big(u(x),u_\infty(x);x\big)\,dx, 
\end{align}
where the nonlinear function $G$ is given in terms of another, monotone function $g$,
\begin{align}
  \label{eq:G}
  G(r,r';x) := \int_{r'}^r g\big(\phi(s)+V(x)-\bar  C\big)\,ds .
\end{align}
Note that this definition strongly depends on the choice of $P$, via the function $\phi$. 
In the non-degenerate situation, i.e.\ $P'(0+)>0$, where $\phi(u_\infty(x))=\bar C-V(x)$ for all $x\in\R^d$, the explicit $x$-dependence can be eliminated from \eqref{eq:G}, which then simplifies to
\begin{align*}
%    \label{eq:GG}
    G(r,r') = \int_{r'}^r g\big(\phi(s)-\phi(r')\big)\,ds.
\end{align*}
The $g$ above will be our primary parameter throughout this work, our results are most easily formulated in terms of $g$. The simplest choice for $g$ is the identity, $g_1(\xi)=\xi$. This gives rise to the functional
\begin{align*}
%    \label{eq:canonical}
    \mathcal H_{g_1}(u|u_\infty) 
    = \int_{\setR^d} \big[\Phi(u) + u V\big]\,dx
    - \int_{\setR^d} \big[\Phi(u_\infty) + u_\infty V\big]\,dx\,,
    \quad \Phi(r):=\int_0^r\phi(s)\,ds.
\end{align*}
In the context of the porous medium equation, i.e. $P(r)=r^m$ with $m>1$, this functional $\mathcal H_{g_1}$ is known as the \emph{Ralston-Newman} entropy \cite{New84,Ral84}. In other communities this entropy is also referred to as \emph{Tsallis entropy} \cite{T88,To14,CT14}, and applied to the porous medium equation without confinement potential, see \cite{To14,CT14} for other generalized R\'enyi entropies. This canonical entropy choice for the family of PDEs \eqref{eq:PDE0} was extensively studied in \cite{otto,CT01,CJMTU,DD02,J16}.

We note that the entropies constructed here resemble those from \cite{BLMV14} where slightly different nonlinear drift-diffusion equations in bounded domains with Dirichlet boundary conditions are discussed, see Remark \ref{rem:ent} for more details.  

The main goal of this paper is to identify relative entropy-type functionals \eqref{eq:H} that decay exponentially in time for the nonlinear equations \eqref{eq:PDE0}, particularly with the rate $2\lambda$ specified by $\nabla^2 V(x)\ge \lambda I_d$ on $\R^d$. Our ansatz is to extend to a certain degree the celebrated approach by Bakry and Émery from the 1980's from the linear to the nonlinear setting. Indeed, recall that if \eqref{eq:PDE0} is \emph{linear} and thus $u_\infty(x)=e^{-V(x)}/Z$ with suitable normalization constant $Z$, a large variety of exponentially decaying functionals has been identified in \cite{BaEm84}. The following is a particular result from that work, and following the Bakry-\'Emery method it is sharp under certain assumptions, see \cite[Section 3.5]{AMTU} and the discussion in \S\ref{sct:linear}.
\begin{theorem}[\cite{BaEm84,BaEm85,BaEm85a, AMTU}]
  \label{thm:BE}
  Assume $P(u)=u$.
  For each $p\in[1,2]$, the functionals
  \begin{align}
    \label{eq:Hp}
    \mathcal H_{g_p}(u|u_\infty) :=
    \begin{cases}
      \displaystyle \int_{\R^d} u\log\left(\frac{u}{u_\infty}\right) \,dx & \text{if $p=1$}, \\[3mm] 
      \displaystyle \frac1{p-1}\int_{\R^d} \left[\left(\frac{u}{u_\infty}\right)^p-1\right]u_\infty\,dx & \text{if $1<p\le2$},
    \end{cases}
  \end{align}
  decay to zero at exponential rate $\exp(-2\lambda t)$.
\end{theorem}
The functionals from \eqref{eq:Hp} are a special case of \eqref{eq:H} above, obtained by choosing 
\begin{align}
  \label{eq:gp}
  g_p(\xi)=\frac{p\,\left(e^{(p-1)\xi}-1\right)}{p-1},
\end{align}
for $p\in(1,2]$.
Functionals for different $p$'s provide different information about the tail behaviour of $u$. 
For instance, $\mathcal H_{g_1}(u|u_\infty)<\infty$ implies that $u$ has a finite second moment,
but no better in general,
while $\mathcal H_{g_2}(u|u_\infty)<\infty$ implies that \emph{all} polynomial moments of $u$ are finite. Moreover, their exponential-in-time decay provides convergence of $u(t)$ to $u_\infty$ in different topologies. 
In this context note that, for a fixed $u$, the values $\H_{g_p}(u|u_\infty)$ are increasing in $p\in[1,2]$. The objective of \S\ref{sct:linear} is to reformulate the linear diffusion results in a framework suitable for a nonlinear extension. 

We pursue two different directions of extending Theorem \ref{thm:BE} to nonlinear equations \eqref{eq:PDE0}. The first is to determine nonlinearities $P$ for which the $g_p$ from \eqref{eq:gp}
still give rise to exponentially decaying functionals $\mathcal H_{g_p}$. Note that, for a nonlinear diffusivity $P$, the definition of the functional $\mathcal H_{g_p}$ in \eqref{eq:H} does not simplify any more to the form \eqref{eq:Hp}.
One of our results in this direction is that Theorem \ref{thm:BE} carries over to a variety of $P$'s that are ``almost linear'' for $u\searrow0$ and $u\to\infty$ (see \S\ref{sec:admiss-sets-p} for details). More generally, \S\ref{sec:entropygivennonlin} is devoted to finding nonlinearities $P$ such that a relative entropy (characterized by a given $g$) is exponentially decaying. 

The second direction, see \S \ref{sec-quasilin}, is to assume $P$ given and to determine all entropies of the form \eqref{eq:H} for which the Bakry-\'Emery method yields exponential decay. Moreover, we want to identify (whenever it exists) an analog of the ``strongest'' entropy $\mathcal H_{g_2}$ from the linear case. That is, we look for $g_2$ that is in a certain sense maximal among all $g$'s which give rise to an exponentially decaying $\mathcal H$. In a variety of cases, we are able to provide a characterization of this maximal $g_2$ in terms of a nonlinear ODE involving $P'$, and there is an interpolating family $g_p$ still giving rise to exponentially decaying entropies. 

An important special case is the power law nonlinearities $P(r)=r^m$ with $m\ge\frac{d-1}{d}$. It turns out (see Proposition \ref{prop:genm}) that the fast-diffusion type equations with $\frac{d-1}{d}\le m<1$ only admit a single entropy, i.e.\ the Ralston-Newman entropy. The reason for this phenomenon is that the just mentioned nonlinear ODE  admits only $g_1$ as a global solution. By contrast, linear Fokker-Planck equations (with $m=1$) and porous-medium type equations (with $m>1$) allow for a large family of entropies. 
In the latter case all the $g_p$ are explicit (see Subsection \ref{sct:power-laws}): With $g_1(\xi)=\xi$ and, e.g.\ for $m=2$, they take the form
\begin{align*}
  g_p(\xi) = B_p\big[ (\xi-\ximin)^{1+\frac{p-1}{\kappa}}-(-\ximin)^{1+\frac{p-1}{\kappa}}\big], \quad p\in(1,2],\quad \kappa=\frac98\frac{d}{d+1}
\end{align*}
with appropriate normalization constant $B_p>0$.
Here, $\ximin:=\phi(0+)+\inf V-\bar C$ is the smallest value attainable by $\phi(u)+V-\bar C$.
Moreover, for any $m>1$ the functionals $\mathcal H_{g_p},\,p\in[1,2]$ are again ordered
such that $\mathcal H_{g_1}$ provides control of the second moment but none above,
while $\mathcal H_{g_2}$ allows to control moments up to a finite order larger than two depending on $m,\,d$ (see \S\ref{prop:momentcontroldegenerate} for details).

Our strategy of proof in \S\ref{sec:BakEm} is a further nonlinear generalization of the Bakry-\'Emery method \cite{BaEm85,AMTU} with respect to \cite{CT01,otto,CJMTU}: we establish a linear control on the second time derivative of the entropy functional by the first time derivative, which then implies a linear control of the entropy dissipation in terms of the entropy itself, and eventually leads to exponential decay via the Gronwall inequality. A significant observation is that this procedure requires a particular condition on the nonlinearity $P$, namely that 
$(d-1)P(u) \leq d u P'(u)$. 
This condition is known as \emph{McCann condition} in the context of the representation of diffusion equations as gradient flows in the $L^2$-Wasserstein metric. It is the sharp condition for the canonical entropy $\H_{g_1}$ to be displacement convex. 

In \S\ref{sec-funcineq} we present three applications of our novel entropy decay results. We first derive generalized Csisz\'ar-Kullback inequalities involving the functionals $\H_g$, hence enlarging the set of possible entropies and nonlinear diffusions compared to \cite{CJMTU} for which exponential decay in $L^1$ at rate $\lambda$ is obtained. As mentioned earlier we find new moment estimates on the solutions for degenerate diffusions with nonlinearity $P$ controlled by the generalized entropy $\H_g$. This also implies exponential convergence in moment-weighted $L^1$-spaces. The nonlinear entropy method from \S\ref{sec:BakEm} also provides new generalized Log-Sobolev inequalities between the entropy $\H_g$ and its entropy dissipation, expanding the results from \cite{CJMTU,DD02}. 

For the linear case sharpness of these generalized Log-Sobolev inequalities was analyzed in \cite{AMTU} for the canonical entropy $\H_{g_1}$. In \S\ref{sec:sharp} we give an analogous characterization for nonlinear diffusions and identify the corresponding optimal functions (being translates and/or scaled versions of $u_\infty$). Like in the linear case, sharpness can only occur if the confinement potential is quadratic in at least one coordinate direction on the (possibly compact) support of $u_\infty$.
\medskip

For convenience of the reader, we include a list of symbols in Appendix \ref{sec-appendix-symbols}.

%%%%%%%%%%%%%%%%%%%%%%%%%%%%%%%%%%%%

\section{Nonlinear Diffusions: Equilibrium States \& Standard Entropy}\label{sec:equilibria}
\setcounter{equation}{0}

We will be interested in the asymptotic behavior of solutions to the Cauchy problem for the general nonlinear Fokker-Planck equation
\begin{equation}
    \partial_t u=\nabla\cdot (u \nabla V(x) +  \nabla P(u)),\qquad x\in \R^d, t>0,
\label{GFP}
\end{equation}
\begin{equation} u(x,t=0) = u_0(x) \geq 0 , \qquad x \in \R^d\,,
\label{GFP2}
\end{equation}
with initial data given by mass densities, i.e., $u_0\in L^1(\R^d)$, $u_0\ge0$ and 
$$
\int_{\R^d} u_0(x)\;dx =: M\in(0,\infty)\,.
$$
We assume that the external potential $V(x)$ is confining in the following sense
\begin{itemize}
\item[{\bf (HV1)}] $V\in W^{3,1}_{loc}(\R^d)$. 
\item[{\bf (HV2)}] $V$ is uniformly convex: $\exists\lambda>0$ such that $\nabla^2 V(x)\geq \lambda I_d$ for all $x\in\R^d$, and without loss of generality $\min_{\R^d} V=0$ and this minimum ist attained at $x=0$. 
\end{itemize}
Note that the uniform convexity of $V$ implies $V$ is bounded from below, and then the assumption $\min_{\R^d} V=0$ is not restrictive.
  
\begin{remark}\label{data1}
The previous set of assumptions on the potential will play an important role in the rates of decay for the family of equations \eqref{GFP}. However, if we are only interested in the existence of stationary states for \eqref{GFP}, then a much less restrictive set of assumptions on the potential is sufficient. We will recall below a result from \cite{CJMTU} which is based on the following assumptions:
\begin{itemize}
\item[{\bf (HV1')}] $V\in W^{1,1}_{loc}(\R^d)$.
\item[{\bf (HV2')}] $\forall A\in\R$: $\{x\in\R^d | V(x)\le A\}$ is
bounded. 
\item[{\bf (HV3')}] $V$ is bounded from below, and without loss of generality $\inf_{\R^d} V=0$. 
\end{itemize}
It is obvious that {\bf (HV1)}-{\bf (HV2)} imply {\bf (HV1')}-{\bf (HV3')}.
Note that {\bf (HV2)} implies $V(x)\to\infty$ as $|x|\to\infty$. 
\end{remark}

The nonlinearities allowed for the diffusive term satisfy the following basic assumptions
\begin{itemize}
\item[{\bf (HP1)}] $P:\R^+_0\!\to\!\R$ is continuous, strictly increasing, $P|_{\R
^+}\!\!\in C^1(\R^+)$, and $P(0)=0$. Moreover, $P'(0+)$ exists in $[0,\infty]$ and if $P'(0+)=0$ we further assume that $P''(0+)$ is finite.
\item[{\bf (HP2)}] The function $\phi$, defined by 
\begin{align}
    \label{eq:P2phi}    
    \phi(u) := \int_1^u \frac{P'(r)}{r}\;dr,\quad \text{for all $u\in(0,\infty)$},
\end{align}
belongs to $L^1_\text{loc}([0,\infty))$.
\end{itemize}

\begin{remark}%\label{data2}
The assumption {\bf (HP2)} implies that
\[
\Phi : [0,\infty) \to \R,\quad \Phi(u):= \int_0^u \phi(s)\;ds 
\]
is well-defined with $\Phi(u)=u\phi(u)-P(u)$, $\Phi'(u)=\phi(u)$ and $u\Phi''(u)=P'(u)$ for all $u\in\R^+$. Since $P$ is strictly increasing then $\phi$ is strictly increasing and the function $\Phi$ is strictly convex. We point out that due to this structural assumption, equation \eqref{GFP} can be written as
\begin{equation*}
    \partial_t u=\nabla\cdot\left(u \nabla \left[V(x) +  \phi(u)\right]\right),\qquad x\in \R^d, t>0.
%\label{GFP3}
\end{equation*}
We also note that the last part of hypothesis {\bf (HP1)} will allow us to distinguish degenerate from non-degenerate cases in Definition \ref{degversusnondeg} below.
\end{remark}

\begin{remark}\label{data3}\
\begin{itemize}
\item [(a)] Throughout this paper we shall be concerned with non-negative
solutions $u$ to \eqref{GFP}-\eqref{GFP2}.
\item [(b)] Canonical examples for $P$ are $P(u)=u^m$ with $m\in(0,\infty)$ and nonlinearities $P(u)$ with power law behaviors close to 0 and/or $\infty$. 
The former leads to $\phi(u)=\frac{m}{m-1}(u^{m-1}-1)$ for $m\ne1$, and $\phi(u)=\log u$ for $m=1$.
\item [(c)] Since $P$ is strictly increasing, $\phi$ is a homeomorphism from
$(0,\infty)$ onto the open interval $(\inf \phi,\sup \phi)=(\phi(0+),\phi(\infty))$ such that $-\infty\le \phi(0+)< 0 < \phi(\infty) \le\infty$ holds.
\item [(d)] It is easy to verify that $\min\Phi = \Phi(1) < 0$ and $\lim_{s\to\infty} 
\Phi(s)=\infty$. From the convexity of $\Phi$ we deduce: There is $s_\circ\in 
(1,\infty)$ such that $\Phi$ is decreasing 
and non-positive on $[0,1]$, increasing and non-positive on $[1,s_\circ]$, 
and increasing and non-negative on $[s_\circ,\infty)$.
\end{itemize}
\end{remark}

In order to understand better the structure of \eqref{GFP}, we introduce the following standard entropy functional introduced in \cite{otto,CJMTU}.

\begin{definition}\label{standentropy}
We define the standard entropy functional $E : L^1_+(\R^d) \to  \R\cup\{\infty\}$ associated to \eqref{GFP} as 
\begin{align*}
E(u) := \int_{\R^d} (Vu + \Phi(u))(x)\;dx , \qquad \mbox{for } \Phi^-(u)\in L^1(\R^d)\,, 
\end{align*}
where $L^1_+(\R^d):=\{u\in L^1(\R^d): u\ge 0\}$ and $\Phi^-(u)=\min (\Phi(u),0)$. 
\end{definition}

It is well-known by now that equation \eqref{GFP} can be understood as a gradient flow of the entropy functional $E$ in the sense of probability measures endowed with the euclidean Wasserstein distance, see \cite{jko,otto, AGS,CMCV06}. Moreover, for potentials satisfying {\bf (HV1)}-{\bf (HV2)} the Cauchy problem is well-posed by variational schemes in the set of densities $L^1_+(\R^d)$ with initial finite entropy. Moreover, it was shown that solutions satisfy the following entropy dissipation identity
\begin{equation*}%\label{disstaent}
  \frac{d}{dt} E(u) = - \int_{\R^d} u \big|\nabla\big(V(x) +  \phi(u)\big)\big|^2 \;dx\leq 0\,.
\end{equation*}
Therefore, equilibrium solutions to \eqref{GFP} should satisfy
\begin{equation*}%\label{stat1}
  V(x) +  \phi(u) = C \quad \mbox{for all } x\in\mbox{\rm supp}(u)\,.
\end{equation*}
Let us point out that depending on the assumptions on the potential $V$ and the nonlinearity $P(u)$, the characterization can be tricky and the zoology of equilibrium solutions might be substantial. This is due to the possible degeneracy of the nonlinearity $P(u)$ at zero allowing for compactly supported steady states with possibly different connected components in its support depending on $V$ --- if it is not convex. 

In the whole generality of assumptions {\bf (HV1')}-{\bf (HV3')} and {\bf (HP1)}-{\bf (HP2)}, we define equilibrium solutions as:

\begin{definition}%\label{stattues} 
Assume {\bf (HV1')-(HV3')}, {\bf (HP1)-(HP2)}. A function $\uinf \in L^1_+(\R^d)$ is an {\rm equilibrium} solution of \eqref{GFP} if and only if $\uinf$ is a global minimizer of $E$ (with $|E(u_\infty)|<\infty$) in 
\[ 
L^1_M:= \left\{ u\in L^1_+(\R^d): \int_{\R^d} u(x)\;dx = M \right\}.
\]
\end{definition}

We refer to \cite[Subsection 3.1]{CJMTU} for a thorough study of the properties for equilibrium solutions to \eqref{GFP}. We just remind the reader of the most important aspects related to our discussion here. It can be proved that an equilibrium solution $u$ to \eqref{GFP} satisfies the corresponding Euler-Lagrange equations:
\begin{equation}\label{stat2} 
\begin{array}{rcl} V(x)+\phi(u(x)) = C & , &\mbox{if}\; u(x)>0 
\\ V(x)+\phi(u(x)) \ge C& , &\mbox{if}\; u(x) =0
\end{array},\end{equation}
with $C\in\R$ a constant such that $u\in L^1_M$. Note that if $\phi(0+)=-\infty$, then the identity $V(x)+\phi(u(x))=C$ holds for all  $x\in\R^d$.

Due to these Euler-Lagrange equations, one can parameterize the set of possible equilibrium solutions by the constant $C$. In fact, let us denote by $U(.,C)$ the solution to \eqref{stat2} for every $C\in\R$. We will be looking for those functions satisfying the mass constraint $U(\cdot,C)\in L^1_M$. The explicit expression of $U(x,C)$ is
\begin{equation}
U(x,C):=\overline{\phi}^{-1}(C - V(x)),
\label{stat1a} \end{equation}
with the ``generalized'' inverse $\overline{\phi}^{-1}$ given by
\begin{equation}\label{phi-inv}
\overline{\phi}^{-1}  :  \R  \to [0,\infty],\quad \overline{\phi}^{-1}(\sigma) :=
 \left\{ \begin{array}{rcl}
0 & , & \sigma\le \phi(0+) \\[0.3cm]
\phi^{-1}(\sigma) & , & \phi(0+) <\sigma< \phi(\infty) \\[0.3cm]
\infty & , & \phi(\infty) \le \sigma \end{array} \right. .
\end{equation}
Let us point out that if $\phi(0+)>-\infty$, all functions $U(x,C)$ are compactly supported and therefore integrable. However, if $\phi(0+)=-\infty$ the integrability of $U(x,C)$ is not given by our assumptions. Therefore, we need a further assumption mixing the nonlinearities and 
the potential:
\begin{itemize}
\item[{\bf (HPV)}] $U(x,C)\in L^1(\R^d)$ for all $C\in\R$.
\end{itemize}

Under the above assumptions the following characterization of equilibrium solutions is proven in Lemma 6 of \cite{CJMTU}. 

\begin{prop}\label{charac} 
Assume {\bf (HV1')-(HV3'), (HP1)-(HP2)}, and {\bf (HPV)}. Then there is a unique minimizer $\uinf$ of $E$ in $L^1_M$ for all masses $M$, i.e., there is a unique equilibrium solution of \eqref{GFP} with mass $M$. Moreover, there exists a unique $\bar C\in \R$ such that $\uinf(x)=U(x,\bar C)$ with $U$ given by \eqref{stat1a}.
\end{prop}

Let us point out that the previous proposition does not imply the uniqueness of steady states for \eqref{GFP}. For instance, taking $V(x)=x_1^4-2x_1^2+|x|^2$ and $P(u)=u^m$, $m>1$, the reader can check that for small enough mass $M$, there are infinitely many stationary states. Actually, one can construct them by filling each of the two wells of the potential at different levels. However, only the one that fills each well at the same height gives the equilibrium solution.

Notice that we dropped the dependence of the equilibrium solution on the mass $M$ for notational simplicity. 

\begin{definition}\label{degversusnondeg}
We will say that the diffusion function $P(u)$ or the equation \eqref{GFP} is degenerate if $\phi(0+)>-\infty$ and non-degenerate if $\phi(0+)=-\infty$.
\end{definition}

\begin{remark}\label{degen}\
\begin{itemize}
\item [(a)] Notice that due to assumptions {\bf (HP1)-(HP2)}, $P(u)$ is degenerate if and only if $P'(0+)=0$. This is a consequence of Taylor expansion at 0 using the last part of hypothesis {\bf (HP1)}. Consequently, $P(u)$ is non-degenerate if and only if $P'(0+)\in(0,\infty]$.

\item [(b)] In the degenerate case $\phi(0+)>-\infty$, $\phi(\uinf)+V(x)=\bar C$ for all $x$ in the support of $\uinf=U(x,\bar C)$, which is compact. Therefore, at the boundary of $\supp\uinf$ we have $\phi(0+)+V(x)=\bar C$. Hence, due to {\bf (HV3')} we deduce that $\bar C> \phi(0+)$ if $\uinf$ has positive mass.
\end{itemize}
\end{remark}

Associated to the unique equilibrium we define the following relative entropy functional:

\begin{definition}\label{standrelentropy}
The relative entropy functional to \eqref{GFP}, $E(.|\uinf) :  L^1_M  \to [0,\infty]$, is given by
$$
 E(u|\uinf) := \int_{\R^d} (\Phi(u)-\Phi(\uinf)
- \Phi'(\uinf)\,(u-\uinf))(x)\;dx\,.
$$
\end{definition}

Let us remark that due to convexity, 
\[ 
\Phi(u)-\Phi(u_{\infty})
- \Phi'(u_{\infty})\,(u-u_{\infty})(x)\ge0
\]
for all $x\in\R^d$ such that the integral in the definition of $E(.|\uinf)$ has a well-defined value in $[0,\infty]$. The following relation between the entropy and the relative entropy functionals is proven in \cite{CJMTU}.

\begin{prop}\label{key} 
Assume {\bf (HV1')-(HV3'), (HP1)-(HP2)}, and {\bf (HPV)}. Furthermore, assume $|E(\uinf)|<\infty$, then
\begin{equation}\label{superkey}
E(u)-E(\uinf)\ge E(u|\uinf), \qquad \mbox{for all } u\in L^1_M,
\end{equation}
where equality holds for all $u\in L^1_M$ if and only if 
\[ V(x) + \phi(\uinf(x)) = \bar C,\quad\mbox{for almost all $x\in\R^d$}.\]
\end{prop}

\begin{remark}
Based on Proposition \ref{key} and the results in \cite{CJMTU}, the relative entropy functional $E(u|\uinf)$ is not of much use in the degenerate diffusion case $\phi(0+)>-\infty$. In fact, it is more useful to think about $E(u)-E(\uinf)$ as the standard relative entropy functional. We will do so when defining more general relative entropy functionals in the next section.
\end{remark}

\begin{remark}
Let us consider the special non-degenerate nonlinearity $P(r)=r^{\frac{d-1}{d}}$ for $d=2$ which will be of special interest for sharpness results in \S\ref{sec:sharp}.  In this case we have (for each fixed $M>0$) the unique steady state 
$$
  u_\infty(x)=\Big(1-\bar C+\frac{\lambda}{2}|x|^2\Big)^{-2}. 
$$
With $\Phi(r)=-2 r^{1/2}+r$. One finds from Definition \ref{standentropy} that $E(u_\infty)=-\infty$. Still, for appropriate $u$, the relative entropy functional $E(u|u_\infty)$ may be finite. Working with the relative entropy functional instead of entropies satisfying $|E(u_\infty)|<\infty$ allows to include the class of more general nonlinearities for which $E(u_\infty)=-\infty$ in the relative entropy method, as in \cite{LeMa}. We prefer to confine ourselves to $|E(u_\infty)|<\infty$ since we want to focus on the question of the maximal set of entropy functionals, specifically in degenerate cases where this is not an issue. However, this particular example will play a role in Proposition \ref{prop:genm} and in the sharpness of certain results for $d=2$, see \S\ref{sec:sharp}.
\end{remark}    

In the rest of this work, we will assume that our potentials and nonlinearities satisfy the assumptions {\bf (HV1)-(HV2), (HP1)-(HP2)}, and {\bf (HPV)} in order to discuss convergence rates to equilibrium. Notice that in this case there is a unique equilibrium $\uinf$ of \eqref{GFP} for each positive mass $M$ due to \eqref{stat2}, and $\uinf$ is the global minimizer of the entropy in Definition \ref{standentropy}.

%%%%%%%%%%%%%%%%%%%%%%%%%%%%%%%%%%%%%

\section{Nonlinear Diffusions: General Entropies}\label{sec:gen-entropies}
\setcounter{equation}{0}

Let us first consider the general setting of nonlinear \emph{non-degenerate} diffusion equations, i.e., we assume that $\phi(0+)=-\infty$. According to the discussion in \eqref{stat2} and Proposition \ref{charac}, the equilibrium then satisfies $\nabla [\phi(\uinf(x))+ V(x)]=0$ for $x\in\R^d$, and thus the equation \eqref{GFP} can be written as
\begin{align}
  \label{eq:6}
  \partial_t u = \Delta P(u) + \nabla\cdot(u\nabla V) = \nabla\cdot\big(u\nabla\big[\phi(u)-\phi(u_\infty)\big]\big)\,.
\end{align}
In the following, we write equation \eqref{eq:6} as
\begin{align}
  \label{eq:3}
  \partial_t u = \nabla\cdot(u\nabla\xi) , 
  \quad \text{with} \quad \xi:=\phi(u)-\phi(u_\infty)=\phi(u)+V-\bar C,
\end{align} 
since Proposition \ref{charac} implies $\phi(u_\infty)=\bar C-V(x)$ holds on $\Rd$ in the non-degenerate case.

Let a strictly increasing $C^3$ function $g:\setR\to\setR$ with $g(0)=0$ be given; we shall refer to $g$ as an \emph{entropy generating function}. Let us define $G:\setR^+_0\times\setR^+\to\setR\cup \{\infty\}$ as 
\begin{align}
  \label{eq:7}
   G(a,b) = \int_b^a g(\phi(s)-\phi(b)) \, ds.
\end{align}   
With this notation, we can define our general notion of relative entropy.

\begin{definition}\label{genrelentropy}
The general relative entropy functional associated to \eqref{GFP} and $g$ for non-degenerate diffusions characterized by $\phi(0+)=-\infty$, is given by $\mathcal{H}_g(.|\uinf) :  L^1_M  \to [0,\infty]$,
$$
 \mathcal{H}_g(u|\uinf)  := \int_{\R^d} G(u,\uinf)(x)\;dx\,.
$$
\end{definition}

Notice that by Taylor expansion of $G(u,u_\infty)$ in the first variable about $u_\infty$, we get
that the integrand is non-negative by convexity of $\Phi$:
\begin{equation}\label{Taylorent}
 \mathcal{H}_g(u|u_\infty) = \frac12 \int_{\R^d} g'(\phi(\eta(x))-\phi(\uinf(x))) \phi'(\eta(x)) (u(x)-\uinf(x))^2 \,dx\geq 0\,,
\end{equation}
with $\eta(x)$ between $u(x)$ and $u_\infty(x)$ for all $x\in\R^d$. Therefore, $\mathcal{H}_g(u|u_\infty)$ is well-defined.
Moreover, due to the (strictly) increasing character of $P$ and $g$, $\mathcal{H}_g(u|\uinf)=0$ if and only if $u=\uinf$.

\begin{example}\label{sam1}
The canonical example is $g_1(\xi)=\xi$, leading to
\begin{align*}
  G_1(a,b) = \Phi(a)-\Phi(b)-\Phi'(b)(a-b)\,
\end{align*}
recovering the  
relative entropy functional from Definition \ref{standrelentropy}: 
$\mathcal{H}_{g_1}(u|u_\infty)=E(u|u_\infty)$.
In the linear diffusion case, $P(u)=u$, $\mathcal{H}_{g_1}(u|u_\infty)$ is the Boltzmann logarithmic entropy.\qed
\end{example}

Recall that $\phi(0+)=-\infty$ in the non-degenerate case. Thus, for the functional $\H_g$ to be finite for functions $u$ that vanish on sets with positive measure, we have to require that $G(0,b)<\infty$ for all $b\in(0,\sup u_\infty]$. Thus we shall require for non-degenerate diffusion equations that the following integrability condition is satisfied:
\begin{equation}\label{int-cond}
    \mbox{The scalar function}\quad g(\phi(s)-\phi(\sup u_\infty))\quad \mbox{is integrable at } s=0.
\end{equation}
The subsequent example illustrates this situation.

\begin{example}%\label{ex-integrable}
Let $P(u)=Du$ with some $D>0$, hence $\phi(u)=D\log u$ and $u_\infty(x)=c\,\exp(-\frac{V(x)}{D})$. Let $g(\xi)=1-\exp(-\frac\xi2)$, and hence
$$
  G(a,b)=\int_b^a g\left(D\log\frac{s}{b}\right)ds 
   = a-b-\frac{2b}{2-D}\Big[\left(\frac{a}{b}\right)^{1-\frac{D}{2}}-1\Big]. 
$$
The above integrability condition for
$$
  g(\phi(s)-\phi(\sup u_\infty)) = 1-\left(\frac{s}{c}\right)^{-\frac{D}{2}}
$$
at $s=0$ is satisfied iff $0<D<2$.
\qed
\end{example}

\medskip
Now, if the nonlinear diffusion is allowed to \emph{degenerate} at zero, i.e., $\phi(0+)>-\infty$, we need to do some adjustments similar to \eqref{superkey} in Proposition \ref{key}. In fact, since $\phi(\uinf(x))+ V(x)=\bar C$ only on the support of $\uinf$, the two representations of $\xi$ in \eqref{eq:3} do not coincide anymore, and the correct representation of equation \eqref{eq:6} is
\begin{align}
  \label{eq:3b}
  \partial_t u = \nabla\cdot(u\nabla \xi) , 
  \quad \text{with} \quad \xi=\phi(u)+V-\bar C\,.
\end{align}
Given a strictly increasing $C^3$ function $g: (\xi_{\min},\infty)\to\setR$ with $g(0)=0$ and $\xi_{\min}\in (-\infty,0)$ defined later, we define the function $\tilde G:\setR^+_0\times\setR^+_0 \times\R^d\to\setR$ as 
\begin{align}\label{eq:7b}
   \tilde G(a,b;x) = \int_b^a g(\phi(s)+V(x)-\bar C) \, ds.
\end{align}  
We shall refer to $g$ as an \emph{entropy generating function}. 
With this notation, we can define our general notion of relative entropy in the degenerate case.

\begin{definition}\label{genrelentropy2}
The general relative entropy functional associated to \eqref{GFP} and $g$ for degenerate diffusions $\phi(0+)>-\infty$, is given by $\widetilde{\mathcal{H}}_g(.|\uinf) :  L^1_M  \to [0,\infty]$,
$$
 \widetilde{\mathcal{H}}_g(u|\uinf)  := \int_{\R^d} \tilde G(u,\uinf;\cdot)(x)\;dx\,.
$$
\end{definition}

\begin{example}\label{Ex:canon-entropy}
The canonical example $g_1(\xi)=\xi$ leads in this case to
\begin{align*}
  \tilde G_1(a,b;x) = \Phi(a)-\Phi(b)+(V(x)-\bar C)(a-b)\,
\end{align*}
recovering $\widetilde{\mathcal{H}}_{g_1}(u|u_\infty)=E(u)-E(u_\infty)\geq E(u|u_\infty)$ by Proposition \ref{key}. In the degenerate diffusion power-law case, i.e.\ $P(u)=u^m$ with $m>1$, $\widetilde{\mathcal{H}}_{g_1}(u|u_\infty)$ is the Ralston-Newman entropy \cite{Ral84,New84}.\qed
\end{example}

Notice that by the definitions of $G$ in \eqref{eq:7} for non-degenerate diffusions, and of $\tilde G$ in \eqref{eq:7b} for degenerate diffusions, we deduce that
\begin{align}\label{relbetent}
\partial_a \tilde G(a,\uinf(x);x) &= g(\phi(a)+V(x)-\bar C) \nonumber\\
&\geq g(\phi(a)-\phi(\uinf(x))) = \partial_a G(a,\uinf(x)) 
\end{align}
since $\uinf$ satisfies \eqref{stat2}, $g$ is increasing, and taking into account Remark \ref{degen} in the last inequality. Integrating in \eqref{relbetent} between $u_\infty(x)$ and $u(x)$ in $a$, we conclude that
$$
\widetilde{\mathcal{H}}_g(u|\uinf)\geq \mathcal{H}_g(u|\uinf)
$$
for all $u\in L^1_M$. Therefore, the functional $\widetilde{\mathcal{H}}_g(u|\uinf)$ in Definition \ref{genrelentropy2} is well-defined. Notice that in the non-degenerate case $G(u(x),u_\infty(x))=\tilde G(u(x),\uinf(x);x)$, and then as a consequence both relative entropies coincide $\widetilde{\mathcal{H}}_g(u|\uinf)=\mathcal{H}_g(u|\uinf)$.

Let us now study the dissipation $J_g$ of $\mathcal{H}_g(u|u_\infty)$ in the non-degenerate case for the equation \eqref{GFP} written as in \eqref{eq:3}. It is given by
\begin{align}
  -J_g:=\frac{d}{dt}\mathcal{H}_g(u|u_\infty)&= \int_{\R^d} \partial_a G (u,u_\infty)\partial_tu \;dx\nonumber \\
  &= \int_{\R^d} g(\xi)\nabla\cdot(u\nabla\xi)\;dx = -\int_{\R^d} ug'(\xi)|\nabla\xi|^2\;dx \leq 0.
\label{tderent}
\end{align}
Analogously, in the degenerate case we compute the dissipation $J_g$ of $\widetilde{\mathcal{H}}_g(u|u_\infty)$ for the equation \eqref{GFP} written as in \eqref{eq:3b} given by
\begin{align}
  -J_g:=\frac{d}{dt}\widetilde{\mathcal{H}}_g(u|u_\infty)\;dx&= \int_{\R^d} \partial_a \tilde G (u,u_\infty;x)\partial_tu \;dx\nonumber\\
  &= \int_{\R^d} g(\xi)\nabla\cdot(u\nabla\xi)\;dx = -\int_{\R^d} ug'(\xi)|\nabla\xi|^2 \;dx\leq 0.
\label{tderent2}
\end{align}
Therefore, the dissipation of the general relative entropy functionals $J_g$ has a similar structure for both the degenerate and the non-degenerate cases. From now on, we will only work with the general relative entropy, Definition \ref{genrelentropy2}, and we shall drop the tilde for $\tilde G$ and $\widetilde{\mathcal{H}}_g$ for notational simplicity.

The form of the entropy dissipation in \eqref{tderent} and \eqref{tderent2} motivates our assumption that $g$ should be strictly increasing, which hence implied that $-J_g\le0$. For linear diffusion equations, this monotonicity of $g$ is equivalent to the strict convexity of the entropy generator $\psi$, see \eqref{eq:15}-\eqref{eq:17} below. Moreover, we shall deduce below exponential decay of $\mathcal{H}_g$ under the (necessary but not sufficient) additional requirement $g''\ge0$ which is equivalent to $y:=f'\ge0$ (see Proposition \ref{prop4.10} with the notation $e^f:=g'$). 

\begin{remark}\label{xi-range}
 Solutions $u$ to the evolution equation \eqref{GFP} are naturally taking values on the interval $[0,\infty)$. Therefore and due to the range of $\phi$, the function $\xi= \phi(u)+V(x)-\bar C$ may have a restricted range within $\R\cup\{-\infty\}$. Accordingly, we define its range by
\begin{eqnarray*}%\label{txi-minmax}
 \xi_{\min} &:= & \phi(0+)+\inf_\Rd V-\bar C=\phi(0+)-\bar C\,,\\
  \xi_{\max} &:= & \phi(\infty)+\sup_\Rd V-\bar C=\infty\,.
\end{eqnarray*}
Notice that in the non-degenerate case $\xi_{\min}=-\infty$ while $\xi_{\min}>-\infty$ in the degenerate case. 
Here we use that general initial conditions $u_0$ may take values in $[0,\infty)$ and that $V(x)\to\infty$ as $|x|\to\infty$ (see Remark \ref{data1}). 
For the power law nonlinearities from Remark \ref{data3}(b) this yields 
$$
  \xi_{\min} =
  \begin{cases}
  -\infty\,, \quad & \mbox{ for }m\le1\,,\\
  -\frac{m}{m-1}-\bar C>-\infty\,, \quad &\mbox{ for }m>1\,.
  \end{cases}
$$
\end{remark}

Before finishing this section, let us connect this definition of general relative entropy functionals to the case of linear diffusion. In fact, general relative entropy functionals were introduced in \cite{AMTU} for the case of $P(u)=Du$ with $D\in \R^+$ of the form
\begin{equation}\label{eq:15}
 H_\psi (u|u_\infty) = \int_{\R^d} \psi\left(\frac{u}{u_\infty}\right)u_\infty\;dx\,,
\end{equation}
with convex functions $\psi:\R_0^+\longrightarrow \R_0^+$. Their dissipation is given by
\begin{equation}\label{eq:16}
 J_\psi = D\int_{\R^d} \psi''\left(\frac{u}{u_\infty}\right)
  \left|\nabla\left(\frac{u}{u_\infty}\right)\right|^2 u_\infty\;dx\,.
\end{equation}
Based on the forms in \eqref{eq:15} and \eqref{eq:16}, one can easily check that the relative entropy $\mathcal{H}_g(u|u_\infty)$ and its dissipation $J_g$ recover the same formulas (as in \eqref{tderent}) for $g$ written in terms of $\psi$ as
\begin{equation}\label{eq:17}
  g'(\xi) = e^{\frac{\xi}{D}}\psi''(e^{\frac{\xi}{D}}) \qquad \mbox{for all } \xi\in\R.
\end{equation}
We will see later on in \S\ref{sec:BakEm} that the case of linear diffusion with $D>0$ can be reduced to $D=1$ by scaling. Thus, when discussing entropies for linear equations we will often consider $D=1$.

Moreover, in this linear case we have $\H_g(u|u_\infty)=H_\psi(u|u_\infty)$, due to $\phi(u)=\log u$ and the following identity for their integrands:
\begin{eqnarray*}
  G(u|u_\infty)(x)&=&\int_{u_\infty(x)}^{u(x)} \int_{u_\infty(x)}^{a} g'\left(\log\frac{\tilde a}{u_\infty(x)}\right) \frac1{\tilde a} \,d\tilde a\,da \\ 
  &=& \int_{u_\infty(x)}^{u(x)}\int_{u_\infty(x)}^{a} \psi''\left(\frac{\tilde a}{u_\infty(x)}\right)  \frac1{u_\infty(x)} \,d\tilde a\,da =
  \psi\left(\frac{u(x)}{u_\infty(x)}\right) u_\infty(x)\,,
\end{eqnarray*}
using $g(0)=0$ and $\psi(1)=\psi'(1)=0$.

\begin{example}
The classical $p$-entropy, $\psi_p(\sigma)=\frac{\sigma^p-1-p(\sigma-1)}{p-1}$, $1<p\leq 2$, for linear equations can be represented (for $D=1$) by
\begin{align}\label{pentropy}
  g_p(\xi):=\frac{p\,\left(e^{(p-1)\xi}-1\right)}{p-1}
\end{align}
using Definition \ref{genrelentropy}.
Notice that $g_p(\xi)\to g_1(\xi)$ as $p\to 1$ for all $\xi\in\R$. As already pointed out above in Example \ref{sam1}, $g_1$ leads to the Boltzmann logarithmic entropy for linear equations.\qed
\end{example}

\begin{remark}\label{rem:ent}
The relative entropy functional from our Definition \ref{genrelentropy2} looks similar to the \emph{relative $\Psi$-entropy} used in \cite{BLMV14}. In our notation the latter reads
$$
  N(u|u_\infty):= \int_{\R^d} \left[ \int_{u_\infty(x)}^{u(x)}g\big(P(s)-P(u_\infty(x))\big)\, ds\right] dx\,.
$$
Note that this entropy evaluates $g$ at the difference of the nonlinearity $P$, while Definition \ref{genrelentropy2} evaluates $g$ at the difference of the function $\phi$. The latter form is important for our generalized Bakry-\'Emery procedure in \S\ref{sec:4.1} below, while the large time analysis of porous medium equations on bounded domains with positive Dirichlet boundary conditions in \cite{BLMV14} just uses a Poincar\'e inequality (but not a Bakry-\'Emery strategy).
\end{remark}

%%%%%%%%%%%%%%%%%%%%%%%%%%%%%%%%%%%%%%%%%%%%%%%%%%%%%%%%
\subsection{Properties of nonlinearity curves}\label{sec:5.1}
To analyze the temporal decay of relative entropies along solutions of %for 
a nonlinear diffusion equation \eqref{GFP} with given nonlinearity $P(u)$, it will be convenient to introduce the functions
\begin{align}
    \label{eq:P2ab}
  \alpha(u):=P(u)/u \in[0,\infty] \qquad \text{and} \qquad \beta(u):=P'(u) \in[0,\infty] \,.     
\end{align}
Correspondingly, we define the \emph{nonlinearity curve} $(\alpha(u), \beta(u))$, $u \ge 0$ in the quarter plane $(\R^+_0)^2$.

Let us start by obtaining basic properties of the nonlinearity curves emanating from the assumptions {\bf (HP1)-(HP2)} on $P(u)$. Given a nonlinearity curve $(\alpha(u),\beta(u))$ with $u\in [0,\infty]$, then
\begin{equation}\label{signdir}
\mbox{sign}(\beta(u)-\alpha(u))=\mbox{sign}(\alpha'(u))\quad \mbox{ for all } u\in (0,\infty)\,,
\end{equation}
since $\alpha'(u)u=\beta(u)-\alpha(u)$. This implies that the parameterizations of nonlinearity curves below the diagonal $\alpha=\beta$ lead to decreasing functions $\alpha(u)$, and to increasing functions $\alpha(u)$ for nonlinearity curves above the diagonal, see Figure \ref{example1} below. Notice that \eqref{signdir} implies that, as soon as the nonlinearity curve crosses or reaches the diagonal at a finite parameter value $u_o\in (0,\infty)$, then it intersects the diagonal with a vertical tangent line.

Let us also note that by L'H\^{o}pital's rule
\begin{equation}\label{limatzero}
\lim_{u\to 0} \frac{P(u)}{u}=P'(0+),
\end{equation}
and hence $\alpha(0)=\beta(0)\in [0,\infty]$. If $\alpha(0)=\infty$, we call $P$ superlinear at the origin, and sublinear if $\alpha(0)=0$.

Since $P$ is strictly increasing, it has a limit at infinity $P(\infty)\in (0,\infty]$. If $P(\infty)=\infty$ then the same argument as above applies to give $\alpha(\infty)=\beta(\infty)\in [0,\infty]$. If $P(\infty)<\infty$ then $\alpha(\infty)=0$. Moreover, in this case
$$
\liminf_{u\to\infty} \beta(u)=0
$$
and hence either $\beta(\infty)=0$ or its limit does not exist. We first analyze the asymptotic behavior of the nonlinearity as $u\to\infty$.

\begin{prop}\label{prop:behaviorinfty}
Under the assumptions  {\bf (HP1)-(HP2)} on the nonlinearity $P(u)$, the limiting behavior of the nonlinearity curve $(\alpha(u),\beta(u))$ as $u\to\infty$ can be classified in the following four cases:
\begin{itemize}
\item[(a)] $\alpha(\infty)=\beta(\infty)\in (0,\infty)$: then the nonlinearity curve has linear behavior at infinity. More precisely, $P(u)\simeq \alpha(\infty) u$ for $u\to\infty$.

\item[(b)] $\alpha(\infty)=\beta(\infty)=0$: then the nonlinearity curve has sublinear or asymptotically linear behavior at infinity. More precisely, if $\frac{\beta(u)}{\alpha(u)}\to \lambda \in (0,1)$ as $u\to \infty$, then for any $\epsilon>0$ arbitrary small there exist $c_1,c_2>0$ such that $c_1 u^{\lambda-\epsilon} \leq P(u)\leq c_2 u^{\lambda+\epsilon}$ for $u\to\infty$.

\item[(c)]  $\alpha(\infty)=0$ and the limit of $\beta(u)$ as $u\to\infty$ does not exist: this is only possible for nonlinearities that saturate, i.e. $P(\infty)<\infty$.

\item[(d)]  $\alpha(\infty)=\beta(\infty)=\infty$: then the nonlinearity curve has superlinear  behavior at infinity.
\end{itemize}
\end{prop}

\begin{remark}\label{examplesnonlin}\
Let us give some examples illustrating the different cases above. Case (a) is illustrated by Example \ref{exnon1}. Case (b) can be subdivided as:
\begin{itemize}
\item[(b1)] $\frac{\beta(u)}{\alpha(u)}\to \lambda \in (0,1)$ as $u\to \infty$: We can have $P(u)\simeq c u^\lambda$ for $u\to\infty$. Also, Examples \ref{exnon2} and \ref{exnon3} correspond to this case. However, the generic behavior can be more general, as stated in Proposition \ref{prop:behaviorinfty} (b). Taking $P(u)\simeq u^\lambda \log (u)$ as $u\to \infty$ shows that the asymptotic behavior does not have to be a power function here.

\item[(b2)] $\frac{\beta(u)}{\alpha(u)}\to 0$ as $u\to \infty$: $P(u)\simeq \log (u)$ or $P(u)\simeq c-\frac1{u}$ as $u\to \infty$ are included here.

\item[(b3)] $\frac{\beta(u)}{\alpha(u)}\to 1$ as $u\to \infty$: $P(u)\simeq u^{1-\frac1{\sqrt{\log(u)}}}$ as $u\to \infty$ is included here.

\item[(b4)] It is also possible that the limit of $\frac{\beta(u)}{\alpha(u)}$ as $u\to \infty$ does not exist. 
\end{itemize}
For case (c) one can construct a nonlinearity $P(u)$ that saturates with $P(\infty)<\infty$, for which we can find two sequences $\{u_n\}_{n\in\N}\nearrow\infty$, $\{\tilde u_n\}_{n\in\N}\nearrow\infty$ such that $\{P'(u_n)\}_{n\in\N}\to 0$ and $P'(\tilde u_n)\geq \epsilon_o$ for all $n\in\N$ for some $\epsilon_o>0$ (see Figure \ref{remark62c}).
\end{remark}

\begin{figure}[htbp]
\begin{center}
\includegraphics[width=6cm]{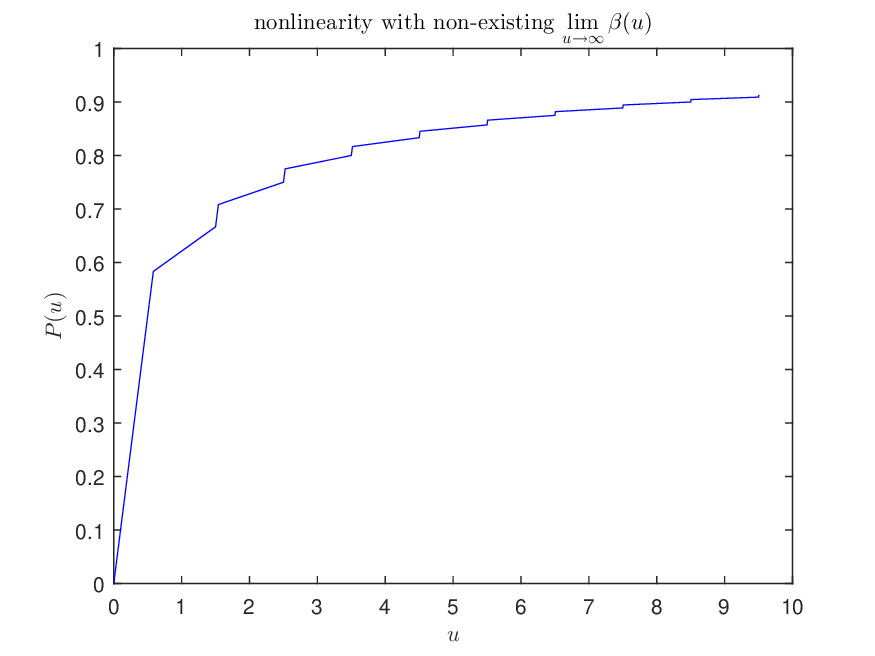}\hfill\includegraphics[width=6cm]{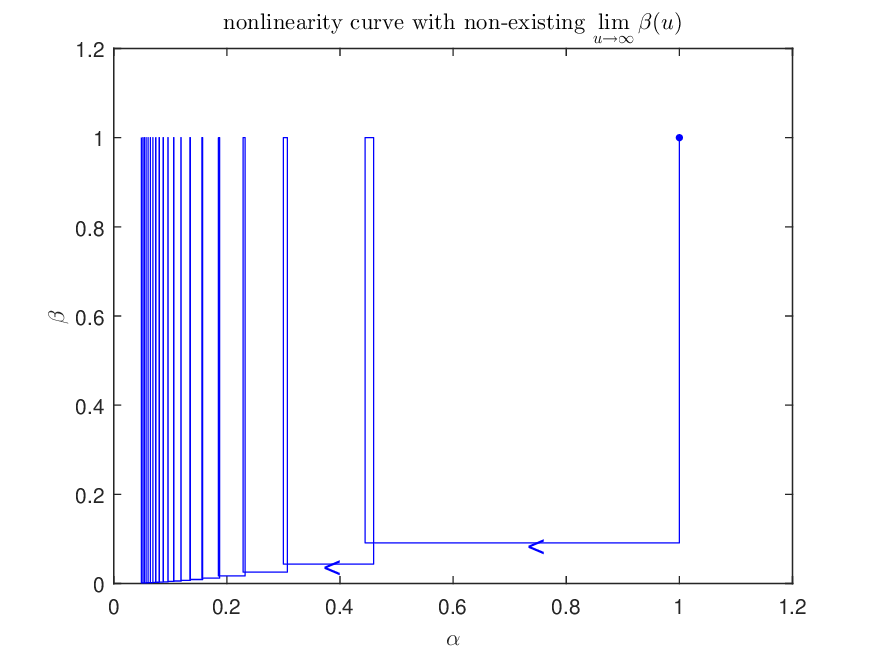}
\end{center}
%\vspace{-0.3cm}
\caption{\label{remark62c} {\footnotesize This nonlinearity $P(u)$ illustrates Proposition \ref{prop:behaviorinfty}(c) and Remark \ref{examplesnonlin}(c). We use $u_n=n$ with $P'(u_n)=\frac{1}{2(n+1)(n+2)-1}$ and $\tilde u_n=n-0.5$ with $P(\tilde u_n)=\frac{n}{n+1}$, $P'(\tilde u_n+)=1$. The right figure shows that $\liminf_{u\to\infty} \beta(u)=0$, $\limsup_{u\to\infty} \beta(u)=1$.\\
For simplicity we used a $P$ that is continuous but only piecewise linear. But a small regularization of it will be $C^1$ with $\alpha(u)$ and $\beta(u)$ behaving almost the same. The arrows indicate the orientation of $\alpha$ and $\beta$ w.r.t.\ $u$.}}
\end{figure}

We now classify our nonlinearities according to their behavior at the origin. We have the following three cases:
\begin{itemize}
\item[1.] {\bf Regular non-degenerate diffusions.} We say that $P(u)$ is a regular non-degenerate diffusion if $\alpha(0)=\beta(0)=P'(0+)\in(0,\infty)$. Notice that the nonlinearity curve emerges from the point $(\alpha(0),\alpha(0))$ on the diagonal. The nonlinearity curve has linear behavior at the origin. More precisely, $P(u)\simeq \alpha(0) u$ for $u\to 0+$.

\item[2.] {\bf Singular non-degenerate diffusions.} We say that $P(u)$ is a singular non-degenerate diffusion if $\alpha(0)=\beta(0)=P'(0+)=\infty$. Notice that the starting point of the nonlinearity curve is at infinity and the behavior at the origin of $P(u)$ is superlinear since $\alpha(0)=\infty$. 

\item[3.] {\bf Degenerate diffusions.} As introduced before, we say that $P(u)$ is a degenerate diffusion if $\alpha(0)=\beta(0)=P'(0+)=0$. Notice that the starting point of the nonlinearity curve is at the origin and the behavior at the origin of $P(u)$ is sublinear since $\alpha(0)=0$. 
\end{itemize}

\begin{remark}%\label{examplesnonlinsing}\
The linear diffusion $P(u)=Du$ is the simplest regular non-degenerate diffusion.
The archetypal example of a singular non-degenerate diffusion is the fast-diffusion equation with $P(u)=u^m$, $0<m<1$. Similarly, the archetypal example of a degenerate diffusion is the porous-medium equation with $P(u)=u^m$, $m>1$. In the last two cases the nonlinearity curve corresponds to the ray $\beta=m\alpha$, $\alpha\ge0$. 
\end{remark}

%%%%%%%%%%%%%%%%%%%%%%%%%%%%%%%%%%%%%%%%%%%%%%%%%%%
\section{Admissible relative entropies for linear diffusion equations}
\label{sct:linear}
\setcounter{equation}{0}

In this section we shall revisit the class of admissible entropies for linear Fokker-Planck equations that make the entropy method (or Bakry-\'Emery approach) possible. It will turn out that exactly this family provides also the relevant entropies for many non-degenerate nonlinear diffusions (see Remark \ref{remlin}-(d) below). Moreover, they are a subset of admissible entropies for some degenerate diffusion equations.

In \cite{BaEm84, AMTU} the entropy method for linear Fokker-Planck equations of the form 
\be\label{linFP}
  \frac{\partial u}{\partial t} = \nabla\cdot(u\nabla V(x)+D\nabla u)\,,
\ee
with some diffusion constant $D>0$, was developed. It applies to the following relative entropies.

\begin{definition}
\label{d:2.1}
Let $\psi\in C({\R^+_0})\cap C^4(\R^+)$
satisfy the conditions
\begin{equation*}
  \psi(1)=\psi'(1)=0,
\end{equation*}
\begin{equation*}
  \psi''\ge 0, \quad\psi''\not\equiv 0\quad\mbox{on } \R^+,
\end{equation*}
\begin{equation}
  \label{e:2.11c}
  (\psi''')^2\le\frac12 \psi''\psi^{IV}\quad\mbox{on } \R^+.
\end{equation}

Let $u_1,\,u_2 \in L^1_+(\Rd)$ with $\int u_1dx=\int u_2dx=1$
and
$u_1/u_2\in {\R^+_0} \;u_2(dx)-$ a.e.
Then
\begin{equation}
  \label{e:2.12}
   H_\psi(u_1|u_2):=\intRd\psi\left(\frac{u_1}{u_2}\right)\,u_2\;dx \ge0
\end{equation}
is called an \emph {admissible relative entropy} (of $u_1$ with respect
to $u_2$) with {\em generating function} $\psi$.
\end{definition}

The condition \eqref{e:2.11c} is equivalent to
\begin{equation}
 \label{e4:2.130}
 \left( \frac 1{\psi''}\right)^{''} \le 0
\end{equation}
whenever $\psi''>0$. Since \eqref{e4:2.130} excludes positive poles of $\frac 1{\psi''}$ we conclude $\psi''>0$ on $\R^+$. Thus admissible entropies are generated by strictly convex functions $\psi$.

The most typical examples of such admissible entropies are the \emph{$p$-entropies} \cite{AMTU, Bec89}, defined by
\begin{equation}\label{psi-p}
  \psi_p(\sigma):=\frac{\sigma^p-1-p(\sigma-1)}{p-1}\,,\qquad\mbox{for } 1<p\le2\,,
\end{equation}
and the logarithmic entropy as its $p\to1$---limit:
$$
  \psi_1(\sigma):=\sigma\log\sigma-\sigma+1\,.
$$
A simple computation shows that these $p$-entropies satisfy the following monotonicity:
$$
  \psi_{p_1}(\sigma)\le \psi_{p_2}(\sigma)\,,\qquad\mbox{for } p_1\le p_2\,, \:\:\sigma\ge0.
$$
For scaled variants of this family see \S2.2 of \cite{AMTU}.\\

The goal of this section is to rewrite the admissibility condition \eqref{e:2.11c} in a form that is more practical for dealing with nonlinear diffusion equations in the subsequent section. We start with the relation \eqref{eq:17} and recall that $g'>0$. Hence we substitute 
\be\label{entrof}
  g'(\xi)=e^{\xi/D}\psi''(e^{\xi/D})=:e^{f(\xi)}\,,\qquad\mbox{for } \xi\in(\xi_{\min},\infty)\,.
\ee
Then, for $D=1$ the condition \eqref{e:2.11c} is equivalent to
\be\label{f-inequ}
  f''(\xi) + f'(\xi) - f'(\xi)^2\ge0\,,\qquad\mbox{for } \xi\in (\xi_{\min},\infty)\,,
\ee
with $\xi_{\min}=-\infty$ for linear equations. For general $D>0$ the condition \eqref{f-inequ} is replaced by 
\be\label{f-inequ2}
 D f''(\xi) + f'(\xi) - D f'(\xi)^2\ge0\,,\qquad\mbox{for } \xi\in (\xi_{\min},\infty)\,.
\ee

For the linear Fokker-Planck equations \eqref{linFP} written as $\partial_t u = \nabla\cdot(u\nabla\xi)$, we have $\xi=D\log u+V(x)$ which can take values in all of $\R$. But for the nonlinear diffusion equations of \S\ref{sec:gen-entropies}, $\xi$ may vary only in semi-infinite intervals. Hence, the following lemma considered for $D=1$ will take into account both cases.

\begin{lemma}\label{f-solutions}\
\begin{enumerate}
\item[(a)] Global $C^2$-solutions to the differential inequality \eqref{f-inequ}, i.e.\ for $\xi\in\R$, satisfy $0\le f'\le1$. 
If $f'(\xi_0)= 1$ and the condition \eqref{f-inequ} holds for all $\xi \in \R$, then $f''(\xi)= 0$ for all $\xi\in (-\infty,\xi_0]$.
If $f'(\xi_0)= 0$ and the condition \eqref{f-inequ} holds, then $f''(\xi)= 0$ for all $\xi\in [\xi_0,\infty)$.

\item[(b)] $C^2$-solutions to \eqref{f-inequ} that exist only on the semi-infinite interval $\xi\in(\xi_1,\infty)$ with $\xi_1>-\infty$ (and are non-extendable) satisfy $f'\le0$. 
If $f'(\xi_0)= 0$ and the condition \eqref{f-inequ} holds for all $\xi \in (\xi_1,\infty)$ with $\xi_1>-\infty$, then $f''(\xi)= 0$ for all $\xi\in [\xi_0,\infty)$.
\end{enumerate}
As a consequence, the admissible entropies for linear diffusion equations satisfy $0\leq f'\leq 1$. 
\end{lemma}

\begin{figure}[htbp]
\begin{center}
\includegraphics[width=10cm]{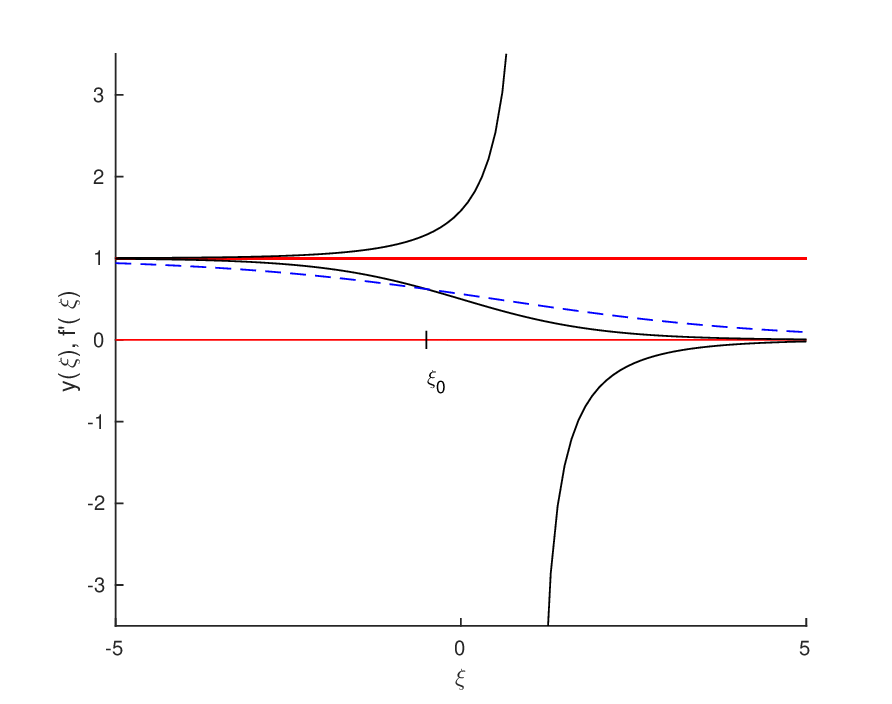}
\end{center}
%\vspace{-0.3cm}
\caption{\label{y_stability} {\footnotesize This plot visualizes the arguments in the proof of Lemma \ref{f-solutions}. Plotted are sample solutions $y(\xi)$ to the ODE \eqref{y-ODE} (solid black curves) and a solution $f'(\xi)$ to the differential inequality \eqref{f-inequ} (dashed blue curve). The two red horizontal lines represent the two critical values of $f'$: 0, 1. [colors only online]}}
\end{figure}

\begin{proof}
(a) We shall use here a simple comparison principle, based on the case of equality in \eqref{f-inequ}: With the substitution $y(\xi)=f'(\xi)$ we consider the ODE
\be\label{y-ODE}
  y'=y(y-1)\,,\qquad\mbox{for } \xi\in\R\,,
\ee
a Bernoulli equation with the general solution $y(\xi)=\frac{1}{1+ce^\xi}$ with $c\ge 0$, as well as $y\equiv0$.
A simple stability analysis shows that all of its solutions with an initial condition $y_0>1$ or $y_0<0$ diverge at a finite value of $\xi$ and are hence not global. Solutions with an initial condition $0\le y_0\le 1$ are global and monotonically decreasing.

Now we compare $y(\xi)$ to a solution $f'(\xi)$ of the differential inequality \eqref{f-inequ}, both having the same initial condition $f'(\xi_0)=y_0$. They satisfy
\begin{equation}\label{f-ineq}
  f'(\xi)\ge y(\xi)\,,\quad\mbox{for } \xi>\xi_0\,;\qquad 
  f'(\xi)\le y(\xi)\,,\quad\mbox{for } \xi<\xi_0\,. 
\end{equation}
Hence, no solution to \eqref{f-inequ} with an initial condition $y_0>1$ or $y_0<0$ can be global. 

\noindent
(b) Due to (a), solutions $f'(\xi)$ with an initial condition $0\le y_0\le1$ stay within this bound and can be extended to all of $\R$. Due to the first inequality in \eqref{f-ineq}, solutions with an initial condition $y_0>1$ cannot exist up to $\xi=\infty$. Hence, $f'$ must 
here also be negative on (a part of) $(\xi_{\min},\infty)$ but, as a consequence of \eqref{f-ineq}, it cannot be extended to all of $\R$. For an example cf.\ to the solid curve in Fig.\ \ref{y_stability}.
\end{proof}

We just proved that the admissible entropies for linear diffusion equations satisfy $0\leq f'\leq 1$, since $\xi_{\min}=-\infty$. However, this will not be the case for degenerate diffusion equations since $\xi_{\min}>-\infty$. We will elaborate on this in Remark \ref{remlin} and Section \ref{sec-quasilin}.

\begin{remark}
\begin{enumerate}
\item[(a)] The two critical values of $f'$ have the following interpretation for the entropy generators $\psi$: $f'\equiv0$ implies $\psi''(\sigma)=\frac{c}{\sigma}$ for some constant $c>0$, and the corresponding entropy is logarithmic. 
$f'\equiv1$ implies $\psi''=const$, and the corresponding entropy is quadratic. 
\item[(b)] $f'$ may satisfy $f'\equiv0$ on some interval $(-\infty,\xi_*]$ and $0<f'\le1$ on $(\xi_*,\infty)$. Then, the corresponding entropy generator $\psi(\sigma)$ is logarithmic for $\sigma\le e^{\xi_*}$. 
\item[(c)] $f'$ may satisfy $0\le f'<1$ on some interval $(-\infty,\xi_*)$ and $f'\equiv1$ on $[\xi_*,\infty)$. Then, the corresponding entropy generator $\psi(\sigma)$ is quadratic for $\sigma\ge e^{\xi_*}$. 
\item[(d)] As a combination of (b) and (c), a logarithmic entropy generator (for $\sigma$ small) may be connected to a quadratic behavior (for $\sigma$ large).
\end{enumerate}
\end{remark}

We first recover the classical examples of $p$-entropies. For simplicity we shall use here $D=1$.

\begin{example}\label{p-entropy}
For the $p$-entropies we start from \eqref{pentropy}:
$$
  g_p'(\xi) = pe^{(p-1)\xi}\,,\qquad 1\le p\le2\,,
$$
which yields $f_p(\xi)=(p-1)\xi+\log p$, $f_p'(\xi)=p-1\in[0,1]$.
Hence, $f_p''\equiv0$ if and only if $H_\psi$ is a $p$-entropy.

We remark that $f_p$ are particularly simple solutions of \eqref{f-inequ}, but they are not generic (e.g.\ maximal) solutions to it.\qed
\end{example}

Next we present a large family of entropies, all for $\xi\in\R$. 

\begin{example}
We consider the functions $f'(\xi)=\frac1{1+c_1e^{a\xi}}$ with $c_1\ge0$, $a\le1$ and $a\ne0$ that all satisfy \eqref{f-inequ}. Hence 
$$
  g'(\xi)=e^\xi\psi''(e^\xi)=e^{f(\xi)}=c_2e^\xi\frac1{(1+c_1e^{a\xi})^{1/a}}\,;
  \quad\mbox{ and } \psi''(\sigma)=c_2\big(1+c_1\sigma^{a}\big)^{-1/a}\,,
$$
for any $c_2>0$. 

For $a=1$ this represents the borderline cases for the inequalities \eqref{f-inequ} and \eqref{e4:2.130}, in the sense of satisfying the corresponding equalities, i.e., $y=f'$ then is a global solution of \eqref{y-ODE}. 
For $c_1>0$ this yields a family of logarithmic entropies with the generator
$$
  \psi(\sigma)=c_3\left[(\sigma+c_4) \log\frac{\sigma+c_4}{1+c_4}-\sigma+1\right]\,,
$$
for $c_3>0$, $c_4\ge0$, which was already presented in \cite{AMTU}.

For $a=\frac12$ (with $c_1=c_2=1$, e.g.) it yields
$$
  \psi(\sigma)=6\sqrt\sigma + 2(\sigma-3)\log(1+\sqrt \sigma)-2(1+\log2)\sigma+6\log2-4\,,
$$
and for $a=-1$ it yields
$$
  \psi(\sigma)=c_2\left(\frac{\sigma^2}{2} + c_1(\sigma\log\sigma-\sigma+1)-\sigma+\frac12\right)\,.
$$
For general $a\le1$, the entropy generator $\psi(\sigma)$ is a hypergeometric function.

For all $a<0$, the functions $f$ are convex on $\R$, which will be relevant in Theorem \ref{fconvex} below.\qed
\end{example}

In preparation for a later discussion we next show another example where $f$ is also a convex function on $\R$:

\begin{example}
We consider the functions $f'(\xi)=\frac{A}{1+c_1e^{-\xi}}$ with $c_1\ge0$ and $A>0$. Hence 
\[
  g'(\xi)=e^\xi\psi''(e^\xi)=e^{f(\xi)}=c_2(e^\xi+c_1)^A\,;
  \quad\mbox{ and } \psi''(\sigma)=c_2\frac{(\sigma+c_1)^A}{\sigma}\,,
\]
for any $c_2>0$.
For $A=2$, e.g., this yields the entropy generators
\[
  \psi(\sigma)=c_2\left(\frac{\sigma^3-3\sigma+2}{6}+c_1(\sigma-1)^2+c_1^2(\sigma\log\sigma-\sigma+1)
  \right)\,.
\]
\qed
\end{example}

Finally, we give an example of a nontrivial non-convex $f$ satisfying \eqref{f-inequ} on $\R$.

\begin{example}
We consider the functions
\[
  g'(\xi)=\exp\left(\frac{1}{\sqrt2}\arctan\big(\frac{\xi}{\sqrt2}\big)\right)\,;\quad\mbox{ and } \psi''(\sigma)=\frac{\exp\left(\frac{1}{\sqrt2}\arctan\big(\frac{\log\sigma}{\sqrt2}\big)\right)}{\sigma}\,.
\]
\qed
\end{example}

In the next sections we shall extend the concept of \emph{admissible entropies} to nonlinear diffusion equations, analyzing the interplay of nonlinearities and (corresponding) admissible relative entropies such that the entropy method is applicable.

%%%%%%%%%%%%%%%%%%%%%%%%%%%%%%%%%%%%%

\section{Dissipation estimates for nonlinear diffusions and general entropies}\label{sec:BakEm}
\setcounter{equation}{0}
Now, we proceed with the computation of the derivative of the dissipation of the general relative entropy defined in \eqref{tderent} or \eqref{tderent2} for general nonlinearities $P(u)$.
Throughout this section, we assume that the potential $V$ and the nonlinearity $P(u)$ satisfy {\bf (HV1)}-{\bf (HV2)} and {\bf (HP1)}-{\bf (HP2)}, respectively, together with {\bf (HPV)}. 
Our goal is to formulate sufficient (and close to optimal) conditions on $g$ under which the dissipation inequality 
\begin{align}
  \label{eq:alphaomega}
  -\frac{d}{dt}J_g(u(t)) - 2\lambda J_g(u(t)) \ge 0
\end{align}
holds for all sufficiently regular solutions to \eqref{GFP}.
The constant $\lambda>0$ in \eqref{eq:alphaomega} is the one from our general hypothesis {\bf (HV2)};
we are neither interested in improving that constant (as could be done, e.g., with perturbation arguments \`a la Holley-Stroock \cite{HolSto87, AMTU}),
nor in the validity of \eqref{eq:alphaomega} for sub-optimal constants $\lambda'<\lambda$.

Validity of \eqref{eq:alphaomega} for a given pair $(P(u),g(\xi))$ leads to a variety of consequences via the so-called Bakry-\'Emery procedure, 
see \cite{BaEm84,BaEm85a,BaEm85,AMTU} for the linear case and \cite{CJMTU} for the nonlinear case. 
Some of these consequences are:
\begin{enumerate}
\item[(a)] Along each solution to \eqref{GFP}, 
  the dissipation $J_g$ goes to zero exponentially fast at rate $2\lambda$,
  \begin{align}\label{mainconclusion1}
    J_g(u(t)) \le e^{-2\lambda(t-s)}J_g(u(s)) \quad \text{for all $t\ge s\ge 0$},
  \end{align}
  where we assume that $J_g(u(s))<\infty$.
\item[(b)] For each solution to \eqref{GFP}, along which ${\mathcal H}_g(u(t)|\uinf)$ goes to zero as $t\to\infty$,
  this convergence is exponentially fast with rate $2\lambda$,
  \begin{align}\label{mainconclusion2}
    {\mathcal H}_g(u(t)|\uinf) \le e^{-2\lambda(t-s)}{\mathcal H}_g(u(s)|\uinf) \quad \text{for all $t\ge s\ge 0$},
  \end{align}
  where we assume that ${\mathcal H}_g(u(s)|\uinf)<\infty$.
\item[(c)] For the initial condition $\bar u=u_0$ of each solution to \eqref{GFP}, along which ${\mathcal H}_g(u(t)|\uinf)\to0$,
  the following functional inequality holds:
  \begin{align}\label{mainconclusion3}
    {\mathcal H}_g(\bar u|\uinf) \le \frac1{2\lambda}J_g(\bar u).
  \end{align}
\end{enumerate}

In the following computations we assume that all terms are smooth and that integration by parts is allowed. This can be made rigorous by approximations from bounded domain cases with no-flux boundary conditions as in \cite{otto,CJMTU} in all nonlinearity types. For the degenerate diffusion cases, this needs a further approximation of the degenerate diffusion nonlinearity $P(u)$ by a sequence of non-degenerate diffusion nonlinearities behaving linearly for small values of $u$. This procedure has been properly done in \cite[Section 5]{otto} for the case of power-law nonlinearities and quadratic confinement potentials, and generalized for nonlinearities $P(u)$ and confinement potentials under the assumptions {\bf (HP1)}-{\bf (HP2)} and {\bf (HV1)}-{\bf (HV2)} in \cite[Section 3]{CJMTU}. In this same spirit, let us remark that the subsequent computations involve up to third derivatives of $V$ (see \eqref{V-3rd}) covered by the assumption {\bf (HV1)}. Note, however, that this higher regularity is only required for the intermediate steps, but not for the final remainder term in \eqref{jjj}, so it could be reduced to two derivatives of $V$ by approximation, although we do not pursue this here. Moreover, in the above mentioned approximations of the diffusion equation, such intermediate steps would be carried out for approximating nonlinearities $P$ by non-degenerate diffusions. Let us finally mention that a direct proof of the inequality \eqref{mainconclusion3} for power-law cases was obtained in \cite{DD02} in terms of Gagliardo-Nirenberg inequalities with sharp constants.

\subsection{Generalized Bakry-\'Emery procedure}\label{sec:4.1}
We differentiate $J_g$ in $t$ and integrate by parts in two of the three integrals to get for $t>0$
\begin{align}
  -\frac12\frac{d}{dt} J_g 
  =&\, -\frac12\int_{\R^d} \partial_tu g'(\xi)|\nabla\xi|^2 \;dx
  - \frac12\int_{\R^d} u''(u)\partial_t u g''(\xi)|\nabla\xi|^2 \;dx\nonumber\\
  &- \int_{\R^d} ug'(\xi)\nabla\xi\cdot\nabla\big(\phi'(u)\partial_tu\big) \;dx\nonumber\\
  = &\, \frac12\int_{\R^d} u\nabla\xi\cdot\nabla\big[g'(\xi)|\nabla\xi|^2\big] \;dx - \frac12\int_{\R^d} u\phi'(u)g''(\xi)|\nabla\xi|^2\big[\nabla u\cdot\nabla\xi + u\Delta\xi\big] \;dx\nonumber\\
  &+ \int_{\R^d} \phi'(u)\nabla\cdot\big[ug'(\xi)\nabla\xi\big]\big[\nabla u\cdot\nabla\xi + u\Delta\xi\big]\;dx\nonumber\\
  =&\, \int_{\R^d} ug'(\xi)\nabla\xi\cdot\nabla^2\xi \cdot \nabla\xi \;dx
  + \frac12 \int_{\R^d} ug''(\xi)|\nabla\xi|^4 \;dx\nonumber\\
  & -\frac12\int_{\R^d} g''(\xi)\nabla P(u)\cdot\nabla\xi|\nabla\xi|^2 \;dx- \frac12\int_{\R^d} u^2\phi'(u)g''(\xi)|\nabla\xi|^2\Delta\xi \;dx\nonumber\\
  & + \int_{\R^d} u\phi'(u)g''(\xi)|\nabla\xi|^2\big[\nabla u\cdot\nabla\xi + u\Delta\xi\big]\;dx\nonumber\\
& + \int_{\R^d} \phi'(u)g'(\xi)\nabla\cdot\big[u\nabla\xi\big]\big[\nabla u\cdot\nabla\xi + u\Delta\xi\big]\;dx\,.
\label{fin1}  
\end{align}
Here and in the sequel we use the notation $a\cdot b:=a^Tb$ and $a\cdot C \cdot b:=a^T C b$ for vectors $a,\,b$, and square matrices $C$.
Let us identify the ``good term'' if convexity of the potential {\bf (HV2)} is assumed.
For the first term in the expression \eqref{fin1}, we obtain
\begin{equation}
  \label{eq:nonopt}
  \begin{split}
    \int_{\R^d} ug'(\xi)\nabla\xi\cdot\nabla^2\xi\cdot \nabla\xi \;dx
    &= \int_{\R^d} ug'(\xi)\nabla\xi\cdot\nabla^2 V \cdot \nabla\xi \;dx\\
    &\quad+ \int_{\R^d} ug'(\xi)\nabla\xi\cdot\nabla^2\phi(u)\cdot\nabla\xi \;dx\\
    &\ge \lambda J_g + \int_{\R^d} ug'(\xi)\nabla\xi\cdot\nabla\big[\nabla \phi(u)\big]\cdot\nabla\xi \;dx\,.
  \end{split}
\end{equation}
Plugging it back in \eqref{fin1}, we deduce for $t>0$
\begin{align}
  -\frac12\frac{d}{dt} J_g - \lambda J_g
  \geq&\,  \int_{\R^d} ug'(\xi)\nabla\xi\cdot\nabla\big[\nabla \phi(u)\big]\cdot\nabla\xi\;dx
  + \frac12 \int_{\R^d} ug''(\xi)|\nabla\xi|^4 \;dx\nonumber\\&-\frac12\int_{\R^d} g''(\xi)\nabla P(u)\cdot\nabla\xi|\nabla\xi|^2\;dx \nonumber\\
  & - \frac12\int_{\R^d} u^2\phi'(u)g''(\xi)|\nabla\xi|^2\Delta\xi\;dx + \int_{\R^d} u\phi'(u)g''(\xi)|\nabla\xi|^2\big[\nabla u\cdot\nabla\xi + u\Delta\xi\big] \;dx\nonumber\\
  & + \int_{\R^d} \phi'(u)g'(\xi)\nabla\cdot\big[u\nabla\xi\big]\big[\nabla u\cdot\nabla\xi + u\Delta\xi\big]\;dx\nonumber\\
  := &\,I_1+I_2+I_3+I_4+I_5+I_6.
 \label{fin2}
\end{align}
Integration by parts in the first term $I_1$ of \eqref{fin2} gives
\begin{align}\label{V-3rd}
  I_1&= - \int_{\R^d} g'(\xi) \nabla\cdot(u\nabla\xi) \big[\nabla \phi(u)\cdot\nabla\xi \big] \;dx
  - \int_{\R^d} ug''(\xi) |\nabla\xi|^2 \big[\nabla \phi(u)\cdot\nabla\xi \big] \;dx\ \nonumber \\
  &\quad - \int_{\R^d} ug'(\xi)\nabla\xi\cdot\nabla^2\xi\cdot\nabla \phi(u) \;dx\nonumber \\
  &= -\int_{\R^d} \phi'(u)g'(\xi) \nabla\cdot(u\nabla\xi)(\nabla u\cdot\nabla\xi)\;dx
  - \int_{\R^d} u\phi'(u)g''(\xi) |\nabla\xi|^2 (\nabla u\cdot\nabla\xi)\;dx\nonumber  \\
  &\quad - \int_{\R^d} g'(\xi) \nabla P(u)\cdot\nabla^2\xi\cdot\nabla\xi \;dx\nonumber \\
  &= -\int_{\R^d} \phi'(u)g'(\xi) \nabla\cdot(u\nabla\xi)(\nabla u\cdot\nabla\xi)\;dx
  - \int_{\R^d} u\phi'(u)g''(\xi) |\nabla\xi|^2 (\nabla u\cdot\nabla\xi)\;dx \nonumber \\
  &\quad 
  + \int_{\R^d} P(u)g''(\xi) \nabla\xi\cdot\nabla^2\xi\cdot\nabla\xi \;dx
  + \int_{\R^d} P(u)g'(\xi)\|\nabla^2\xi\|^2\;dx \nonumber \\
  &\quad+ \int_{\R^d} P(u)g'(\xi)\nabla\xi\cdot\nabla\Delta\xi\;dx\,.
\end{align}
Here and in the sequel, we will denote by $\|\nabla^2\xi\|$ the Frobenius norm of $\nabla^2\xi$.
For the third term in expression \eqref{fin2}, we also integrate by parts to obtain
\begin{align*}
  I_3&=\frac12\int_{\R^d} P(u)g'''(\xi)|\nabla\xi|^4 \;dx+ \frac12\int_{\R^d} P(u)g''(\xi)|\nabla\xi|^2\Delta\xi \;dx\nonumber\\
  &\quad+ \int_{\R^d} P(u)g''(\xi)\nabla\xi\cdot\nabla^2\xi\cdot\nabla\xi\;dx\,.
\end{align*}
Substituting $I_1$ and $I_3$ in \eqref{fin2} yields 
\begin{align*}
  -\frac12\frac{d}{dt}J_g -\lambda J_g
  &\ge \int_{\R^d} P(u)g'(\xi)\|\nabla^2\xi\|^2 \;dx+ 2 \int_{\R^d} P(u)g''(\xi)\nabla\xi\cdot\nabla^2\xi\cdot\nabla\xi \;dx\\
  &\quad+ \frac12\int_{\R^d}\big[P(u)g'''(\xi)+ug''(\xi)\big]|\nabla\xi|^4\;dx\\
  &\quad + \frac12\int_{\R^d} \big[P(u)-u^2\phi'(u)\big]g''(\xi)|\nabla\xi|^2\Delta\xi \;dx+ \int_{\R^d} P(u)g'(\xi)\nabla\xi\cdot\nabla\Delta\xi \;dx\\
  &\quad + \int_{\R^d} u^2\phi'(u)g''(\xi)|\nabla\xi|^2\Delta\xi\;dx
  + \int_{\R^d} u\phi'(u)g'(\xi)\nabla\cdot\big[u\nabla\xi\big]\Delta\xi\;dx
\end{align*}
for $t>0$, leading to
\begin{align*}
  -\frac12\frac{d}{dt}J_g -\lambda J_g
  &\ge \int_{\R^d} P(u)g'(\xi)\|\nabla^2\xi\|^2 \;dx+ 2 \int_{\R^d} P(u)g''(\xi)\nabla\xi\cdot\nabla^2\xi\cdot\nabla\xi \;dx\\
  &\quad + \frac12\int_{\R^d}\big[P(u)g'''(\xi)+ug''(\xi)\big]|\nabla\xi|^4 \;dx+\! \int_{\R^d} uP'(u)g'(\xi)\big(\Delta\xi\big)^2dx\\
  &\quad + \frac12 \int_{\R^d} \big[uP'(u)-P(u)\big]g''(\xi)|\nabla\xi|^2\Delta\xi \;dx\\
  &\quad + \int_{\R^d} P(u)g'(\xi)\nabla\xi\cdot\nabla\Delta\xi \;dx+ \int_{\R^d} P(u) g''(\xi)|\nabla\xi|^2\Delta\xi\;dx\\
  &\quad 
  + \int_{\R^d} P'(u)g'(\xi)\big(\nabla u\cdot \nabla\xi\big)\Delta\xi \;dx\,,
\end{align*}
for $t>0$, by rearranging several terms where we used $u\phi'(u)=P'(u)$.
For final simplification, we use that
\begin{align*}
  \nabla\cdot\big[P(u)g'(\xi)\nabla\xi\Delta\xi\big] &= P(u)g'(\xi)\nabla\xi\cdot\nabla\Delta\xi
  + P(u)g''(\xi)|\nabla\xi|^2\Delta\xi\\
  & \quad
  + P'(u)g'(\xi)\big(\nabla u\cdot\nabla\xi\big)\Delta\xi + P(u)g'(\xi)(\Delta\xi)^2\,,
\end{align*}
to substitute the last three terms from above. This leaves us with the final expression
\begin{align}
  -\frac12\frac{d}{dt}J_g - \lambda J_g
  &\ge 
\int_{\R^d} P(u)g'(\xi)\|\nabla^2\xi\|^2 \;dx+ 2 \int_{\R^d} P(u)g''(\xi)\nabla\xi\cdot\nabla^2\xi\cdot\nabla\xi \;dx\nonumber\\
  &\quad + \frac12\int_{\R^d} \big[P(u)g'''(\xi)+ug''(\xi)\big]|\nabla\xi|^4 \;dx\nonumber\\
  &\quad + \frac12\int_{\R^d} \big[uP'(u)-P(u)\big]g''(\xi)|\nabla\xi|^2\Delta\xi \;dx\nonumber\\
  &\quad + \int_{\R^d} \big[uP'(u)-P(u)\big]g'(\xi)(\Delta\xi)^2 \;dx=: \mathcal{R}\,,
  \label{jjj}
\end{align}
for $t>0$. 

We now use the substitution \eqref{entrof} of the generators of general entropies. 
While \eqref{entrof} was used in the linear case only for $D=1$, we will use the same transformation 
\begin{equation}\label{entrofgen}
    g'(\xi)=e^{f(\xi)} \qquad \mbox{for all } \xi\in(\xi_{\min},\infty)\,,
\end{equation} 
as in \eqref{entrof} for the general case since $g'>0$. This substitution yields
\begin{equation}\label{finform1}
  -\frac12\frac{d}{dt}J_g - \lambda J_g\ge \mathcal{R} =\int_{\R^d} u e^{f(\xi)} R(\alpha,\beta,\xi) \;dx
\end{equation}
with
\begin{align}
R(\alpha,\beta,\xi) &=  \alpha \|\nabla^2\xi\|^2 
                        + 2 \alpha f'(\xi)\nabla\xi\cdot\nabla^2\xi\cdot\nabla\xi 
                        + \frac12 \big[\alpha (f'(\xi)^2+f''(\xi))+f'(\xi)\big]|\nabla\xi|^4 \nonumber\\
                        &\quad+ \frac12 (\beta-\alpha)f'(\xi)|\nabla\xi|^2\Delta\xi 
                        + (\beta-\alpha)(\Delta\xi)^2\,, 
\label{finform2}
\end{align}
and we recall the definitions 
\begin{equation}\label{finform3}
  \alpha(u):=P(u)/u \qquad \mbox{and} \qquad \beta(u):=P'(u) \,. 
\end{equation}
In formula \eqref{finform1} we use, for factoring out $u$, that $u>0$ or $P'(0)$ exists in $[0,+\infty)$. If $P'(0)$ is finite, note that $\alpha$ is well defined when $u\to 0$ by L'H\^opital's rule since $P(0)=0$ with $\alpha(0)=\beta(0)$. If $P'(0)=+\infty$, one expects solutions to \eqref{GFP} to be instantaneously positive as in the classical fast-diffusion equations, see \cite{VazquezPME} and the references therein. 

Note that $R$ in \eqref{finform1} could be modified by adding some ``null forms'', but we do not know if this may help the analysis.

\begin{remark}[Implications on the non-degenerate diffusion case]\label{linear}
\
\begin{enumerate}
\item[(a)]
Choosing $P(u)=u$ in \eqref{finform1}-\eqref{finform2}-\eqref{finform3}, we obtain the formula in \cite{AMTU} where
\begin{align*}%\label{Rlin}
R(1,1,\xi) 
&= \|\nabla^2\xi\|^2 
                        + 2 f'(\xi)\nabla\xi\cdot\nabla^2\xi\cdot\nabla\xi 
                        + \frac12 \big[ f'^2(\xi)+f''(\xi)+f'(\xi)\big]|\nabla\xi|^4 \,.
\end{align*}
It is easy to check that 
$R(1,1,\xi)= \tr(XY)$ with
$$
X:= \left(\begin{array}{cc} 1 & f'(\xi) \\ 
                        f'(\xi) & \frac12 \big[f'^2(\xi)+f''(\xi)+f'(\xi)\big]\end{array}\right),
                        \quad
Y:=\left(\begin{array}{cc} \|\nabla^2\xi\|^2 & \nabla\xi\cdot\nabla^2\xi\cdot\nabla\xi \\ \nabla\xi\cdot\nabla^2\xi\cdot\nabla\xi & |\nabla\xi|^4 \end{array}\right).
$$
The Cauchy-Schwarz inequality gives that $Y$ is positive semidefinite. The matrix $X$ is positive semidefinite if and only if $\det(X)=\tfrac12 (f''(\xi)+f'(\xi)-f'^2(\xi))\geq 0$, which is the equivalent condition \eqref{f-inequ} to the admissibility condition \eqref{e:2.11c} for entropies in the linear case. Hence, \eqref{f-inequ} implies the non-negativity of the remainder, i.e., $R(1,1,\xi)\ge0$.
Actually, we will show below, see Remark \ref{remlin}-(a), that \eqref{f-inequ} is even equivalent to the remainder condition \eqref{rem-condb} for the linear case.
\item[(b)] 
Next we apply the remainder formula \eqref{finform2} to the general linear diffusion $P(u)=Du$ with $D>0$ (and hence $\alpha=\beta=D$) and general entropies of the form \eqref{eq:17}. Motivated by the form of \eqref{eq:17} we set $f(\xi)=:\tilde f(\xi/D)$ and $\bar\xi:=\xi/D$. This yields
\begin{align*}
  R(D,D,\xi)
 =  D^3\Big\{ \|\nabla^2\bar\xi\|^2 
                        + 2 \tilde f'(\bar\xi)\nabla\bar\xi\cdot\nabla^2\bar\xi\cdot\nabla \bar\xi 
                        + \frac12 \big[\tilde f'^2(\bar\xi)+\tilde f''(\bar\xi)+\tilde f'(\bar\xi)\big]\,|\nabla\bar\xi|^4 \Big\}\,,
\end{align*}
and we recover again the condition \eqref{f-inequ} on the entropies, i.e., $\tilde f''(\bar \xi)+\tilde f'(\bar \xi)-\tilde f'^2(\bar \xi)\geq 0$ for each fixed $D>0$.
\end{enumerate}
\end{remark}

\begin{remark}[Implications on the nonlinear diffusion case]\label{StandardEntropy}
Choosing $g_1(\xi)=\xi$, i.e. $f_1(\xi)=0$, in \eqref{finform1}-\eqref{finform2}-\eqref{finform3} for general equations of the form \eqref{GFP}, we recover the formula in the proof of \cite[Theorem 11]{CJMTU} where
\begin{align}\label{cjmtu}
R(\alpha,\beta,\xi) &= \alpha \|\nabla^2\xi\|^2 
                        + (\beta-\alpha)(\Delta\xi)^2 \,.
\end{align}
It is easy to check using Cauchy-Schwarz for symmetric matrices (see \eqref{CS}) that $R(\alpha,\beta,\xi)\ge 0$ for all $u\geq 0$ and all functions $\xi$ if and only if $\tfrac{\alpha}{d}+\beta-\alpha\geq 0$ for all $u\geq 0$, or equivalently
\be\label{mccann}
\frac{d-1}{d} P(u) \leq uP'(u)\quad \mbox{for all } u\geq 0\,.
\ee
The latter is McCann's condition \cite{McCann} for displacement convexity of the standard entropy functional.
Specifically for $d=1$, the standard entropy always yields $R(\alpha,\beta,\xi)\ge 0$ and hence \eqref{eq:alphaomega}. Note that if $u=0$, then \eqref{mccann} is satisfied by dividing by $u$ and taking the limit as $u\to 0^+$ due to \eqref{limatzero}.
\end{remark}

\subsection{Conclusions from the generalized Bakry-\'Emery procedure}
We will now simplify in several steps the condition under which the generalized Bakry-\'Emery procedure is successful, meaning that the right hand side of \eqref{finform1} is non-negative.
\begin{prop}\label{expdecay}
  Suppose that the integral $\mathcal R$ on the right-hand side of \eqref{finform1} is non-negative for all sufficiently smooth functions $u$.
  Then \eqref{eq:alphaomega} holds for all solutions to \eqref{GFP}.
\end{prop}
The implication stated in Proposition \ref{expdecay} above is presumably even an equivalence.
Indeed, any non-equivalence between $\mathcal R\ge0$ and the validity of \eqref{eq:alphaomega} can only result from the single estimate that we performed in the Bakry-\'Emery calculations, namely the convexity estimate on $V$ in step \eqref{eq:nonopt}.
Thus, non-equivalence would mean that there is a choice of $P$, $V$, and $g$ such that 
\begin{align*}
  \mathcal R(u) + \int_{\R^d}ug'(\xi)\,\big(\nabla\xi\cdot\nabla^2V \cdot\nabla \xi-\lambda|\nabla\xi|^2\big)\,dx \ge 0
\end{align*}
\emph{for all} sufficiently smooth $u$, whereas $\mathcal R(u^*)<0$ \emph{for some} (non-smooth) $u^*$.
This seems highly unlikely, however, we are not aware of a proof of equivalence in Proposition \ref{expdecay}.

We continue by formulating sufficient conditions for the non-negativity of the integral $\mathcal R$ in terms of pointwise conditions on the integrand. That is, instead of asking that $\mathcal R(u)\ge 0$ for all relevant functions $u$, we ask that its integrand $R$ satisfies 
\begin{align}
  \label{eq:ptw}
\boxed{  R\big(\alpha,\beta,\xi\big)(x)\ge 0  \quad \text{for all $x\in\R^d$}} 
\end{align}
where $\alpha=P(u)/u$, $\beta=P'(u)$, and $\xi=\phi(u)+V(x)-\bar C$ for all $u>0$. The boxed inequality \eqref{eq:ptw} is our decisive criterion in the rest of this paper, to judge whether the entropy method works.

\begin{remark}
This pointwise condition \eqref{eq:ptw} is presumably stronger than the integral condition.
Indeed, optimality would mean that whenever there is a function $u^*$ such that $R(\alpha^*,\beta^*,\xi^*)$ is negative \emph{at some point $x$}, there should also exist an admissible function $\tilde u$ (possibly, but not necessarily, $\tilde u=u^*$) such that the integral $\mathcal R(\tilde u)$ is negative. Such ``trial functions'' $\tilde u$ have been constructed, for instance, 
to prove optimality of certain entropy estimates for the fourth order thin film equation on a one-dimensional periodic domain by Laugesen \cite{laugesen}, and that procedure has been generalized subsequently \cite{JM}. In the situation at hand, it is unclear to the authors when to expect equivalence of the integral and the pointwise condition, and how to construct suitable trial functions in that case --- even for linear equations with $P(u)=u$.
\end{remark}

Actually, we go one step further and replace in the expression for $R(\alpha,\beta,\xi)$ on the right-hand side of \eqref{finform2}
the derivatives $\nabla\xi$ and $\nabla^2\xi$ of $\xi$ by a vector $v$ and a symmetric matrix $S$, respectively,
that are no more related to $u$.
In a chain of steps (see Lemmata \ref{prop4.5}, \ref{lem-char}, and \ref{lem-char2}, below) this will yield a criterion that is equivalent to \eqref{eq:ptw} but easier for computations. 
\begin{lemma}\label{prop4.5}
  For a given nonlinearity $P(r)$ and an entropy generated by $f=\log g'$, suppose that
  \be\label{rem-cond}
  \bar R\big(\alpha(r),\beta(r),z,v,S\big)\ge0\quad \text{for all 
  $r>0$, $z\ge\phi(r)-\bar C$, $v\in\R^d$, $S\in\R^{d\times d}_\text{symm}$},
  \ee
  with (below, $\alpha$ and $\beta$ act as placeholders) 
  \begin{align}
    \bar R(\alpha,\beta,z,v,S) 
    &:=  \alpha \|S\|^2 
      + 2 \alpha f'(z)v\cdot S \cdot v 
      + \frac12 \big[\alpha (f'(z)^2+f''(z))+f'(z)\big]|v|^4 \nonumber\\
    &\quad+ \frac12 (\beta-\alpha)f'(z)|v|^2\tr S 
      + (\beta-\alpha)(\tr S)^2\,.
      \label{finform2b}
  \end{align}
  Then \eqref{eq:alphaomega} holds.
\end{lemma}
Notice that the hypothesis above, which uses general $z$, $v$ and $S$, respectively, in place of $\xi(x)$, $\nabla\xi(x)$, $\nabla^2\xi(x)$,
is still equivalent to the original pointwise condition \eqref{eq:ptw},
since for each admissible choice of $\bar r,\bar z,\bar v,\bar S$,
one can easily find a suitable point $\bar x$ and a smooth function $\bar u$ such that 
\begin{align*}
  &\phi(\bar r) + V(\bar x) -\bar C = \bar z, \quad
  \bar u(\bar x) = \bar r,\\
  &\nabla (\phi(\bar u)) (\bar x) = \bar v -\nabla V (\bar x),\quad
  \nabla^2 (\phi(\bar u)) (\bar x) = \bar S -  \nabla^2 V (\bar x).
\end{align*}
However, for keeping the largest possible set of admissible entropies, it is essential that $\bar z$ satisfies the condition $\bar z\in[\phi(\bar r)-\bar C,\infty)$, since by our general hypotheses {\bf (HV1)-(HV2)}, $V$ has minimal value zero and grows to infinity for $|x|\to\infty$.

While $z$ and $r$ are clearly coupled above, it is tempting to simplify condition \eqref{rem-cond} further by decoupling $z$ from $r$: 
\begin{lemma}
  Suppose that
  \be\label{rem-condb}
  \bar R\big(\alpha(r),\beta(r),z,v,S\big)\ge0\,\quad \text{for all $r>0$, $z>\xi_{\min}$, $v\in\R^d$, $S\in\R^{d\times d}_{\text{symm}}$},
  \ee
  with $\bar R$ given in \eqref{finform2b} above, and $\xi_{\min}$ defined in Remark {\rm\ref{xi-range}}.
  Then \eqref{eq:alphaomega} holds.
\end{lemma}
In the linear case $P(u)=Du$, condition \eqref{rem-condb} is actually equivalent to condition \eqref{rem-cond}: this is obvious from $\alpha\equiv\beta\equiv D$, which makes $\bar R(\alpha,\beta,z,v,S)$ independent of $r$. For nonlinear diffusions, on the other hand, condition \eqref{rem-condb} is typically substantially stronger than \eqref{rem-cond}. The difference between these two conditions is illustrated in Section \ref{sct:xmp-pme2} below in the case $P(r)=r^2$.

For that reason, we shall \emph{not} pursue the simplification \eqref{rem-condb} above, but preserve the relation between $r$ and $z$. Instead, we shall see that it is possible to simplify \eqref{rem-cond} by eliminating the ``dummy variables'' $v$ and $S$. For brevity, introduce
\begin{align}
\label{eq:def-nu}
    \mu:=\beta-\frac{d-1}d\alpha
    \quad\text{and}\quad
    \kappa := \bar\kappa_d + \frac{(\mu-3\alpha/d)^2}{8\alpha\mu}
    \quad\text{with}\quad \bar\kappa_d=1-\frac1{2d},
\end{align}
that are functions of $r>0$.
\begin{lemma}\label{lem-char}
    For a given pair $r>0$ and $z\in(\xi_{\min},\infty)$, 
    the condition 
    \begin{align}\label{rem-cond2a}
        \bar R\big(\alpha(r),\beta(r),z,v,S\big)\ge0\,\quad \text{for all $v\in\R^d$, $S\in\R^{d\times d}_{\text{symm}}$},
    \end{align}
    which is part of the hypothesis \eqref{rem-cond}, is equivalent to the following condition:
    \begin{align}
        \label{rem-condz}
        \mu(r)\ge0 \quad \text{and}\quad Z\big(\alpha(r),\mu(r),f'(z),f''(z)\big)\le 0
    \end{align}
    with the polynomial (below, $\alpha$, $\mu$, and $a$, $b$ act as placeholders)
    \begin{align}
        \label{eq:Z}
        Z(\alpha,\mu,a,b) := \big(\mu^2+2(4-5/d)\mu\alpha+9\alpha^2/d^2\big)a^2-8\mu(a+\alpha b)\,,
    \end{align}
    and in addition, 
\begin{equation}\label{rem-condz2}
    \mbox{if }\mu(r)=0,\mbox{ then } f''(z)\ge0.
\end{equation}    
\end{lemma} 
Before proving Lemma \ref{lem-char}, we derive an alternative representation of the conditions \eqref{rem-condz} and \eqref{rem-condz2} that is useful in several computations. These will be the final conditions really used in the sequel.
\begin{lemma}\label{lem-char2}
    For a given pair $r>0$ and $z\in(\xi_{\min},\infty)$, 
    the conditions in Lemma \ref{lem-char} are equivalent to the following: 
    \begin{itemize}
    \item Either $\mu(r)=0$ and $f'(z)=0$ and $f''(z)\ge0$,
    \item or $\mu(r)>0$, and
        \begin{align}
        \label{rem-cond2old}
            \kappa(r)f'(z)^2
            \le f''(z) + \frac1{\alpha(r)} f'(z).
        \end{align}
    \end{itemize}
\end{lemma}
\begin{proof}[Proof of Lemma \ref{lem-char2}]
    For $r>0$ we have $\alpha(r)>0$. 
    If $\mu(r)>0$, then \eqref{rem-cond2old} is equivalent to $Z\big(\alpha(r),\mu(r),f'(z),f''(z)\big)\le 0$ from \eqref{rem-condz}: simply multiply \eqref{rem-cond2old} by $8\alpha(r)\mu(r)>0$. If $\mu(r)=0$, then 
    \begin{align*}
     Z\big(\alpha(r),\mu(r),f'(z),f''(z)\big) = \big(3\alpha(r)f'(z)/d\big)^2,
    \end{align*}
    and $Z\le0$ is equivalent to $f'(z)=0$.
\end{proof}
\begin{remark}\label{noticevalues}
    The criteria \eqref{rem-condz} and \eqref{rem-cond2old} are central for the rest of the paper. Several remarks are in order.
    \begin{enumerate}
    \item[(a)] If $f'(z)=0$ and $f''(z)\ge0$, then $Z(\alpha(r),\mu(r),f'(z),f''(z))\le0$, independently of the values of $\alpha(r)>0$ and $\mu(r)\ge0$.
    \item[(b)] $\kappa$ is bounded below by $\bar\kappa_d>0$. It diverges to $+\infty$ as $\mu(r)\downarrow0$.
    \item[(c)] \eqref{rem-cond2old} is the generalization of \eqref{f-inequ} to nonlinear diffusions; for $P(u)=u$ it reduces to \eqref{f-inequ}, since then $\alpha=\kappa\equiv 1$. For nonlinear diffusions (even for the power laws $P(u)=u^m$), there is apparently no simple analog of \eqref{e:2.11c} that characterizes all admissible entropies.  
  \end{enumerate}
\end{remark}
Our proof of Lemma \ref{lem-char} requires the following auxiliary result about matrices.
\begin{lemma}
    \label{lem:algebra}
    Let $d>1$. 
    For any symmetric matrix $A\in\R^{d\times d}$ and any unit vector $w\in\R^d$,
    \begin{align}
        \label{eq:algebra}
        \|A\|^2 \ge \frac1d(\tr A)^2 + \frac{d}{d-1}\left(w\cdot Aw - \frac{\tr A}d\right)^2.
    \end{align}
    Moreover, for any prescribed values of $w\cdot Aw$ and $\tr A$, there exist $A$ and $w$ such that \eqref{eq:algebra} is an equality.
\end{lemma}
\begin{proof}[Proof of Lemma \ref{lem:algebra}]
    Let $z_1,\ldots,z_d$ be an orthonormal basis of $\R^d$, with $z_d=w$. Then
    \begin{align*}
        \|A\|^2 
        = \sum_{i,j=1}^d (z_j\cdot Az_i)^2 
        \ge \sum_{j=1}^d (z_j\cdot Az_j)^2 
        &= \sum_{j=1}^{d-1} (z_j\cdot Az_j)^2 + (z_d\cdot Az_d)^2 \\
        &\ge \frac1{d-1}\left(\sum_{j=1}^{d-1} z_j\cdot Az_j\right)^2 + (w\cdot Aw)^2 \\
        &= \frac1{d-1}\big(\tr A-(w\cdot Aw)\big)^2 + (w\cdot Aw)^2 \\
        &= \frac1d(\tr A)^2 + \frac{d}{d-1}\left(w\cdot Aw - \frac{\tr A}d\right)^2,
    \end{align*}
    where we have use the Cauchy-Schwarz inequality for sums for the second estimate. To prove the sharpness result, let $a:=w\cdot Aw$ and $b:=\tr A$ be given. Define $z_1,\ldots,z_d$ as the canonical basis, so in particular $w=z_d=(0,\ldots,0,1)$, and $A$ as diagonal matrix with entries $(b-a)/(d-1)$ in lines one to $d-1$, and $a$ in the last line. In the calculation above, the first inequality is an equality because the chosen $A$ is diagonal, and the second inequality is an equality because the Cauchy-Schwarz inequality is saturated for identical entries.
\end{proof}
\begin{proof}[Proof of Lemma \ref{lem-char}]
\noindent
\underline{Case $d>1$:} 
The non-negativity of \eqref{finform2b} for all $v\in\R^d\setminus \{0\}$, $S\in\R^{d\times d}_{\text{symm}}$ is equivalent to $\bar R(\alpha,\beta,z,v,|v|^2 S)\geq 0$ for all $v\in\R^d\setminus \{0\}$, $S\in\R^{d\times d}_{\text{symm}}$: Indeed, by homogeneity, we have
$$
\bar R(\alpha,\beta,z,v,|v|^2 S)=|v|^4 \bar R\left(\alpha,\beta,z,\frac{v}{|v|}, S\right)\,.
$$
Introducing the variables
  \begin{align*}
y_1:=\frac{v\cdot S\cdot v}{|v|^2} - \frac{\tr S}{d},      \quad
    y_2:=\frac{\tr S}{d}, 
\end{align*}
  inequality \eqref{eq:algebra} implies with $w:=v/|v|$:
  \begin{align}\label{S-inequ}
    \|S\|^2 \ge d y_2^2 + \frac{d}{d-1}y_1^2.
  \end{align}
Lemma \ref{lem:algebra} also states that this inequality is sharp in the sense that for any given values of $\bar y_1$ and $\bar y_2\in\R$,
  there are a vector $v$ and a symmetric matrix $S$ with $y_1(v,S)=\bar y_1$ and $y_2(S)=\bar y_2$, 
  for which equality holds in \eqref{S-inequ}.
Using \eqref{S-inequ} in \eqref{finform2b} we conclude that
  $$
  \bar R(\alpha,\beta,z,v,|v|^2S)\geq |v|^4 p_2(y)
  $$ 
with
  \begin{align*}
    p_2(y) := &\,\alpha\left(d y_2^2 + \frac{d}{d-1}y_1^2\right) + d^2(\beta-\alpha) y_2^2 
       +2\alpha f'(z) (y_1+y_2) + \frac{d}2(\beta-\alpha) f'(z) y_2  \\
       &+\frac12\big(\alpha[f'(z)^2+f''(z)]+f'(z)\big) \\
             =&\, d^2\mu y_2^2+\frac{d\alpha}{d-1}y_1^2 +\frac12(3\alpha+d\mu)f'(z)y_2+2\alpha f'(z)y_1+\frac\alpha2\big(f'(z)^2+f''(z)\big)+\frac12f'(z).
  \end{align*}

As a consequence of the cases of equality in \eqref{S-inequ}, $\bar R(\alpha,\beta,z,v,S)$
is non-negative for all $v\in\R^d$, $S\in\R^{d\times d}_{\text{symm}}$ if and only if $p_2(y)$ is non-negative. Note that we re-included here the case $v=0$ due to continuity of $\bar R$ w.r.t.\ $v$. 
The polynomial $p_2$ is of the special form
  \begin{align*}
    p_2(y) = a_2y_2^2+b_2y_2+a_1y_1^2+b_1y_1+c, \quad y=(y_1,y_2),
  \end{align*}
in which $a_1=d\alpha/(d-1)>0$. Moreover, $a_2=d^2\mu$ must be non-negative, as $p_2$ would not be bounded below otherwise.

First, assume that $a_2=d^2\mu>0$. Then $p_2$ is $c$ plus the sum of the two parabolas $a_1y_1^2+b_1y_1$ and $a_2y_2^2+b_2y_2$, with respective minima $-b_1^2/(4a_1)$ and $-b_2^2/(4a_2)$. The polynomial is thus non-negative if and only if the sum of these minimal values plus $c$ is non-negative, or equivalently
  \begin{align*}
      4a_1a_2c-a_2b_1^2-a_1b_2^2 \ge 0.
  \end{align*}
  After substitution of the respective expressions for $a_1$ to $c$ and elementary manipulations, this inequality becomes equivalent to $Z\le0$ for the polynomial $Z$ in \eqref{eq:Z}. 
  
  Next assume that $\mu=0$, then $a_2=d^2\mu=0$. Boundedness from below of $p_2$ is equivalent to $0=b_2=3\alpha f'(z)/2$, which in view of $\alpha>0$ implies $f'(z)=0$ and hence $Z=0$. The polynomial then simplifies to
  \begin{align}\label{p2}
      p_2(y)=\frac{d\alpha}{d-1}y_1^2+\frac\alpha2f''(z).
  \end{align}
  The minimum is non-negative if and only if $f''(z)\ge0$, which hence has to hold if condition \eqref{rem-cond2a} is satisfied. 
  
  For the reverse direction, note that $Z$ from \eqref{eq:Z} reduces to $Z=(3\alpha f'(z)/d)^2$ in the case $\mu=0$. So, $Z\le0$ implies $f'(z)=0$, and the assumption $f''(z)\ge0$ implies that $p_2$ from \eqref{p2} is non-negative. 

\smallskip
\noindent
\underline{Case $d=1$:} Since this case is similar to the previous case (but simpler), we shall only sketch it. 
The remainder term reads
$$
  \bar R(\alpha, \beta, z, v, v^2S)=v^4 \tilde p_2(S)\,,
$$
with the quadratic polynomial 
$$
  \tilde p_2(S)=\beta S^2 + \frac{3\alpha+\beta}{2}f'(z)S +\frac12 [\alpha\big(f'(z)^2+f''(z)\big)+f'(z)]\,.
$$
If $\mu=\beta>0$, $\tilde p_2$ takes the minimal value $-Z/(16\beta)$. Hence, non-negativity of $\tilde p_2$ and non-positivity of $Z$ are equivalent. If $\mu=\beta=0$, boundedness below of $\tilde p_2$ implies $f'(z)=0$. Thus, $\tilde p_2(S)$ reduces to the constant $\alpha f''(z)/2$, and hence $f''(z)\ge0$ must hold.
\end{proof}
\begin{remark}\label{remlin}\
    \begin{enumerate}
     \item[(a)] Notice that in the linear case $P(u)=u$, the (simplified) remainder condition \eqref{rem-condb} (with $z$ and $r$ decoupled) is equivalent to the admissibility condition \eqref{f-inequ} on entropies -- due to Lemma \ref{lem-char} and Remark \ref{noticevalues}-(c). In \cite{AMTU,BaEm84} only one of the directions was proven, i.e.: Under the Bakry-\'Emery condition {\bf (HV2)}, the entropy condition \eqref{e:2.11c} implies the non-negativity of the remainder \eqref{rem-condb}.
    \item[(b)]   After some elementary manipulations assuming $f'\neq0$ in $(\xi_{min},\infty)$, the condition \eqref{rem-cond2old} is equivalent to
       \begin{align}
    \label{rem-cond3}
    \left(\frac1{f'}\right)' + \kappa \le \frac1\alpha\frac1{f'}\,.
  \end{align}
    \item[(c)] In the linear case $P(u)=Du$ with some constant $D>0$, we have $\alpha(u)=\beta(u)=D$, and hence \eqref{rem-cond3} for $f'\neq 0$ in $\R$ reduces to
      \begin{align}\label{condD}
        \left(\frac1{f'}\right)' + 1 \le \frac1D\frac1{f'},
      \end{align}
      which is equivalent to  \eqref{f-inequ2}.
       \item[(d)] Assuming for nonlinear diffusions that $P'(0+)$ is finite, then $(\alpha(r),\beta(r))\to (P'(0+),P'(0+))$ as $r\to 0$. Assume that the strict McCann condition $d\beta(r)-(d-1)\alpha(r)>0$ holds for all $r >0$. Thus condition \eqref{rem-cond2old} for $r\to 0+$ and using $\phi(0+)-\bar C=\xi_{\min}$ in \eqref{rem-cond} implies 
       $$
       \beta(0) (f')^2 \leq \beta(0) f'' + f'\qquad \forall\,z\in (\xi_{\min},\infty)
       $$
       in the limit, distinguishing the following two cases:   
       \begin{itemize}
           \item In the particular case of degenerate diffusions, $P'(0+)=0$, then this reduces to $f'\ge 0$.
           
           \item  For regular non-degenerate diffusions with $P'(0+)\in(0,\infty)$, we have $\xi_{\min}=-\infty$, and then a scaled version of Lemma \ref{f-solutions}-(a) implies that $0\leq f'\leq \beta(0)^{-1}$. 
           
           With the scaling from the Remark \ref{linear}-(b), and $D$ now replaced by $\beta(0)$, we again deduce $\tilde f''(\bar \xi)+\tilde f'(\bar \xi)-\tilde f'^2(\bar \xi)\geq 0$ for all $\bar\xi\in\R$. 
           
           This has an important consequence: The linear behavior of $P$ at the origin, i.e. $P'(0+)\in(0,\infty)$, implies that corresponding admissible entropies must satisfy \eqref{f-inequ2} with $D=\beta(0)$. Hence, they form a subset of the admissible entropies in the linear diffusion case.

           \item For nonlinearities with $P'(0+)=0$ or $P'(0+)=\infty$, however, we will in general \emph{not} obtain the conditions \eqref{e:2.11c} or \eqref{f-inequ2}.
       \end{itemize}
    \end{enumerate}
\end{remark}
{F}rom this point on, we can proceed in different ways: 
On the one hand we can fix the nonlinearity $P(u)$, and hence the (continuous) \emph{nonlinearity curve} $(\alpha(u),\,\beta(u))$, $u\ge0$ in the quarter plane $(\R_0^+)^2$ as well as the range $(\ximin,\infty)$ of the function $\xi$. Then, the goal is to find all entropies such that the generalized Bakry-\'Emery procedure of \S\ref{sec:4.1} is feasible. More precisely, we define:
\begin{definition}\label{def:adm-entr}
    Let a nonlinearity $P$ with associated $(\alpha,\beta)$ from \eqref{eq:P2ab} and $\kappa$ from \eqref{eq:def-nu}, as well as an entropy generating function $g\in C^3((\ximin,\infty))$ with associated $f$ from \eqref{entrofgen} be given such that $f'\ge0$. We say that the relative entropy functional $\mathcal H_g$ from Definition \ref{genrelentropy} or Definition \ref{genrelentropy2}, respectively, is \emph{admissible for $P$} if for all $r>0$ and for all $z\in[\phi(r)-\bar C,\infty)$, either $d\beta(r)=(d-1)\alpha(r)$ and $f'(z)=0$, or $d\beta(r)>(d-1)\alpha(r)$ and
    \begin{align}
      \label{rem-cond2}
      \kappa(r)f'(z)^2
      \le f''(z) + \frac1{\alpha(r)} f'(z).
    \end{align}
    If this is the case, we shall also call the generator $g$ and function $f$ admissible for $P$, as well as $P$ admissible for $g$ or $f $. 
\end{definition}
Note that the above hypothesis $f'\ge0$ follows from the condition in Lemma \ref{lem-char2}, see Proposition \ref{prop4.10} below (there written as $y\ge0$). Hence the hypothesis $f'\ge0$ could be dropped in Definition \ref{def:adm-entr}. A consequence of $f'\ge0$ is that if $f'(z)=0$ implies $f''(z)=0$. Hence by Lemma \ref{lem-char}, $\mathcal H_g$ is admissible if and only if for all $r>0$ and $z\in[\phi(r)-\bar C,\infty)$ it holds that
\begin{align}
    \label{rem-cond2c}
    \mu(r)\ge0 \quad \text{and}\quad Z\big(\alpha(r),\mu(r),f'(z),f''(z)\big)\le0. 
\end{align}
With the aim to find all entropies, the characterization of admissible entropies via a family of differential inequalities obtained by combining \eqref{rem-cond2} and \eqref{rem-cond} will be crucial. This is the main goal of \S \ref{sec-quasilin}. 

On the other hand we can fix a general entropy represented by the {\it entropy function} $f(\xi)$, with $\xi\in(\xi_{\min},\infty)$, 
and aim at finding all corresponding nonlinearities $P(u)$ such that the entropy method is applicable. This is the main goal of \S\ref{sec:entropygivennonlin}.  More precisely, $P(u)$ is called an \emph{admissible nonlinearity} for a given entropy  $\mathcal H_g$, if the latter is an admissible entropy for $P(u)$. 

We can also try to find the set of nonlinearities for which all admissible entropies of the linear case 
still lead to the entropy inequality \eqref{eq:alphaomega}.
To this end the equivalent remainder condition \eqref{rem-cond2} is convenient. These are the main goals of \S \ref{sec:admiss-sets-for alllin}.

In general, these results will depend on the dimension $d$. In fact, we anticipate that the remainder condition  \eqref{rem-cond2} becomes more restrictive with increasing dimension, see \S\ref{sec:geompic}. Hence, it will yield fewer admissible entropies for a given nonlinearity, and similarly fewer admissible nonlinearities for a given entropy as $d$ increases. This will be illustrated with a particular example already in \S\ref{sct:xmp-pme2}. 

%%%%%%%%%%%%%%%%%%%%%%%%%%%%%%%%%%%%%%%%%%%%%%%%%%%%%%%%%%%%%%%%%%%%%%%%%%%%%%%%

\section{Admissible relative entropies for nonlinear diffusion equations}\label{sec-quasilin}
\setcounter{equation}{0}

The goal of this section is to find all admissible entropies for a given nonlinearity. Recall that $P$ defines $\xi_{\min}\in[-\infty,0)$ via Remark \ref{xi-range}, and $g$ is defined on $(\xi_{\min},+\infty)$ with $g(0)=0$, and that $f=\log g'$. In addition, we also introduce $y:=f'$ and the notation 
\begin{equation}\label{def-varphi}
\varphi(z):=\overline{\phi}^{-1}(z+\bar C), 
\end{equation}
with the ``generalized'' inverse $\overline{\phi}^{-1}$ from \eqref{phi-inv}, and $\bar C$ given in Remark \ref{xi-range}. Note that $\varphi(z)>0$ for all $z>\xi_{\min}$ and it is increasing. Recall further the notations $\mu$ and $\kappa$ from \eqref{eq:def-nu}.

\subsection{A motivating example}
\label{sct:xmp-pme2}
To motivate the considerations in this section, we start with an example on the specific choice $P(r)=r^2$, for which a family of entropy functionals can be computed explicitly. The rest of this section is largely devoted to a generalization of the findings in this special case.

For $P(r)=r^2$, we have $\alpha(r)=r$, $\beta(r)=2r$, and $\phi(r)=2(r-1)$ with $\xi_\text{min}=-2-\bar C$. 
Here $\bar C>0$ is given by Proposition \ref{charac}, i.e., it is chosen to adjust the mass of the stationary solution 
\begin{align}
\label{eq:PME2stst}
    u_\infty(x) = \left(\frac{\bar C-V(x)}{2}+1\right)_+\,.
\end{align}
We shall now use Lemmas \ref{prop4.5} and \ref{lem-char} to obtain non-standard entropies. Introduce accordingly $\mu(r)=(d+1)r/d>0$ and $\kappa(r)\equiv \frac98\frac{d}{d+1}>0$. Since $\mu$ is always positive, the condition on a function $g$ to determine an entropy according to Definition \ref{genrelentropy2} is that its cousin $f=\log g'$ satisfies the differential inequality in \eqref{rem-cond2}. Below, we only consider $g$'s for which $f$ is increasing; this is actually no restriction as we shall see from Proposition \ref{prop4.10} below. 

To simplify the differential inequality in \eqref{rem-cond2}, we introduce $y:=f'\ge0$. Using that $1+z/2$ is the inverse of $\phi(r)$, we have $\varphi(z)=1+(z+\bar C)/2$ and conditions \eqref{rem-cond} and \eqref{rem-cond2} are combined to give 
\begin{equation}
\label{eq:dineq2}
    \begin{split}
      y'(z)&\ge  \sup_{\phi(r)-\bar C<z}\left[y(z)\left(\kappa y(z)-\frac1{\alpha(r)}\right)\right] =
      y(z)\left(\kappa y(z) - \inf_{r<\varphi(z)}\frac1r\right) \\ &=y(z)\left(\kappa y(z) - \frac1{1+(z+\bar C)/2}\right),\quad z>-2-\bar C.
    \end{split}
\end{equation}
Consider the corresponding differential equation. The general form of non-trivial solutions that are non-negative and defined for all $z>-2-\bar C$ is easily determined:
\begin{align*}
    \tilde y_B(z) = \frac1{(2+z+\bar C)(\kappa+B(2+z+\bar C))},
\end{align*}
with a parameter $B\geq 0$. Among these $\tilde y_B$, there is a pointwise largest one, namely $\tilde y_0(z)=\frac{1/\kappa}{2+z+\bar C}$. This gives rise to the entire family of functions 
\[ y_p := (p-1)\tilde y_0, \quad 1\le p\le 2, \]
satisfying the differential \emph{in}equality in \eqref{eq:dineq2} above. Indeed, $y_2=\tilde y_0$ clearly satisfies the inequality since it is a solution to the corresponding equation, and validity for the other $y_p$'s follows immediately from the special structure of \eqref{eq:dineq2}; recall that $\kappa>0$.

By Lemmas \ref{prop4.5} and \ref{lem-char}, the $g_p$'s associated to these $y_p$'s define entropies via Definition \ref{genrelentropy2} and \eqref{eq:7b}, for which the dissipation inequality \eqref{eq:alphaomega} holds. We make these entropies explicit choosing the integration constant of $f_p$ such that the resulting entropies $\mathcal H_p$ will then be increasing in $p$ and recover  $\mathcal H_1$ as the classical entropy (see Proposition \ref{prop:genm} for further details): 
\begin{align*}
    f_p'(z) &= \frac{(p-1)/\kappa}{2+z+\bar C}, 
    \quad 
    f_p(z) = \frac{p-1}\kappa\log\left(\frac{2+z+\bar C}{2+\bar C} \right)+\log\left(1+\frac{p-1}{\kappa}\right),\\
    g_p'(z) &= e^{f_p(z)} = \left(1+\frac{p-1}{\kappa}\right)\left(\frac{2+z+\bar C}{2+\bar C} \right)^{\frac{p-1}{\kappa}}, \\ 
    g_p(z) &= \left[\left(\frac{2+z+\bar C}{2+\bar C} \right)^{1+\frac{p-1}{\kappa}} - 1\right]\left(2+\bar C\right),
    \end{align*}
    \begin{align*}
    G_p(a,b;x) &= \int_b^a g_p\big(\phi(r)+V(x)-\bar C\big)\,dr \\
          &= (2+\bar C)^{-(p-1)/\kappa} \int_b^a \big(2r+V(x)\big)^{1+(p-1)/\kappa}\,dr -(a-b) (2+\bar C)\\
          &= \frac{(2+\bar C)^{-(p-1)/\kappa} }{2(2+(p-1)/\kappa)}
          \big(2r+V(x)\big)^{2+(p-1)/\kappa}\Big|_{r=b}^{r=a} -(a-b)(2+\bar C).
\end{align*}
Recalling the form \eqref{eq:PME2stst} of the stationary state $u_\infty$, we obtain
\begin{align*}
    G_p\big(u(x),u_\infty(x);x\big) 
    =\,&\frac{ \big(2u(x)+V(x)\big)^{2+(p-1)/\kappa} - \big(2u_\infty(x)+V(x)\big)^{2+(p-1)/\kappa}}{2(2+\bar C)^{(p-1)/\kappa}(2+(p-1)/\kappa)}\\
      &-(u(x)-u_\infty(x))(2+\bar C)\,.
\end{align*}
After integration in $x$, the second term disappears thanks to the conservation of mass. In summary, the entropy functionals take the form
\begin{align}\label{H_p}
    \mathcal H_p(u|u_\infty)
    =\int_{\R^d}  \frac{ \big(2u(x)+V(x)\big)^{2+(p-1)/\kappa} - \big(2u_\infty(x)+V(x)\big)^{2+(p-1)/\kappa}}{2(2+\bar C)^{(p-1)/\kappa}(2+(p-1)/\kappa)}\,dx.
\end{align}
These integrals cannot be split since both terms are not integrable separately (since $V\to\infty$ for $|x|\to\infty$). However, for $p=1$, the expression under the integral simplifies significantly and gives back the standard entropy,
\begin{align*}
    \mathcal H_1(u|u_\infty)
    = \int_{\R^d} \big[u(x)^2-u_\infty(x)^2+V(x)\big(u(x)-u_\infty(x)\big)\big]\,dx.
\end{align*}

Next we consider the dimensional dependence of the entropy functional $\mathcal H_2(u|u_\infty)$ which corresponds to the maximal admissible function 
$$
  Y_2(z)=y_{2,2}(z)=\frac{1}{\kappa_2(d)(z+\bar C+2)},
$$ 
see Proposition \ref{prp:Pintegrable} and Remark \ref{rem:Y-max-entropy} below. Then the exponent in \eqref{H_p} for $p=2$ is $2+1/\kappa_2(d)=\frac{26d+8}{9d}$; it decreases from $\frac{34}{9}$ for $d=1$ to $\frac{26}{9}$ in the limit $d\to\infty$. Hence, the $\mathcal H_2$ entropy with exponent $\frac{34}{9}$ is admissible for $d=1$, but not for higher dimensions -- in agreement with Remark \ref{noticevalues}(d).

We shall now demonstrate that for different values of $p\in[1,2]$, the corresponding entropies $\mathcal H_p$ indeed contain different information about the behaviour of $u$. Recall that for any exponent $q>1$, there are positive constants $c_q<C_q$ such that
\begin{align*}
    c_q (qs^{q-1}+t^{q-1}) \le (2s+t)^{q-1} \le C_q (qs^{q-1}+t^{q-1})
    \quad\text{for all $s,t\ge0$}.
\end{align*}
An integration in $s$ yields with new positive constants $c_q'<C_q'$ that
\begin{align*}
    c_q'(s^q+t^{q-1}s)\le(2s+t)^q-t^q\le C_q'(s^q+t^{q-1}s) \quad\text{for all $s,t\ge0$}.
\end{align*}
Let $q:=2+(p-1)/\kappa$, substitute $t:=V(x)$, and use first $s:=u(x)$, and then $s:=u_\infty(x)$. Subtraction of the resulting inequalities yields, after elementary manipulations:
\begin{align*}
    &c_q'\big[u^{2+(p-1)/\kappa}+V(x)^{1+(p-1)/\kappa}u\big]
    -C_q'\big[u_\infty^{2+(p-1)/\kappa}+V^{1+(p-1)/\kappa}u_\infty\big] \\
    &\quad \le \big(2u(x)+V(x)\big)^{2+(p-1)/\kappa} - \big(2u_\infty(x)+V(x)\big)^{2+(p-1)/\kappa} \\
    &\quad \quad \le C_q'\big[u^{2+(p-1)/\kappa}+V(x)^{1+(p-1)/\kappa}u\big].
\end{align*}
Assume (only for simplicity of presentation) that $V(x)=|x|^2$. Then, there are positive constants $c<C$ such that
\begin{align*}
    &c\int \big[u(x)^{2+(p-1)/\kappa}+|x|^{2+2(p-1)/\kappa}u(x)\big]\,dx - C \\
    &\le \mathcal H_p(u|u_\infty)
    \le C\int \big[u(x)^{2+(p-1)/\kappa}+|x|^{2+2(p-1)/\kappa}u(x)\big]\,dx.
\end{align*}
That is, $\mathcal H_p(u|u_\infty)$ is finite if and only if $u\in L^{2+(p-1)/\kappa}$ and $u$'s moment of order $2+2(p-1)/\kappa$ is finite. Our results imply that, if the initial datum $u_0$ in \eqref{GFP2} satisfies the aforementioned conditions for some $p\in[1,2]$, then the solution $u(t)$ to \eqref{GFP} satisfies the same conditions at any later time $t>0$. Moreover, one has the bound
\begin{align*}
    &c\int \big[u(x,t)^{2+(p-1)/\kappa}+|x|^{2+2(p-1)/\kappa}u(x,t)\big]\,dx \\
    &\le C\left(1+e^{-4t}\int \big[u_0(x)^{2+(p-1)/\kappa}+|x|^{2+2(p-1)/\kappa}u_0(x)\big]\,dx\right).
\end{align*}
It follows that the information contained in the $\mathcal H_p$'s are not equivalent for different values of $p\in[1,2]$, but become stronger as $p$ increases. That is, the standard entropy $\mathcal H_1$ admits the largest variety of initial data $u_0$, but also yields the least control on integrability of the solution $u(t)$, whereas $\mathcal H_2$ is most restrictive for $u_0$, but also provides the strongest control on $u(t)$.

As an important final remark, let us emphasize that the use of condition \eqref{rem-cond} instead of the easier condition \eqref{rem-condb} was essential here. In fact, combining \eqref{rem-condb} with \eqref{rem-cond2} would then give instead of \eqref{eq:dineq2} the following condition for a non-negative function $y:(\xi_\text{min},\infty)\to\R$:
\begin{align*}
    y'(z)\ge\sup_{r>0}\left[y(z)\left(\kappa y(z)-\frac1{\alpha(r)}\right)\right]
    = y(z)\left(\kappa y(z) - \inf_{r>0}\frac1r\right)
    = \kappa y(z)^2.
\end{align*}
However, this is a contradiction to $y$ being defined for all $z>\xi_\text{min}$, since all positive solutions to the associated differential equation $\tilde y'=\kappa\tilde y^2$ blow up in finite $z$. In particular, among the above entropies $\mathcal H_p$, only $\mathcal H_1$ --- the standard one, corresponding to $y\equiv0$ --- meets the criterion \eqref{rem-condb}.

%%%%%%%%%%%%%%%%%%%%%%%%%%%%%%%%%%%%%%%
\subsection{Fundamental characterization of admissible entropies}
%%%%%%%%%%%
Below, we consider general nonlinearities $P$ with the goal to characterize functions $g$ that satisfy conditions \eqref{rem-cond} and \eqref{rem-cond2}, and thus give rise to entropy functionals $\mathcal H_g$ that have the dissipation property \eqref{eq:alphaomega}. 

With the help of Lemma \ref{lem-char}, we start by deriving a simpler condition equivalent to that in Lemma \ref{prop4.5}. To simplify the remainder condition \eqref{rem-cond}, we define the following function for $z>\xi_\text{min}$ and $\eta\ge0$:
\begin{align}\label{deftheta}
  \vartheta(z,\eta) := \sup_{0<r<\varphi(z)} \left\{\kappa(r)\eta-\frac1{\alpha(r)}\right\}\,.
\end{align}
For later reference, we summarize some properties of $\vartheta$.
\begin{lemma}
    \label{lem:thetaprop}
    $\vartheta$ is non-decreasing in both arguments, and is convex lower semi-continuous with respect to $\eta$. Moreover, $\vartheta$ might attain $+\infty$, but never $-\infty$.
\end{lemma}
\begin{proof}
    As $\varphi$ is non-decreasing, also the range of $r$'s on which the supremum in \eqref{deftheta} is taken can only grow as $z$ increases, proving $\vartheta$'s monotonicity with respect to $z$. Since $\kappa(r)>0$ for all $r\ge0$, the expression $\eta\mapsto\kappa(r)\eta$ is non-decreasing in $\eta$ for each fixed $r\ge0$, and so is the supremum with respect to $r\in(0,\varphi(z))$. Further, as supremum of affine functions in $\eta$, $\vartheta$ is convex lower-semicontinuous.
    
    Finally, observe that $\alpha(r)=P(r)/r$ is positive for each $r>0$. Thus for each $z>\xi_\text{min}$ and $\eta\ge0$, the supremum in \eqref{deftheta} is over a non-empty set of real values and thus possibly $+\infty$, but never $-\infty$. 
\end{proof}
\begin{prop}\label{prop4.10}
    Condition \eqref{rem-cond} is equivalent to the following:
    $\mu(r)\ge0$ for all $r>0$, $y(z)\ge0$ for all $z>\xi_{\min}$, 
    and either
    \begin{itemize}
    \item[(a)] $\mu(r)>0$ for all $r>0$,
        $\vartheta(z,y(z))<\infty$ for all $z>\xi_{\min}$,
        and
        \begin{equation}\label{cond-aa}
            y'(z) \ge y(z)\,\vartheta(z,y(z)),\quad z>\xi_{\min} \,,
        \end{equation}
   \end{itemize}
        or
    \begin{itemize}
    \item[(b)] $\mu(r)=0$ for some $r>0$, and $y\equiv 0$.
    \end{itemize}
    In particular, these conditions imply \eqref{eq:alphaomega}, and therefore also \eqref{mainconclusion1}, \eqref{mainconclusion2} and \eqref{mainconclusion3}.
\end{prop}
The function $y=(\log g')',\:(\xi_{\min},\infty)\to\R$ is called \emph{admissible} for $P(u)$ is the corresponding relative entropy $\mathcal H_g$ is admissible, as defined in Definition \ref{def:adm-entr}. 
Proposition \ref{prop4.10} yields the following characterization of admissible entropies, both in terms of the functions $y$ and $g$: 
\begin{corollary}\label{adm-entr}\
\begin{itemize}
    \item[(a)] 
    Assume $P(u)$ is such that $\mu(r)\ge0$ for all $r>0$, and let $y:=f'\: :\: (\xi_{\min},\infty)\to [0,\infty)$ be a $C^1$ function. Then $y$ is \emph{admissible for $P(u)$} if and only if it satisfies condition (a) or (b) from Proposition \ref{prop4.10} above. 
        
    \item[(b)] Assume $P(u)$ and the admissible function $y$ are as in Part (a), and let a function $g\: :\: (\xi_{\min},\infty)\to \R$ satisfy $y=(\log g')'$ with $g(0)=0$. Then the general relative entropy functional $\mathcal H_g$ from Definition \ref{genrelentropy} or Definition \ref{genrelentropy2} is \emph{admissible for $P(u)$}.
\end{itemize}
    Note that the admissible function $y$ defines $g$ uniquely, up to a positive multiplicative constant which appears as an integration constant for $f$.
\end{corollary}

\begin{remark}
    The fact that $y$ is defined on all of $(\xi_{\min},+\infty)$ is an essential part of the definition. We shall see below that for a variety of relevant choices for $P(u)$, the only global non-negative solution to the differential inequality \eqref{cond-aa} is $y\equiv0$. 
\end{remark}
\begin{remark}
    In the linear case $P(u)=Du$ with $D>0$, Definition \ref{adm-entr} coincides with Definition \ref{d:2.1}: in this case, we have $\alpha=\beta\equiv D$, $\mu\equiv D/d$ and $\kappa\equiv1$. Hence, Condition \eqref{cond-aa} simplifies to $y'(z)\ge y(z)\big(y(z)-\frac1D\big)$, $z\in\R$, which is equivalent to \eqref{f-inequ}, cf.\ \eqref{condD}.
\end{remark}
\begin{proof}[Proof of Proposition \ref{prop4.10}]
First we show that $y$ is non-negative. Applying Lemma \ref{lem-char} to \eqref{rem-cond} yields $\mu(r)\ge0$ and one of the following two cases:
\begin{itemize}
\item[(a)] If $\mu\ne0$ on $\R^+$, then $\mu>0$ due to Lemma \ref{lem-char} and \eqref{rem-cond2} is equivalent to the differential inequality 
\begin{equation*}%\label{cond-d}
  y'(z) \ge y(z)\left(\kappa(r)y(z)-\frac1{\alpha(r)}\right)
\end{equation*}
for each $z>\xi_{\min}$, and all $r>0$ such that $\phi(r)-\bar C<z$. Notice that this family of differential inequalities is equivalent to
\begin{equation}\label{cond-a}
 y'(z) \ge \sup_{0<r<\varphi(z)} \left[y(z)\left(\kappa(r)y(z)-\frac1{\alpha(r)}\right)\right],
\end{equation}
for all $z>\xi_{\min}$.

\item[(b)] If $\mu$ has the zero $\bar r>0$, then Lemma \ref{lem-char} 
applied to the pairs $(\bar r,z)$ with $z\in[\phi(\bar r)-\bar C,\infty)$ according to \eqref{rem-cond} 
yields the condition $y(z)=0$ for $z\in [\phi(\bar r)-\bar C,\infty)$. Let us define $r_0:=\inf\{r>0\,|\,\mu(r)=0\}\geq 0$, then $y(z)=0$ for $z\in (\phi(r_0)-\bar C,\infty)$. Since $\phi$ is strictly increasing, then the differential inequality 
\begin{equation*}%\label{cond-d}
  y'(z) \ge y(z)\left(\kappa(r)y(z)-\frac1{\alpha(r)}\right)
\end{equation*}
holds for $0<r<r_0$ and $z\in(\phi(r)-\bar C,\phi(r_0)-\bar C]$. 
Notice that this family of differential inequalities is equivalent to
\begin{equation}\label{cond-d}
  y'(z) \ge \sup_{0<r<\varphi(z)} \left[y(z)\left(\kappa(r)y(z)-\frac1{\alpha(r)}\right)\right],
\end{equation}
for all $z\in(\xi_{\min},\phi(r_0)-\bar C)$.
\end{itemize}

\

Case (a): Assume first that $\mu>0$ on $\R^+$. We start by considering non-degenerate diffusions, i.e., $\xi_\text{min}=-\infty$. Let a $C^1$ function $y:(-\infty,\infty)\to\R$ be given that satisfies \eqref{cond-a} and there exists a $z_*$ such that $y(z_*)<0$.
Then \eqref{cond-a} implies particularly
\begin{equation}\label{e5.5a}
  y'(z_*)\ge y(z_*)\left(\bar\kappa_d y(z_*)-\frac1{\alpha(0)}\right)>0,
\end{equation}
where $\alpha(0)\in(0,\infty]$. The sign of $y'(z_*)$ implies $y(z)<0$ on the maximal left neighborhood $z\in(\tilde z,z_*]$ with $\tilde z\in[-\infty,z_*)$ and $y(\tilde z)=0$ (if $\tilde z\ne-\infty$). 
In analogy to \eqref{e5.5a} we also have
\begin{equation}\label{e5.6a}
  y'(z)\ge y(z)\left(\bar\kappa_d y(z)-\frac1{\alpha(0)}\right)
\end{equation}
on $(\tilde z,z_*]$. If $\tilde z$ were finite, $y$ would have a (negative) minimum at some $\bar z\in(\tilde z,z_*)$, i.e.\ $y(\bar z)<0$ and $y'(\bar z)=0$, contradicting the analog of \eqref{e5.5a}, with $\bar z$ replacing $z_*$. Hence $\tilde z=-\infty$. 
Then, the same argument as in Lemma \ref{f-solutions}-(b) implies that $y$ would diverge to $-\infty$ at some finite $\bar z<z_*$. Hence a global solution of \eqref{e5.6a} must satisfy $y\ge0$.

In the degenerate case, where $\xi_\text{min}=\phi(0+)-\bar C>-\infty$, we have to argue differently to show the non-negativity of $y$. We use instead that $\alpha(0+)=P'(0+)=0$. Given an arbitrary $z>\xi_\text{min}$, we have $\phi(r)-\bar C<z$ for sufficiently small $r>0$ due to the monotonicity of $\phi$. Rewriting the differential inequality \eqref{cond-a} yields
  \begin{align*}
    y(z)\ge\alpha(r)\big(\kappa(r)y(z)^2-y'(z)\big)
  \end{align*}
for any $z>\xi_\text{min}$ and sufficiently small $r>0$, using $\alpha(r)>0$. In the limit $r\to0+$, this produces $y(z)\ge0$.\\

Since $y\geq 0$, the inequality \eqref{cond-aa} in statement (a) follows directly from \eqref{cond-a} together with the definition of $\vartheta$ in \eqref{deftheta} since $y\in C^1$. Moreover, since $y\in C^1$ then $y'(z)=f''(z)<\infty$ in condition \eqref{rem-cond2}. This together with $\vartheta(z,0)\le0$ imply that  $\vartheta(z,y(z))<\infty$ for all $z>\xi_{\min}$ by taking the supremum in $0<r<\varphi(z)$ in condition \eqref{rem-cond2}.\\

Case (b): We show that $y\geq 0$, even if $\mu$ has a zero. It is enough to notice that that the proof above in Case (a) still applies to a differentiable $y\: :\:(\xi_{\min},\infty)\to\R$ satisfying \eqref{cond-d}, implying that $y\ge0$ in this case.

Next we shall show that condition \eqref{rem-cond} implies statement (b). As a consequence of being $y\geq 0$, this is equivalent to show that a positive function $y$ satisfying \eqref{cond-d} cannot be connected continuously to $y\equiv 0$ on $[\phi(r_0)-\bar C,\infty)$. From \eqref{cond-d}, $\kappa(r)>0$, and $y\ge0$ we conclude that
\begin{equation}\label{e5.6b}
  y'(z)\ge y(z)\left(-\frac1{\alpha(r_*)}\right)
\end{equation}
holds on any interval $[\bar z,\phi(r_0)-\bar C)$ with $\bar z>\xi_{\min}$ and some $r_*\in(0,\varphi(\bar z))$. Note that $\varphi(\bar z)<r_0$. For the uniform choice of $r_*$ (w.r.t. $z$) we used that $\varphi$ is strictly increasing.\\
\indent
Assume now that there exists a $z_*\in[\bar z,\phi(r_0)-\bar C)$ with $y(z_*)>0$. Then \eqref{e5.6b} implies 
$$
  y(z)\ge y(z_*)e^{-\frac{z-z_*}{\alpha(r_*)}},\quad 
  z\in[z_*,\phi(r_0)-\bar C).
$$
But this exponential lower bound does not allow for a continuous connection to $y\equiv0$ on $[\phi(r_0)-\bar C,\infty)$. Hence $y\equiv0$ on $(\xi_{min},\infty)$ follows.\\

For the reverse direction we first assume \eqref{cond-aa} (in the case $\mu\ne0$ on $\R^+$). This implies \eqref{rem-cond2} for each $z>\xi_{\min}$ with $r\in(0,\varphi(z))$. Hence, \eqref{rem-cond} follows due to Lemma \ref{lem-char}. If $\mu$ has a zero on $\R^+$, then we assume that $y\equiv 0$. Hence, both conditions of Lemma \ref{lem-char} trivially hold, and this implies \eqref{rem-cond} too in this case.
\end{proof}

%%%%%%%%%%%
\subsection{Classes of non-/admissible entropies}
%%%%%%%%%%%
We shall draw various conclusions from Proposition \ref{prop4.10}.
\begin{corollary}
  If $P$ satisfies the McCann condition $\mu(r)\ge0$ for all $r>0$, then $y\equiv0$ is admissible. The corresponding entropy is the standard one, given in Definition \ref{standrelentropy}. If the McCann condition is violated, then there is no admissible entropy at all.
\end{corollary}
\begin{proof}
  This is obvious since $y\equiv0$ satisfies \eqref{cond-aa}.
\end{proof}
\begin{example}
    In any dimension $d>1$, consider the following regular non-degenerate nonlinearity with sublinear growth:
    \[ P(r) = (r+1)^q-1 \quad\text{with some}\quad 0<q<1.\]
    We then have
    \begin{align*}
        \mu(r) 
        &= \beta-\frac{d-1}d\alpha 
        = q(1+r)^{q-1}-(d-1)\frac{(r+1)^q-1}{dr} \\
        &=\frac{(dq-(d-1))(1+r)^q-dq(1+r)^{-(1-q)}+(d-1)}{dr},
    \end{align*}
    and thus
    \begin{align*}
        \frac{dq-(d-1)}{dr}((1+r)^q-1)\le\mu(r)\le\frac{(d-1)-((d-1)-dq)(1+r)^q}{dr}. 
    \end{align*}
    If $(d-1)/d\le q<1$, then we have $\mu(r)\ge0$ for all $r>0$, and so at least $y\equiv0$ is admissible. If instead $0<q<(d-1)/d$, then $\mu(r)<0$ for all sufficiently large $r>0$. Consequently, there is no admissible function at all.
\end{example}
\begin{corollary}\label{cor:pfamily}
  If $y_*$ is admissible, then also $y_p:=(p-1)y_*$ is admissible, for each $p\in[1,2]$.
\end{corollary}
\begin{proof}
    If $y_*\equiv0$, there is nothing to prove.
    Assume that $y_*$ is not trivial, hence that $y_*$ satisfies the differential inequality \eqref{cond-aa}.
    Recall that $\kappa(r)>0$ by Remark \ref{noticevalues}, and so, for every $r>0$ and $z>\xi_{\min}$, we have $(p-1)\kappa(r)y_*(z)\le\kappa(r)y_*(z)$,
    and thus $\vartheta(z,y_p(z))\le\vartheta(z,y_*(z))$ by definition of $\vartheta$ in \eqref{deftheta}. It follows that
    \begin{align*}
        y_p'(z) = (p-1)y_*'(z) \ge (p-1)y_*(z)\vartheta(z,y_*(z)) \ge y_p(z)\vartheta(z,y_p(z)),
    \end{align*}
    meaning that $y_p$ satisfies the differential inequality as well.
\end{proof}

The first main observation is a negative result about the singular non-degenerate case, i.e., $\alpha(0+)=+\infty$.
\begin{prop}
  \label{cor:admin1}
    Assume that $\alpha(0+)=+\infty$. Then only $y\equiv0$ is admissible. 
\end{prop}
\begin{proof}
    Let $y$ be admissible, and hence $y(z)\ge0$ on $\R$. We assume $\mu(r)>0$ for all $r>0$ as otherwise, there would be nothing to prove (due to Proposition \ref{prop4.10}-(b)). Using $\alpha(0+)=\infty$, \eqref{cond-aa} then implies
    \[ 
    y'(z)\ge\bar \kappa_d y(z)^2,\quad z\in\R, 
    \]
    taking into account \eqref{eq:def-nu}. 
    But this differential inequality possesses no global solution $y$ that is positive at some $z_*>\xi_{\min}$: by comparison, $y(z)$ would be bounded below by the solution to the corresponding differential equation for $z\ge z_*$, that is,
    \[ y(z) \ge \frac1{1/y(z_*)-\bar \kappa_d(z-z_*)}.\] 
    This bound blows up as $z\uparrow z_*+1/(\bar \kappa_d y(z_*))$. Thus $y\equiv0$.
\end{proof}
From now on, we focus on the two remaining cases, the regular non-degenerate and degenerate diffusions, that is $\alpha(0)\in[0,\infty)$.
\begin{prop}
  \label{prp:Pintegrable}
  Assume $\alpha(0+)<\infty$, and that $\phi\: :\: (0,\infty)\to(\xi_{\min}+\bar C,\infty)$ is a bijection, i.e., $r\mapsto P'(r)/r$ is not integrable at $r\to+\infty$. 
  
  If some non-trivial admissible function $y$ exists, then $r\mapsto \frac{\kappa(r)}{r}\,\frac{P'(r)}{P(r)}$ is integrable at $r\to\infty$. Moreover, if this is the case, then $y(z)\le Y(z)$, $z>\xi_{\min}$, where $Y:(\xi_{\min},\infty)\to[0,\infty)$ is defined by
    \begin{align}
    \label{eq:hatYdef}
        Y(z)
        := \left(P(\varphi(z))\int_{\varphi(z)}^\infty\frac{\kappa(r)}{r}\frac{P'(r)}{P(r)}dr \right)^{-1}.
    \end{align}
\end{prop}
\begin{remark}
    In some cases, detailed below, $Y$ itself is an admissible function in the sense of Corollary \ref{adm-entr}. If that happens, then $Y$ is obviously the pointwise maximal admissible function. 

    We also note that such a  pointwise maximal admissible function always exists, but we do not give the proof here, as it would not help our further discussion. 
    This pointwise maximal admissible function is the analog of the upper threshold $y\equiv1$, which corresponds to the quadratic entropy for linear Fokker-Planck equations \eqref{GFP}, see Lemma \ref{f-solutions}.
\end{remark}
\begin{proof}[Proof of Proposition \ref{prp:Pintegrable}]
  Let $y:(\xi_\text{min},\infty)\to [0,\infty)$ be a non-trivial admissible function. Since $y$ is non-trivial, there exists a $z_*>\xi_\text{min}$ with $y(z_*)>0$. It follows further from \eqref{cond-a} at any $z>\xi_{\min}$ that
  \begin{align}
    \label{eq:y2}
    y'(z) \ge y(z)\left(\kappa(\varphi(z))y(z)-\frac1{\alpha(\varphi(z))}\right),
  \end{align}
  since $\varphi(z)<\infty$ under the assumptions of this proposition. More precisely, $\varphi\::\: (\xi_{\min},\infty)\to(0,\infty)$ is a bijection.
  
    To construct a contradiction, assume that $u\mapsto\frac{\kappa(u)}{u}\,\frac{P'(u)}{P(u)}$ were not integrable. Define $y_1:[z_*,z^*)\to\R$ by
    \begin{align*}
        y_1(z) := \left(P(\varphi(z))\left[\frac{1}{P(\varphi(z_*))y(z_*)}
        -\int_{\varphi(z_*)}^{\varphi(z)}\frac{\kappa(u)}{u}\frac{P'(u)}{P(u)}du\right]\right)^{-1},
    \end{align*}
    where the unique $z^*\in(z_*,\infty)$ is such that the expression inside the squared brackets vanishes at $z=z^*$, i.e.,
    \begin{align*}
        \int_{\varphi(z_*)}^{\varphi(z^*)}\frac{\kappa(u)}{u}\frac{P'(u)}{P(u)}du=\frac{1}{P(\varphi(z_*))y(z_*)}.
   \end{align*}
   Moreover, $\lim_{z\uparrow z^*}y_1(z)=+\infty$.
   Observe that $y_1(z_*)=y(z_*)$. Using $\phi(\varphi(z))=z+\bar C$ one can verify that, for all $z\in[z_*,z^*)$,
   \begin{align}
        \label{eq:y1}
        y_1'(z) = y_1(z)\left(\kappa(\varphi(z))y_1(z)-\frac1{\alpha(\varphi(z))}\right).
   \end{align}
 Indeed, to evaluate the derivative of $y_1$, we need to compute the derivative of $\varphi$. Recall that $\phi:(0,\infty)\to(\xi_{\min}+\bar C,\infty)$ is a bijection since $r\mapsto P'(r)/r$ is not integrable at $r\to\infty$. By definition of $\varphi$ as the generalized inverse of $\phi$, we thus obtain
\begin{equation}\label{dervarphi}
    \varphi'(z) = \frac1{\phi'(\varphi(z))} = \frac{\varphi(z)}{P'(\varphi(z))}. 
\end{equation}
With this expression, we can compute
 \begin{align*}
        y_1'(z) 
        &= -y_1(z)^2\left[P'(\varphi(z))\varphi'(z)\left[\frac{1}{P(\varphi(z_*))y(z_*)}
        -\int_{\varphi(z_*)}^{\varphi(z)}\frac{\kappa(u)}{u}\frac{P'(u)}{P(u)}du\right] -\varphi'(z)\frac{\kappa(\varphi(z))}{\varphi(z)}{P'(\varphi(z))}\right] \\
        &= -y_1(z)\frac{P'(\varphi(z))\varphi'(z)}{P(\varphi(z))} + y_1(z)^2\frac{P'(\varphi(z))\varphi'(z)}{\varphi(z)}\kappa(\varphi(z)) \\
        &= y_1(z)\left[\kappa(\varphi(z))y_1(z)-\frac1{\alpha(\varphi(z))}\right].
    \end{align*}
   
   In view of \eqref{eq:y2} and \eqref{eq:y1}, the comparison principle for solutions of scalar ODEs yields that $y$ needs to blow up somewhere in between $z_*$ and $z^*$, contradicting the definition of $y$ on $(\xi_{\min},\infty)$.
   Consequently, $\frac{\kappa(u)}{u}\,\frac{P'(u)}{P(u)}$ needs to be integrable at $u\to\infty$.
   
   Now let $y:(\xi_{\min},\infty)\to\R$ be a non-negative solution to \eqref{cond-aa}. We need to verify that $y(z_*)\le Y(z_*)$ at each $z_*>\xi_{\min}$. 
    Since $y$ is admissible, it satisfies in particular
    \begin{align}
        \label{eq:hatygood}
        y'(z) \ge y(z)\left(\kappa(\varphi(z))y(z)-\frac1{\alpha(\varphi(z))}\right),
    \end{align}
    since $\varphi(z)<\infty$ under the assumptions of this proposition.  
    
    Towards a contradiction, assume that $y(z_*)>Y(z_*)$, so that
    \begin{align*}
        K_*:=\frac{y(z_*)}{Y(z_*)}>1.
    \end{align*}
    Define (the unique) $z^*>z_*$ by
    \begin{align*}
        K_*\int_{\varphi(z_*)}^{\varphi(z^*)}\frac{\kappa(u)}{u}\frac{P'(u)}{P(u)}du
        =\int_{\varphi(z_*)}^{\infty}\frac{\kappa(u)}{u}\frac{P'(u)}{P(u)}du,
    \end{align*}
    due to the integrability of $\frac{\kappa(u)}{u}\,\frac{P'(u)}{P(u)}$ and $\hat y:(\xi_{\min},z^*)\to\setR$ by
    \begin{align*}
        \hat y(z) :=\left(P(\varphi(z))\left[
        \frac1{K_*}\int_{\varphi(z_*)}^{\infty}\frac{\kappa(u)}{u}\frac{P'(u)}{P(u)}du
        -\int_{\varphi(z_*)}^{\varphi(z)}\frac{\kappa(u)}{u}\frac{P'(u)}{P(u)}du
        \right]\right)^{-1}.
    \end{align*}
    One verifies that
    \begin{align*}
        \hat y'(z) = \hat y(z)\left(\kappa(\varphi(z))\hat y(z)-\frac1{\alpha(\varphi(z))}\right),
    \end{align*}
    and from the definition of $z^*$, it is obvious that $\hat y(z)\to\infty$ as $z\uparrow z^*$.
    Now recall \eqref{eq:hatygood}, and observe that $\hat y(z_*)=y(z_*)$ by definition. The comparison principle for scalar ODEs implies that $y(z)\ge\hat y(z)$ for all $z\in[z_*,z^*)$, hence $y$ blows up somewhere in between $z_*$ and $z^*$, contradicting that $y$ is defined on all of $(\xi_{\min},\infty)$.
\end{proof}
\begin{example}\label{xmp:exponential}
  We consider the regular non-degenerate nonlinearity $P(u)=e^u-1$. 
  Accordingly, 
  \begin{align*}
    &\alpha(r) = \frac{e^r-1}r,\quad\beta(r)=e^r,\quad 
    \kappa(r) = 1 + \frac r8 \frac{1-\frac{1-e^{-r}}r}{1-e^{-r}}
    \frac{1-\left(1+\frac8d\right)\frac{1-e^{-r}}r}{1-\left(1-\frac1d\right)\frac{1-e^{-r}}r},\\
    &\mu(r)=e^r-(d-1)\frac{e^r-1}{dr}.
  \end{align*}  
  Since $\phi:(0,\infty)\to(-\infty,\infty)$ is a bijection, Proposition \ref{prp:Pintegrable} applies.
  Since
  $$
    \lim_{u\to\infty} \frac{\kappa(u)}{u}\frac{P'(u)}{P(u)}=\frac18,
  $$
  this function is not integrable at $u\to\infty$.
  So $y\equiv0$ is the only admissible function.\qed
\end{example}
For $P'(r)/r$ not integrable at $r\to\infty$, 
Proposition \ref{prp:Pintegrable} implies that integrability of $r\mapsto \frac{\kappa(r)}{r}\frac{P'(r)}{P(r)}$ at $r\to\infty$ is necessary for the existence of non-trivial admissible functions (see Example \ref{xmp:exponential}). Proposition \ref{prp:nontrivialy} below complements this: the aforementioned integrability is also sufficient, at least if $\alpha$ and $\kappa$ have the monotonicity property \eqref{eq:lengthy}.
\begin{prop}\label{prp:nontrivialy}
    Assume $\alpha(0+)<\infty$, and that $r\mapsto P'(r)/r$ is not integrable at $r\to\infty$. Assume further that there is some $r_*\ge0$ such that
    \begin{align}\label{eq:lengthy}
        \kappa(r')\le\kappa(r)
        \quad\text{and}\quad
        \alpha(r')\le\alpha(r)
        \quad\text{for all $r\ge r_*$, and all $0<r'\le r$.}
    \end{align}
    If the function
    \begin{equation}\label{r-map}
        r\mapsto\frac{\kappa(r)}{r}\frac{P'(r)}{P(r)}
    \end{equation} 
    is integrable at $r\to\infty$, then there exists a non-trivial admissible function $y:(\xi_{\min},\infty)\to\R$.
\end{prop}
\begin{remark}\label{rem:Y-max-entropy}\
    \begin{enumerate}
    \item[(a)] Hypothesis \eqref{eq:lengthy} is obviously fullfilled with $r_*=0$ if both $\alpha$ and $\kappa$ are non-decreasing functions. In that case, $Y$ from \eqref{eq:hatYdef} itself is admissible, and is actually the pointwise largest admissible function. A situation where $r_*>0$ is needed is analyzed in Example \ref{ex:power-exp} below.
    \item[(b)] The proof provides a quite explicit construction of a non-trivial $y$ with the help of $Y$ from \eqref{eq:hatYdef}. In particular, for any sufficiently small $\eps>0$, one may choose $y$ such that $y(z)=Y(z)$ for all $z\ge \phi(r_*)-\bar C+\eps$.
    \end{enumerate}
\end{remark}
\begin{proof}[Proof of Proposition \ref{prp:nontrivialy}]
    Let $z_*:=\phi(r_*)-\bar C$. We differentiate the expression for $Y(z)$ to show that it satisfies \eqref{cond-aa} on $(z_*,\infty)$. Recalling \eqref{dervarphi} and proceeding similarly to \eqref{Y-Y'}, one obtains 
    \begin{align}\label{Y-Y'}
        Y'(z) 
        &= -Y(z)^2\left[P'(\varphi(z))\varphi'(z)\int_{-\varphi(z)}^\infty\frac{\kappa(r)}r\frac{P'(r)}{P(r)}\,dr -\varphi'(z)\frac{\kappa(\varphi(z))}{\varphi(z)}{P'(\varphi(z))}\right] \nonumber \\
        &= -Y(z)\frac{P'(\varphi(z))\varphi'(z)}{P(\varphi(z))} + Y(z)^2\frac{P'(\varphi(z))\varphi'(z)}{\varphi(z)}\kappa(\varphi(z)) \\
        &= Y(z)\left[\kappa(\varphi(z))Y(z)-\frac1{\alpha(\varphi(z))}\right]. \nonumber 
    \end{align}
    Thanks to hypothesis \eqref{eq:lengthy}, we have for every $z\ge z_*$ and $\eta\ge0$ that
    \begin{align}\label{theta-Y}
        \vartheta(z,\eta) = \sup_{0<r<\varphi(z)} \left[\kappa(r)\eta-\frac1{\alpha(r)}\right]
        = \kappa(\varphi(z))\eta-\frac1{\alpha(\varphi(z))}. 
    \end{align}
    This shows that $Y$ satisfies \eqref{cond-aa} on $(z_*,\infty)$, even with equality. 
    
    It remains to find a $C^1$-extension $y:(\xi_\text{min},\infty)\to\R$ of $Y$ that satisfies \eqref{cond-aa} on all of its domain. 
    $Y(z_*)>0$ follows from \eqref{eq:hatYdef}, since $\varphi(z_*)=r_*<\infty$ and assumption \eqref{r-map}. 
    If $Y'(z_*)\ge0$, then such an extension is given by defining $y(z):=Y(z)$ for $z\ge z_*$, and by setting 
    \[ y(z) := Y(z_*)\,\exp\left((z-z_*)\,\frac{Y'(z_*)}{Y(z_*)}\right) \]
    for $z\in(\xi_\text{min},z_*)$. On the latter range of $z$'s, we have $0<y(z)\le Y(z_*)$ and
    \[ y'(z) = y(z)\frac{Y'(z_*)}{Y(z_*)} = y(z)\vartheta\big(z_*,Y(z_*)\big), \]
    by using \eqref{Y-Y'} and \eqref{theta-Y}.
    Global $C^1$-regularity of $y$ is obvious. Validity of \eqref{cond-aa} for $z\ge z_*$ is inherited from $Y$, while for $\xi_\text{min}<z<z_*$, it follows from the fact that $\vartheta(z,y(z))\le\vartheta(z_*,Y(z_*))$ by monotonicity of $\vartheta$ in both arguments, see Lemma \ref{lem:thetaprop}. 

    If instead $Y'(z_*)<0$, then we define a different extension in order to guarantee $y(z)\le Y(z_*)$ for $\xi_\text{min}<z<z_*$ also in this case. Thanks to $Y$'s $C^1$-regularity, $z\mapsto Y'(z)/Y(z) = \vartheta(z,Y(z))$ is a continuous function, 
    and the last equatily holds for $z\ge z_*$. 
    Thus we may choose an $\eps>0$ such that $\vartheta(z,Y(z))<0$ for $z_*\le z\le z_*+\eps$. In particular, 
    \begin{align}
        \label{eq:monomono}
        Y(z_*+\eps)\le Y(z') \le Y(z) \le Y(z_*) \quad \text{for $z_*\le z'\le z\le z_*+\eps$}.
    \end{align}  
    Now define $y$ by 
    \[
    y(z):=\begin{cases}
            \exp\left(\frac1\eps\int_{z_*}^{z_*+\eps}\log Y(z')\,dz'\right) & \text{for $\xi_\text{min}<z\le z_*$}, \\
            Y(z)^{(z-z_*)/\eps}\exp\left(\frac1\eps\int_z^{z_*+\eps}\log Y(z')\,dz'\right) & \text{for $z_*<z<z_*+\eps$},\\
            Y(z) & \text{for $z\ge z_*+\eps$}.
        \end{cases}
    \]
    Notice that, by Jensen's inequality, we have for $z_*<z<z_*+\eps$ that
    \begin{align}
        \label{eq:monomo}
        y(z) \le \frac{Y(z)^{(z-z_*)/\eps}}{z_*-z+\eps}\int_z^{z_*+\eps}Y(z')^{1-(z-z_*)/\eps}\,dz' \le Y(z),
    \end{align}
    where the last inequality follows from \eqref{eq:monomono}.
    
    Continuity of $y$ is immediately verified. $y'(z)=0$ for $\xi_\text{min}<z<z_*$, and
    \begin{align}
        \label{eq:yprime}
        y'(z) = \frac{z-z_*}{\eps}\frac{Y'(z)}{Y(z)}y(z)
    \end{align}
    for $z_*<z<z_*+\eps$, showing continuity of $y$'s derivative as well. Validity of \eqref{cond-aa} for $z>z_*+\eps$ is inherited from $Y$, and for $\xi_\text{min}<z<z_*$, it follows from 
    \[ \vartheta(z,y(z)) \le \vartheta(z_*,Y(z_*)) < 0  \]
    thanks to the monotonicity of $\vartheta$, see Lemma \ref{lem:thetaprop}. Now let $z_*<z<z_*+\eps$. Recalling \eqref{eq:monomo}, that $Y'(z)<0$, and once again the monotonicity of $\vartheta$, we conclude from \eqref{eq:yprime} that
    \begin{align*}
        \frac{y'(z)}{y(z)} 
        = \frac{z-z_*}{\eps}\frac{Y'(z)}{Y(z)}
        \ge \frac{Y'(z)}{Y(z)}
        = \vartheta(z,Y(z))
        \ge \vartheta(z,y(z)),
    \end{align*}
    showing validity of \eqref{cond-aa} also on this range of $z$'s.
\end{proof}
\begin{corollary}
    Assume $\alpha(0+)<\infty$, the monotonicity property \eqref{eq:lengthy}, and that the function $\frac{\beta(r)}{\alpha(r)}= rP'(r)/P(r)\to\infty$ as $r\to\infty$. Then, a non-trivial admissible $y$ exists if and only if $r\mapsto[P'(r)/P(r)]^2$ is integrable for $r\to\infty$.
\end{corollary}
\begin{proof}
    To apply Propositions \ref{prp:Pintegrable} and \ref{prp:nontrivialy}, we first verify that $r\mapsto P'(r)/r$ is not integrable for $r\to\infty$. Since $rP'(r)/P(r)\to\infty$, we have in particular that $(\log P)'(r)\ge 2/r$ for all sufficiently large $r$, and consequently also $P(r)\ge cr^2$ for these $r$ and some $c>0$. It follows that
    \[ \frac{P'(r)}r = c\frac{rP'(r)}{P(r)}\frac{P(r)}{cr^2} \ge c\frac{rP'(r)}{P(r)} \to \infty \]
    as $r\to\infty$, so this expression is not integrable.  
  
    For $\kappa$, we obtain directly from the definition \eqref{eq:def-nu} that
    \begin{equation}\label{kappa}
     \kappa(r) = 1+ \frac18\frac{d\left(\frac{rP'(r)}{P(r)}-1\right)-8}{d\left(\frac{rP'(r)}{P(r)}-1\right)+1}\left(\frac{rP'(r)}{P(r)}-1\right), 
    \end{equation}
    which --- since $rP'(r)/P(r)\to\infty$ --- implies that 
    \[ \frac1C\frac{rP'(r)}{P(r)} \le \kappa(r) \le C\frac{rP'(r)}{P(r)} \]
    with a constant $C>1$, uniformly for large $r$.
    Multiply this inequality by $P'(r)/(rP(r))$ to obtain
    \[ \frac1C\left(\frac{P'(r)}{P(r)}\right)^2\ \le \frac{\kappa(r)}r\frac{P'(r)}{P(r)} \le C\left(\frac{P'(r)}{P(r)}\right)^2. \]
    Hence, integrability of $\frac{\kappa(r)}r\frac{P'(r)}{P(r)}$ as $r\to\infty$ is equivalent to integrablity of $[P'(r)/P(r)]^2$.
\end{proof}
\begin{example}\label{ex:power-exp}
    For parameters $m\ge1$ and $\delta>0$, consider the degenerate (if $m>1$) or regular non-degenerate (if $m=1$) exponential-type nonlinearity $P(r)=r^me^{r^\delta}$. 
    Then, $\alpha(r)=r^{m-1}e^{r^\delta}$ is strictly increasing on $[0,\infty)$ and $\kappa$ from \eqref{kappa} satisfies the monotonicity property \eqref{eq:lengthy} with $r_*^\delta=(\frac8d-m+1)/\delta.$  
    Moreover  
    \[\frac{P'(r)}{P(r)} = \frac mr + \delta r^{\delta-1}. \]
    Thus, clearly $rP'(r)/P(r)\to\infty$ as $r\to\infty$. Moreover, since
    \[ 
    \left(\frac{P'(r)}{P(r)}\right)^2 \le 2\frac{m^2}{r^2}+2\delta^2r^{2(\delta-1)},
    \]
    $r\mapsto[P'(r)/P(r)]^2$ is integrable at $r\to\infty$ if and only if $2(\delta-1)<-1$, that is $0<\delta<1/2$.
    
    In particular, there are non-trivial admissible $y$'s for $P(r)=r^me^{\sqrt[3]{r}}$ with arbitrary $m\ge1$, but none for $P(r)=r^me^r$.
\end{example}

%
%%%%%%%%%%%
\subsection{Application: power-type nonlinearities}\label{sct:power-laws}
%%%%%%%%%%%
%
We shall now apply the previously derived general results in the case of power-law nonlinearities $P(u)=u^m$ with $m>0$. These represent prototypical examples of the  nonlinear Fokker-Planck equation \eqref{GFP}. 
\begin{prop}\label{prop:genm}\
    \begin{itemize}
    \item[(a)] For $0<m<(d-1)/d$, there are no admissible entropies.
    \item[(b)] For $(d-1)/d\le m<1$, the only admissible entropy is the standard one,
        \begin{equation*}
        \mathcal H_{m,1}(u|u_\infty) = \int \left[\frac{u(x)^m-u_\infty(x)^m}{m-1}+V(x)\big(u(x)-u_\infty(x)\big)\right]\,dx. \end{equation*} 
    \item[(c)] For $m=1$, there is a continuous family $\mathcal H_{1,p}$, $1\le p\le2$ of admissible entropies corresponding to \eqref{pentropy} (and many more). Equivalently, it can be written as \eqref{e:2.12} with $\psi_p$ from \eqref{psi-p}.
    \item[(d)] For $m>1$, there is a continuous family $\mathcal H_{m,p}$, $1\le p\le2$ of admissible entropies corresponding to    
        \begin{align}\label{eq:gm}
        g_{m,p}(z) = \left[\left(\frac{m+(m-1)(z+\bar C)}{m+(m-1)\bar C} \right)^{1+\frac{p-1}{(m-1)\kappa_m}} - 1\right]\left(\frac{m}{m-1}+\bar C\right)
        \end{align}
    for all $z>\xi_{\min}=-\frac{m}{m-1}-\bar C$, with the positive constant
    \[ \kappa_m = 1+ \frac{m-1}8 \frac{d(m-1)-8}{d(m-1)+1}. \]
    For fixed $u$, the expressions $\mathcal H_{m,p}(u|u_\infty)$ are increasing in $p$. Finally, there is a  constant $A=A(m,p)>1$ such that
    \begin{align}\label{eq:orderent}
    \begin{split}
        \frac1A\intRd \Big[u(x)^{m+\frac{p-1}{\kappa_m}}+V(x)&^{1+\frac{p-1}{(m-1)\kappa_m}}u(x)\Big]\,dx-A
        \le\mathcal H_{m,p}(u|u_\infty)\\
        &\le A\intRd \Big[u(x)^{m+\frac{p-1}{\kappa_m}}+V(x)^{1+\frac{p-1}{(m-1)\kappa_m}}u(x)\Big]\,dx.
    \end{split}
    \end{align}
    \end{itemize}
\end{prop}
\begin{remark}
\
\begin{enumerate}
    \item[(a)] In the limit towards linear diffusions (i.e.\ $m\searrow1$), $g_{m,p}(z) \to e^{(p-1)z}-1$ which coincides, for $p>1$, with $g_p(z)$ from \eqref{pentropy}, up to the factor $\frac{p}{p-1}$. But this is no contradiction as entropies can always be scaled by a positive constant. Moreover, since the entropies $\mathcal H_{1,p}$ can be written explicitly, their $p$-ordering can be based on the monotonicity of $G_{1,p}$ (or $\psi_p$ from \eqref{psi-p}), and this allows for a ``sharper ordering'' than using the monotonicity of $g_{1,p}$, as done in the subsequent proof for the case $m>1$.
    \item[(b)]
    The relation \eqref{eq:orderent} shows a strict ordering among the functionals $\mathcal H_{m,p}$ for different parameters $p$: The functional is finite if and only if $u\in L^{m+\frac{p-1}{\kappa_m}}(\Rd)$ and $V^{1+\frac{p-1}{(m-1)\kappa_m}}u\in L^1(\Rd)$. Notice in particular that the two integrabilities --- of a power of $u$ and of a moment --- cannot be separated: If $u_0$ satisfies only one of these conditions, we have no conclusion. 
    A sharper lower bound and consequences on the long time asymptotics of \eqref{GFP} are discussed in Section \ref{sec-funcineq}. 
    \item[(c)] For $p=1$, \eqref{eq:gm} yields the standard entropy $g_{m,1}(z)=z$ that was already introduced in Example \ref{Ex:canon-entropy}.
    \item[(d)] For $p=2$, \eqref{eq:gm} corresponds to $f'_{m,2}(z)=Y_m(z)$ (see \eqref{Ym} below), the pointwise largest admissible function, which saturates the differential inequality \eqref{cond-aa} on all of $(\xi_{\min},\infty)$, see the proofs of Propositions \ref{prp:nontrivialy} (with $r_*=0$, $z_*=\xi_{\min}$) and \ref{prop:genm}(d).
\end{enumerate}
\end{remark}
\begin{proof}[Proof of Proposition \ref{prop:genm}]
    To begin with, observe that for $P(r)=r^m$ with $m>0$, we have
    \begin{align}\label{power-constants}
        &\phi(r)=\begin{cases}
            \frac{m}{m-1}(r^{m-1}-1) & \text{for $m\neq1$}, \nonumber\\
            \log r & \text{for $m=1$},
        \end{cases},
        \quad
        \varphi(z) = \left[\frac{m-1}m(z+\bar C)+1\right]_+^{1/(m-1)}, \\
        &\xi_{\min}= \begin{cases}
            -\infty & \text{for $m\le 1$},\\
            -\frac{m}{m-1}-\bar C & \text{for $m>1$},
        \end{cases}, \\
        &\alpha(r)=r^{m-1},\quad \beta(r)=mr^{m-1},\quad \mu(r)=\left(m-1+\frac1d\right)r^{m-1}, \nonumber\\
        &\kappa(r)\equiv\kappa_m:=1+\frac{m-1}8 \frac{d(m-1)-8}{d(m-1)+1}.\nonumber
    \end{align}
With all these formulas, the function $\vartheta(z,\eta)$ in \eqref{cond-aa} is given by
$$
\vartheta(z,\eta)=\kappa_m\eta-\frac{1}{\frac{m-1}m (z+\bar C)+1}\,.
$$

    Part (a) follows from Proposition \ref{prop4.10} since $d\mu(r)=-((d-1)-dm)r^{m-1}$ is negative for all $r>0$.

    Part (b) is a direct consequence of Proposition \ref{cor:admin1}, since $\alpha(r)=r^{-(1-m)}$ blows up as $r\downarrow0$.

    Part (c) has been exhaustively discussed in \S\ref{sct:linear} above and in \cite{AMTU}.

    It is only Part (d) that requires some work. Since $\alpha(r)=r^{m-1}$ is increasing, with $\alpha(0+)=0$, since $\kappa(r)\equiv\kappa_m$ is constant, and since $P'(r)/r=r^{m-2}$ is not integrable at $r\to\infty$, Proposition \ref{prp:nontrivialy} is applicable, with $r_*=0$. Specifically, since
    \begin{align*}
        r\mapsto \frac{\kappa(r)}{r}\frac{P'(r)}{P(r)} = m\kappa_m r^{-2}
    \end{align*}
    is integrable at $r\to\infty$, there are non-trivial admissible $y$, and the largest one is given by
    \begin{align}\label{Ym}
        Y_m(z) &= 
        \left(m\kappa_m\varphi(z)^{m}\int_{\varphi(z)}^\infty\frac{dr}{r^2}\right)^{-1}
        = \frac1{m\kappa_m\varphi(z)^{m-1}}
        = \frac1{\kappa_m\big[(m-1)(z+\bar C)+m\big]}.
    \end{align}
    By Corollary \ref{cor:pfamily}, an entire family of admissible functions is given by
    \begin{align*}
        y_{m,p}(z) = (p-1)Y_m(z), \quad \text{with $1\le p\le2$.}
    \end{align*}
    Correspondingly, there are functions $g_{m,p}$ with $y_{m,p}=(\log g_{m,p}')'$, with normalization $g_{m,p}(0)=0$. The $g_{m,p}$ are determined up to multiplicative constants. It is immediately verified that an appropriate choice of these constants yields \eqref{eq:gm}.
    
    Concerning the ordering with respect to $p$, it suffices to observe that $\partial_p g_{m,p}(z)>0$ for $z>0$, and $\partial_p g_{m,p}(z)<0$ for $z<0$. Indeed, let $x$ be in the support of $u_\infty$, which implies that $\phi(u_\infty(x))+V(x)-\bar C=0$ from \eqref{stat1a}. If $u(x)\ge u_\infty(x)$, then $\phi(u(x))+V(x)-\bar C\ge 0$ by monotonicity of $\phi$, and so
    \begin{align}\label{eq:monoGmp}
        \partial_pG_{m,p}(u,u_\infty;x) = \int_{u_\infty(x)}^{u(x)}\partial_pg_{m,p}(\phi(r)+V(x)-\bar C)\,dr \ge 0,
    \end{align}
    since $\phi(r)+V(x)-\bar C\ge0$ on the interval of integration. Likewise, if $u(x)\le u_\infty(x)$, then also $\phi(u(x))+V(x)-\bar C\le 0$, and therefore
    \begin{align*}
        \partial_pG_{m,p}(u,u_\infty;x) = -\int_{u(x)}^{u_\infty(x)}\partial_pg_{m,p}(\phi(r)+V(x)-\bar C)\,dr \ge 0,
    \end{align*}
    since now $\phi(r)+V(x)-\bar C\le0$. Now assume $u_\infty(x)=0$. Then $\phi(u_\infty(x))+V(x)-\bar C\ge0$, and $u(x)\ge u_\infty(x)$. The argument is now similar again as in \eqref{eq:monoGmp} above.

    Finally, concerning the bounds, we use the elementary fact that, for each $q\ge1$,
    \begin{align*}
        c_q(a^q+b^q)\le(a+b)^q\le C_q(a^q+b^q) \quad \text{for all $a,b\ge0$},
    \end{align*}
    with appropriate positive constants $c_q$ and $C_q$. Specifically, recalling \eqref{eq:gm},
    \begin{align*}
        g_{m,p}(\phi(r)+V(x)-\bar C) 
        = \left[\left(\frac{mr^{m-1}+(m-1)V(x)}{m+(m-1)\bar C}\right)^{1+\frac{p-1}{(m-1)\kappa_m}}-1 \right]\left(\frac{m}{m-1}+\bar C\right),
    \end{align*}
    and hence
    \begin{align*}
        c_p'\Big(r^{m-1+\frac{p-1}{\kappa_m}}+V(x)^{1+\frac{p-1}{(m-1)\kappa_m}}\Big)&\le g_{m,p}(\phi(r)+V(x)-\bar C)+\frac{m}{m-1}+\bar C\\
        &\le C_p'\Big(r^{m-1+\frac{p-1}{\kappa_m}}+V(x)^{1+\frac{p-1}{(m-1)\kappa_m}}\Big).
    \end{align*}
    Thus
    \begin{align*}
        c_p''\Big(r^{m+\frac{p-1}{\kappa_m}}+V(x)^{1+\frac{p-1}{(m-1)\kappa_m}}r\Big) &\le G_{m,p}(r,0;x)+\big(\frac{m}{m-1}+\bar C\big)r\\
        &\le C_p''\Big(r^{m+\frac{p-1}{\kappa_m}}+V(x)^{1+\frac{p-1}{(m-1)\kappa_m}}r\Big).
    \end{align*}
    Defining
    \begin{align*}
        \mathcal{K}(u) := \intRd \Big(u(x)^{m+\frac{p-1}{\kappa_m}}+V(x)^{1+\frac{p-1}{(m-1)\kappa_m}}u(x)\Big)\,dx,
    \end{align*}
    we conclude, since $G_{m,p}(u,u_\infty;x) = G_{m,p}(u,0;x)-G_{m,p}(u_\infty,0;x)$, and thanks to the fact that $u$ and $u_\infty$ have the same mass, that
    \begin{align*}
        c_p''
        \mathcal{K}(u) - C_p''\mathcal{K}(u_\infty) \le \mathcal H_{m,p}(u|u_\infty) \le C_p''\mathcal{K}(u) - c_p''\mathcal{K}(u_\infty).
    \end{align*}
    Since $\mathcal{K}(u_\infty)$ is a finite positive quantity, we arrive at \eqref{eq:orderent}.
 \end{proof}

%%%%%%%%%%%%%%%%%%%%%%%%%%%%%%%%%%%%%%%%%%%%%%%%%%%%%%%%%%%%%%%%%%%%%%%%%%%%%%%%
%%%%%%%%%%%%%%%%%%%%%%%%%%%%%%%%%%%%%%%%%%%%%%%%%%%%%%%%%%%%%%%%%%%%%%%%%%%%%%%%

\section{(Non-)Admissibility of given entropies}\label{sec:entropygivennonlin}
\setcounter{equation}{0}

As mentioned at the end of \S\ref{sec:BakEm}, we shall now change perspective and ask the following question: \emph{Given a generating function $g$, which nonlinearities $P$ can be chosen such that the corresponding entropy functional $\mathcal H_g$ is admissible?} This question is more subtle than the one asked in the previous section --- which has been: \emph{Given a nonlinearity $P$, what are the admissible $\mathcal H_g$?} --- already since the implicit dependence of the functional $\mathcal H_g$ on $P$ is less intuitive and more difficult to analyze than $\mathcal H_g$'s dependence on $g$ for fixed $P$.

In the following, we assume that an entropy generating function  $g\in C^3((\ximin,\infty))$ is given. Recall the definition of $f=\log g'$, and define further the \emph{entropy curve} $\xi\mapsto(f'(\xi),f''(\xi))$ for $\xi>\ximin$. The goal is to determine (mainly necessary) criteria on $P$, or rather the auxiliary functions $\alpha(u)=P(u)/u$ and $\beta(u)=P'(u)$, such that the condition in Definition \ref{def:adm-entr} is satisfied. From Lemma \ref{lem-char} we directly see that the standard entropy $g(\xi)=\xi$, with trivial entropy curve $(f',f'')\equiv(0,0)$, yields the weakest restriction, namely just $(1-d)\alpha +d\beta\ge0$, and hence the largest set of admissible nonlinearities. On the other hand, note that already the given value $\ximin\in[-\infty,0)$ imposes a restriction on $P$, see Remark \ref{xi-range}: namely, $\phi$ derived from $P$ by means of \eqref{eq:P2phi} has to satisfy $\phi(0+)=\ximin+\bar C$, and so $\ximin=-\infty$ implies that $P$ is non-degenerate.

\subsection{Towards a geometric picture}\label{sec:geompic} 
We start by reformulating the remainder condition \eqref{rem-cond2c}, equivalent to Definition \ref{def:adm-entr} of admissibility, in the following way: for any $r>0$ and $z\in [\phi(r)-\bar C,\infty)$, we have that $d\beta(r)\ge(d-1)\alpha(r)$, and that --- recall $\mu(r)=\beta(r)-\frac{d-1}{d}\alpha(r)$ ---
\begin{align*}
    T(\alpha(r),\beta(r);f'(z),f''(z))=Z\big(\alpha(r),\mu(r),f'(z),f''(z)\big)\le0,
\end{align*}
with the quadratic polynomial (in $\alpha$ and $\beta$)
\begin{align}\label{quadric-cond}
  T(\alpha,\beta;a,b):=& \nonumber
  \left[\left(9-16\frac{d-1}{d}\right)a^2+8\frac{d-1}{d}b\right]\alpha^2-2\left[\left(1-4\frac{d-1}{d}\right)a^2+4b\right]\alpha\beta +a^2\beta^2 \\
  &+8\frac{d-1}{d}a\,\alpha - 8a\,\beta\,,
\end{align}
and the parameters $a,\,b\in\R$.
Note that we perform a change of notation from $Z$ to $T$ because we prefer to discuss the shape of the admissible regions in terms of the more natural parameters $(\alpha,\beta)$ instead of $(\alpha,\mu)$.

Towards the geometric interpretation of that condition, define for given $(a,b)\in\Rnn\times\R$ the sets
\begin{align}
    \label{def-Q}
    Q_{(a,b)}&:=\{(\alpha,\beta)\in \R^2 \text{ such that } T(\alpha,\beta;a,b)\leq 0\} \\
    \nonumber
    \mathcal{Q}_{(a,b)} &:= Q_{(a,b)} \cap \{(\alpha,\beta)\in\R^2 \text{ such that } \alpha\ge0 \text{ and } (d-1)\alpha\le d\beta \}\ .
\end{align}
We remark that the sets $Q_{(a,b)}$ and $\mathcal{Q}_{(a,b)}$ are, w.r.t.\ changing the parameters $a$ and $b$, upper semicontinuous in the sense of set inclusion. This follows from the continuity of the inequality in \eqref{def-Q} w.r.t.\ the parameters $a$ and $b$.

The remainder condition can now be restated as follows.
\begin{corollary}\label{corentgen}
    A nonlinearity $P$ with associated functions $(\alpha,\beta)$ is admissible for a generator $g$ with corresponding entropy curve $(f',f'')$ if and only if 
    \begin{equation}\label{qg}
        (\alpha(r),\beta(r))\in\mathcal{Q}_{(f'(\xi),f''(\xi))} \qquad \text{for all $\xi>\ximin$ and all $0<r\le\varphi(\xi)$},
    \end{equation}
    with $\varphi$ defined in \eqref{def-varphi}.
\end{corollary}

The polynomial $T$ can be rewritten for $a\neq 0$ as
\begin{align*}  
  T(\alpha,\beta;a,b)
  = a^2\begin{pmatrix} \alpha\\ \beta \end{pmatrix}^T\Omega\Big(\frac{b}{a^2}\Big)\begin{pmatrix} \alpha\\ \beta \end{pmatrix}
  + a\omega^T\begin{pmatrix} \alpha\\ \beta \end{pmatrix}
\end{align*}
where
\begin{align*}
  \Omega(s) =
  \begin{pmatrix}
    9-16\frac{d-1}d+8\frac{d-1}ds & 4\frac{d-1}d-1-4s \\[2mm]
    4\frac{d-1}d-1-4s & 1
  \end{pmatrix},
                        \quad
                        \omega = 
                   \begin{pmatrix}
                     8\frac{d-1}d \\[2mm] -8
                   \end{pmatrix} \ .
\end{align*}
For later purposes we also define the matrix $\Xi(a,b):=a^2 \Omega\big(b/a^2\big)$  for all $a,b$. Notice that, thanks to the scaling property 
\begin{align*}
  T(\tau^{-1}\alpha,\tau^{-1}\beta;\tau a,\tau^2b)=T(\alpha,\beta;a,b)
  \quad \text{for all $\tau>0$},
\end{align*}
one has
\begin{align}
  \label{eq:Qscaling}
  Q_{(a,b)} = \tau Q_{(\tau a,\tau^2b)}
  \quad \text{and} \quad
  \mathcal Q_{(a,b)} = \tau\mathcal Q_{(\tau a,\tau^2b)}
  \quad \text{for all $\tau>0$}.
\end{align}
\medskip

Next we shall discuss the dimensional dependence of the sets $\mathcal Q_{(a,b)}$. 
For $d\beta>(d-1)\alpha$, the polynomial $T$ can be written as
    \begin{align*}
        T(\alpha,\beta;a,b) 
        = 8\left(\beta-\frac{d-1}{d}\alpha \right)\big[\kappa\alpha a^2-a-\alpha b\big].
    \end{align*}
    For fixed values of $\alpha>0$ and $\beta\ge0$, the corresponding expression
    \[
        \kappa = \frac{\beta+7\alpha}{8\alpha} -\frac{9}{8}\frac{\beta-\alpha}{d(\beta-\alpha)+\alpha}
    \]
    is non-decreasing in $d$ (both for $\beta-\alpha$ positive and negative, as long as $d\beta>(d-1)\alpha$). Hence, since the pre-factor $\beta-\frac{d-1}{d}\alpha$ is non-negative and $\alpha>0$, the condition $T\le0$ becomes more restrictive in higher dimensions; only in the linear case $P(r)=Dr$ where $\alpha=\beta\equiv D$, $\kappa$ is independent of the dimension. 
So, the family of sets $\mathcal Q_{(a,b)}$ is decreasing in dimension $d$. Its limit for $d\to\infty$ is given by 
\begin{align*}
    \mathcal{Q}^\infty_{(a,b)} &:= Q^\infty_{(a,b)} \cap \{(\alpha,\beta)\in\R^2 \text{ such that } \alpha>0 \text{ and } (d-1)\alpha\le d\beta \}\quad \text{with}\\
    Q^\infty_{(a,b)} &:= \{ (\alpha,\beta)\in\R^2\ \text{such that}\ T^\infty(\alpha,\beta,a,b)\le 0\},
\end{align*}
and
\begin{align*}    
    T^\infty(\alpha,\beta;a,b) 
    &= a^2(\beta-\alpha)\Big(\beta-\frac8a-\left(\frac{8b}{a^2}-7\right)\alpha\Big).
\end{align*}
Note that $\mathcal Q^\infty_{(a,b)}$ is actually the intersection of the decreasing family of sets $\mathcal Q_{(a,b)}$ as $d\to\infty$.

For given $a>0$ and $b$, the set $Q^\infty_{(a,b)}$ is the wedge in between the two lines
\begin{align}
    \label{eq:dimfreelines}
    \beta = \alpha \quad \text{and} \quad \beta = \frac8a+\Big(\frac{8b}{a^2}-7\Big)\alpha,
\end{align}
and the restricted set $\mathcal Q^\infty_{(a,b)}$ in the first quadrant is never empty, and is easily characterized as
\begin{itemize}
    \item the triangle formed by the origin and the two points $(0,8/a)$ and $(\frac a{a^2-b},\frac a{a^2-b})$ if $b<a^2$;
    \item an infinite strip between the diagonal and its vertical translate by $8/a$ if $b=a^2$;
    \item the (unbounded) intersection of the space between the two lines given in \eqref{eq:dimfreelines} with the first quadrant if $b>a^2$.
\end{itemize}

From now on, we will only be interested in pairs $(a,b)$ on the entropy curve. 
In the first part of this section (up to \S\ref{sec:admiss-sets}) we shall study the geometric properties of the sets $Q_{(a,b)}$ and $\Q_{(a,b)}$ for one single point $(a,b)$ on the entropy curve. 
\S\ref{sec:7.6}-\S\ref{sec:admiss-sets-for alllin} are devoted to find conditions on admissible nonlinearities $P(u)$ for a given entropy curve (i.e., $g$ and $\xi_{\min}$ given) or a family of them. We remind the reader that a nonlinearity $P(u)$ with $\phi(0+)=\xi_{\min}+\bar C$ must be admissible for a certain interval $u\in(0,\varphi(\xi))$ for \emph{every point}  $(f'(\xi),f''(\xi))$, $\xi\in(\xi_{\min},\infty)$ on the entropy curve, as discussed in Corollary~\ref{corentgen}. The rest of the section is on stability questions related to families of entropies.

%%%%%%%%%%%%

\subsection{A motivating example}

To illustrate the effectiveness of this angle of attack in the determination of admissible nonlinearities for a given entropy, we consider a specific example.
\begin{example} \label{entropydeg}
  Let us consider the entropy generator $g:(\xi_{\min},\infty)\to\R$ with $\xi_{\min}\in(-\infty,0)$,  
\begin{equation}\label{entropy-ex-power}
  g'(\xi)=(\xi-\xi_{\min})^q\quad \mbox{for some} \quad q>0, 
\end{equation}
  and accordingly
  \begin{align*}
    f'(\xi) = \frac q{\xi-\xi_{\min}},
    \quad
    f''(\xi) = -\frac q{(\xi-\xi_{\min})^2}.
  \end{align*}
Since $\xi_{\min}>-\infty$, the only possible admissible nonlinearities are degenerate.  
  The quotient
  \begin{align*}
    \frac{f''(\xi)}{f'(\xi)^2} = -\frac1q
  \end{align*}
  is independent of $\xi$,
  and in particular $(\xi-\xi_{\min}) f'(\xi)=q$ and $(\xi-\xi_{\min})^2f''(\xi)=-q$.
  By the scaling property \eqref{eq:Qscaling},
  $\mathcal Q_{(f'(\xi),f''(\xi))}=(\xi-\xi_{\min})\mathcal Q_{(q,-q)}$, i.e.,
  each $\mathcal Q_{(f'(\xi),f''(\xi))}$ is just a dilation of the set $\mathcal Q_{(q,-q)}$.

Due to Subsection \ref{sec:admiss-sets} and Lemma \ref{Q-prop2} below, $\Q_{(q,-q)}$ is bounded, convex, and includes the origin. 
Hence, and due to the scaling property $\mathcal Q_{(f'(\xi),f''(\xi))}$, $\xi>\xi_{\min}$ is a nested family of sets, increasing with $\xi$. Thus, and due to the monotonicity of $\varphi$, condition \eqref{qg} can be reformulated in this example as
\begin{equation}\label{qg3}
 \big(\alpha(\varphi(\xi)),\beta(\varphi(\xi))\big)\in(\xi-\xi_{\min})\mathcal Q_{(q,-q)} \qquad \mbox{for all } \xi>\xi_{\min}\,. 
\end{equation}
Using $\tilde \xi:=\xi-\xi_{\min}$ and the definition of $\varphi$ we finally rewrite it as
\begin{equation}\label{qg4}
 \frac{1}{\tilde\xi}\left(\alpha\big(\overline{\phi}^{-1}(\tilde\xi+\phi(0+))\big),\beta\big(\overline{\phi}^{-1}(\tilde\xi+\phi(0+))\big)\right)\in\mathcal Q_{(q,-q)} \qquad \mbox{for all } \tilde\xi>0\,. 
\end{equation}
For a nonlinearity to be checked for its admissibility, it is straightforward to evaluate the l.h.s.\ of \eqref{qg4}. Moreover, since $\varphi(\xi_{\min})=0$ and $\alpha(0)=\beta(0)=0$ for degenerate diffusion equations, \eqref{qg3} is trivially satisfied in the limit $\xi\searrow\xi_{\min}$, reading  $(0,0)\in 0\cdot\Q_{(q,-q)}$. This shows that the  $\xi_{\min}$-dependence of \eqref{entropy-ex-power} is consistent with the definition of $\xi_{\min}$ in an (admissible) nonlinearity.\\

As an application let us check the power-law nonlinearities  $P(u)=\tau u^m$, with some $\tau>0$ and $m>1$ on their admissibility for the entropy \eqref{entropy-ex-power}. Using \eqref{power-constants} we find $\alpha(\varphi(\xi))=\frac{m-1}{m}\tilde\xi$, $\beta(\varphi(\xi))=(m-1)\tilde\xi$, and hence \eqref{qg4} simplifies to the admissibility condition 
$$
  \Big(\frac{m-1}{m},\,m-1\Big) \in\mathcal Q_{(q,-q)}\,.
$$
With the identification $q=\frac{p-1}{(m-1)\kappa_m}$ and up a multiplicative scaling factor, the entropies \eqref{entropy-ex-power} coincide with the admissible entropies $g_{m,p}$, $p\in[1,2]$ from \eqref{eq:gm} for power-type nonlinearities.
\qed
\end{example}
Intriguing as the example above might be --- it also very clearly indicates that the determination of \emph{all} admissible nonlinearities for a given entropy $\mathcal H_g$, with $g\in C^3(\xi_{\min},\infty)$, is a daunting task. This is already due to the intricate loop between $\xi_{\min}$ and $P(u)$ in condition \eqref{qg}. Even when fixing a consistent $\xi_{\min}$ is sorted as in Example \ref{entropydeg}, it is very difficult to find all nonlinearities verifying condition \eqref{qg4}. However, if the right hand side of the condition \eqref{qg} does not depend on $\xi$, i.e. $\mathcal{Q}_{(f'(\xi),f''(\xi))}=\mathcal{Q}$, then we can easily check that the condition \eqref{qg} simplifies to
\begin{equation}\label{qg2}
 (\alpha(r),\beta(r))\in\mathcal{Q} \qquad \mbox{for all } r>0  .
\end{equation}
Indeed, it suffices to take the limit $\xi\to\infty$ in condition \eqref{qg} since $\varphi(\xi)\to\infty$ as $\xi\to\infty$. This simplified condition will be useful to answer several interesting questions in the next subsections.

\subsection{Classification of the sets $Q_{(a,b)}$.}

The sets $Q_{(a,b)}$, defined in the $\alpha-\beta$-plane in \eqref{def-Q}, have the following properties.

\begin{lemma}\label{lem:geometry0}
  Assume $a> 0$. Then
  $Q_{(a,b)}$ intersects the McCann line $\mathcal L:=\{(d-1)\alpha=d\beta\}$ precisely at the origin, and $\mathcal L$ is tangent to $\partial Q_{(a,b)}$.
\end{lemma}
\begin{proof}
  Since $\mathcal L=\{t(d,d-1)|t\in\R\}$,
  and since
  \begin{align*}
    T\big(td,t(d-1);a,b)
    = a^2t^2
    \begin{pmatrix} d \\ d-1 \end{pmatrix}^T
    \Omega(b/a^2)
    \begin{pmatrix} d \\ d-1 \end{pmatrix}
    +
    at\omega^T \begin{pmatrix} d \\ d-1 \end{pmatrix}
    = 9a^2t^2
  \end{align*}
  is positive except for $t=0$,
  it follows that $\mathcal L\cap Q_{(a,b)}=\{(0,0)^T\}$.
\end{proof}
\begin{lemma}\label{lem-elip}
For $a> 0$, the closed set $Q_{(a,b)}$ is
\begin{enumerate}
\item[(a)] the bounded region enclosed by an ellipse
  if and only if $\frac{2d-4}{2d}<\frac b{a^2}<\frac{2d-1}{2d}$;
\item[(b)] the unbounded convex region enclosed by a parabola 
  if and only if  $\frac b{a^2}=\frac{2d-4}{2d}$ or $\frac b{a^2}=\frac{2d-1}{2d}$;
\item[(c)] the union of the two convex regions enclosed by either of the two branches of a hyperbola
  if and only if $\frac b{a^2}<\frac{2d-4}{2d}$ or $\frac b{a^2}>\frac{2d-1}{2d}$.
\end{enumerate}
\end{lemma}
\begin{proof}
  The fact that the boundary $\partial Q_{(a,b)}$ is 
  an ellipse, a parabola, or the two branches of a hyperbola, respectively,
  follows straightforwardly from the form of the matrix $\Omega$ in the definition of $T$.
  Specifically, one uses that $\Omega(s)_{22}=1$, 
  and that the quadratic polynomial
  \begin{equation*}
    \det \Omega(s) = -16\left(s-\frac{d-2}{d}\right)\left(s-\frac{2d-1}{2d}\right)
  \end{equation*}
  is positive if and only if $\frac{2d-4}{2d}<s<\frac{2d-1}{2d}$.
  
  That $Q_{(a,b)}$ is in each case the described enclosed region 
  follows from Lemma \ref{lem:geometry0} above,
  particulary from the fact that there is a line, namely $\mathcal L$, that intersects $Q_{(a,b)}$ only in one point.
\end{proof}

\begin{remark}
    Notice that the parameters $(a,b)$ are in fact not independent but determined by the entropy function $f(\xi)$, i.e., $a=f'(\xi)$ and $b=f''(\xi)$.
\end{remark}

\subsection{Properties of the quadrics $\partial Q_{(a,b)}$ and their ``interiors'' $Q_{(a,b)}$.}
The center of the ellipses and hyperbolas is given by
\begin{equation*}
\alpha_c = \frac{4\left(4b +\frac{5-4d}{d}a^2\right)}{a^3\det\Omega(b/a^2)}
\end{equation*}
and
\[
    \beta_c = -\frac{4\left(4\frac{1-d}{d} b +\frac{4d-9-\frac{4}d}{d}a^2\right)}{a^3\det\Omega(b/a^2)} \,.
\]
Let $\theta$ be the angle of the major axis of the hyperbola/ellipse with the positive part of the $\alpha$-axis. Let us define the parameter 
\begin{equation}\label{defk}
k:=-\frac{b}{a^2}\,;
\end{equation}
then we can express this angle in terms of $k$ through
\begin{equation}\label{theta}
\tan 2\theta = \left\{ \begin{array}{cc}
                                 -\frac14 + k & \mbox{for } d=1\\[2mm]
                                 \frac{d}{1-d} + \frac{d^2-d-4}{4(1-d)(2-d+k(1-d))} & \mbox{for } d\geq 2 
                              \end{array}\right. .
\end{equation}
Observe that the angle is an increasing function of $k$ for dimensions 1 and 2 while decreasing for $d\geq 3$. Formula \eqref{theta} also gives the angle between the axis of symmetry of a parabola with the positive part of the $\alpha$-axis.

Special points on $\partial Q_{(a,b)}$: Notice that the origin always belongs to the boundary of the set $Q_{(a,b)}$. The slope of the tangent line at the origin is given by
$\frac{d-1}{d}$ since the tangent line $(1-d)\alpha+d\beta=0$ is determined from the McCann condition. We also observe that the equation defining $Q_{(a,b)}$ evaluated on this tangent line satisfies
$$
T\left(\alpha, \frac{d-1}{d}\alpha; a,b\right)=9a^2\alpha^2\left(\frac{2d-1}{d}\right)^2\,.
$$
\begin{itemize}
\item 
If $a> 0$ then $T\left(\alpha, \frac{d-1}{d}\alpha; a,b\right)>0$ and we conclude that $Q_{(a,b)}$ is ``inside'' the quadric. Moreover, the McCann line intersects $Q_{(a,b)}$ only at the origin. Therefore, in the case of hyperbolas, it lies outside and it separates the two connected components of $Q_{(a,b)}$. Moreover, due to Lemma \ref{lem-char}, the only admissible connected component is the one above the McCann line.

\item If $a=0$ then
\begin{equation}\label{a0-case}
T\left(\alpha, \beta; 0,b\right)=-8b\alpha\left(\beta-\frac{d-1}{d}\alpha\right)\,,
\end{equation}
from which we read that $Q_{(0,b)}$ is a double-wedge (degenerate hyperbola) for $b\neq 0$.
\end{itemize}

We can also find another intersection point of $\partial Q_{(a,b)}$ with the $\beta$-axis. It is given by $\beta_0=\frac{8}{a}$ whenever $a> 0$. Moreover, we can find the slope of the tangent line to the quadric at this point; it is given by
$$
s_0=-7+\frac{9}{d}+\frac{8b}{a^2}=-7+\frac{9}{d}-8k\,.
$$

For later reference we also note that the point $(1,1)\in \Q_{(a,b)}$ if and only if $b\ge a(a-1)$. This inequality is exactly condition \eqref{f-inequ}, which was derived for linear diffusion equations.

\subsection{Properties of the admissible sets $\mathcal{Q}_{(a,b)}$.}\label{sec:admiss-sets}
We shall divide this discussion into two cases, depending on the value of $a=f'$. This is motivated by Lemma \ref{f-solutions}, where it was shown that $f'=0$ and $f'=1$ are particular cases.\\

\noindent
\underline{Case $a=0$:} From \eqref{a0-case} we see that, for $b\ge0$, $\mathcal{Q}_{(a,b)}$ consists of the closed wedge between the McCann line and the positive $\beta$-axis. For $b<0$ it consists only of the McCann ray with $\alpha\ge 0$.\\

\noindent
\underline{Case $a>0$:} Using Lemma \ref{lem-elip} together with the fact that the McCann line separates the two branches of the hyperbola and that $\beta_0>0$, then the relevant part of $Q_{(a,b)}$ coincides with the closure of the interior of the ellipse, parabola, or the branch of the hyperbola lying above the McCann line. Moreover, the McCann line is tangent at the origin to the boundary $\partial Q_{(a,b)}$. 

\begin{figure}[ht!]
\begin{center}
\includegraphics[width=10cm]{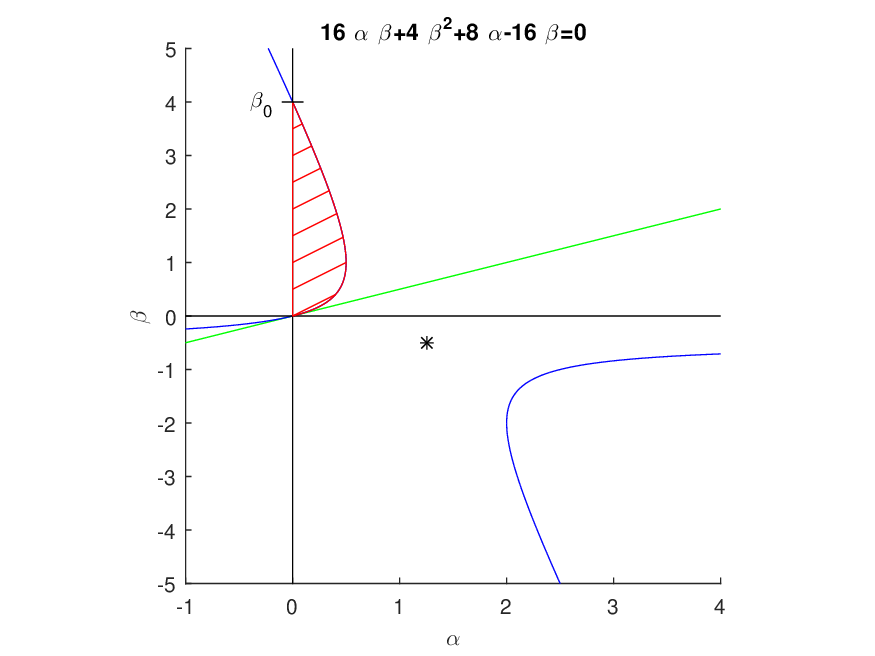}
\end{center}
%\vspace{-0.3cm}
\caption{{\footnotesize This plot visualizes the arguments in the proof of Lemma \ref{Q-prop2}. Plotted are the McCann line for $d=2$ (in green) and the hyperbola $Q_{(2,-1)}$. The admissible set $\mathcal{Q}_{(2,-1)}$ consists of the closure of the interior of the upper  hyperbola branch restricted to the first quadrant (in red). The center of the hyperbola is marked with a star. [colors only online]}}
\end{figure}

Therefore, the admissible set $\mathcal{Q}_{(a,b)}$ coincides with the intersection of these quadric-interiors with the wedge between the McCann line and the positive $\beta$-axis. As a consequence, $\mathcal{Q}_{(a,b)}$ is convex for all quadric types and the origin lies at its boundary. The boundedness (or not) of the set is obviously only relevant in the case of the parabolas and hyperbolas:

\begin{lemma}\label{Q-prop2}
Let $a>0$. The set $\mathcal{Q}_{(a,b)}$ is unbounded if and only if $b\geq \frac{2d-1}{2d} a^2$.
\end{lemma}
\begin{proof}
In order to distinguish the boundedness we use the slopes of the tangent lines at the origin and the point $(0,\beta_0)$. In fact, hyperbolas and parabolas will be unbounded if and only if the wedge between the tangent lines at the origin, given by $\mathcal L$, and at the point $(0,\beta_0)$ opens to the right, that is $s_0>\frac{d-1}d$. This gives the desired result taking into account Lemma \ref{lem-elip}.
\end{proof}

\begin{remark}\label{rm:pentroparam}
Recall from Example \ref{p-entropy} that each $p$-entropy curve $f_p(\xi)$, $1\le p\le 2$, corresponds to the single point $(p-1,0)$. Therefore, the admissible set of nonlinearities for the $p$-entropy is determined by the set $\mathcal{Q}_{(p-1,0)}$, independently of $\xi$. Let us point out that, for $p>1$ in $d=1$, this is a truncated ellipse, in $d=2$ it is a bounded truncated parabola, and in $d\ge3$ it is a bounded truncated branch of a hyperbola. Observe that Corollary \ref{cor:boundedsets} below implies that the nonlinearities
$P(u)=u^m$, $m\neq 1$, do not admit the $p$-entropies of the linear diffusions with $1< p\le 2$.

For $p=1$ the admissible set $\mathcal{Q}_{(0,0)}$ is the closed wedge between the McCann line and the positive $\beta$-axis.
\end{remark}

As a summary, we give a sketch of the sets $Q_{(a,b)}$ and admissible sets $\mathcal{Q}_{(a,b)}$ in Figure \ref{fig:sketchsets}. We remind the reader that the plotted curves fit the one dimensional values. But in other dimensions the lower red parabola changes its shape being the $a$-axis for $d=2$, and convex for $d\geq 3$. In any case, the structure of sketch of the different cases does not change depending on the dimension.

Notice that in this section we were just interested in the properties of the admissible sets for one single point on an entropy curve. If we insist in discussing admissible nonlinearities for entropies satisfying condition \eqref{f-inequ}, such as entropies for linear and non-degenerate diffusions, then our relevant parameters $(a,b)=(f'(\xi),f''(\xi))$ have to be on or above the blue dotted parabola, $b=a(a-1)$, in Figure \ref{fig:sketchsets}. Moreover, the set of relevant values correspond to $0\leq a<1$ or $(a,b)=(1,0)$ due to Lemma \ref{f-solutions}-(a).

\begin{figure}[htbp]
\begin{center}
\includegraphics[width=10cm]{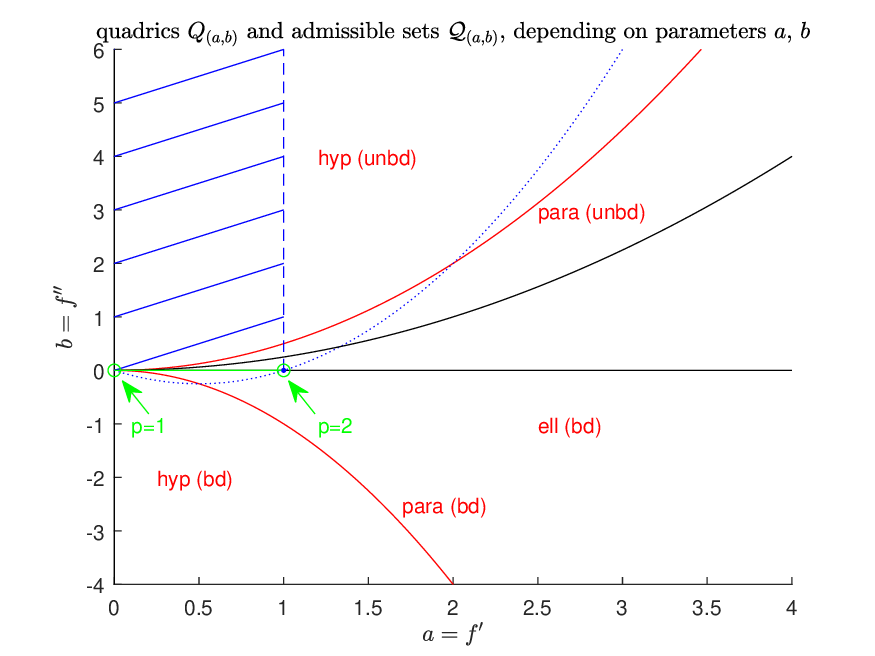}
\end{center}
%\vspace{-0.3cm}
\caption{\label{fig:sketchsets} {\footnotesize This plot shows the different types of quadrics $\partial Q_{(a,b)}$ in the parameter plane $(a,b)\in\R^+_0\times\R$, for $d=1$. The regions of ellipses, parabolas, and hyperbolas are separated by the two solid, red parabolas and marked with the abbreviations \emph{ell}, \emph{para}, and \emph{hyp}. Moreover, the un/boundedness of the sets $\Q_{(a,b)}$ is marked by \emph{bd}, \emph{unbd}.\\
The closure of the interior of the dotted, blue parabola $b=a(a-1)$ corresponds to the entropy condition \eqref{f-inequ}. Moreover, its relevant subset with $0\le a<1$ (cf.\ Lemma \ref{f-solutions} for linear diffusion and Remark \ref{remlin}-(d) for nonlinear non-degenerate diffusions) is shaded in blue. \\
Each $p$-entropy (of linear diffusion equations) with $1\le p\le2$ corresponds to the single parameter point $(a,b)=(p-1,0)$ on the non-negative $a$-axis, plotted in green (cf.\ Example \ref{p-entropy}).\\
Finally, the black parabola $b=a^2/4$ illustrates the indices in Lemma \ref{Qldecre} for $l=\tfrac14$. [colors only online]}}
\end{figure}

\subsection{First implications on the nonlinearity}
\label{sec:7.6}
We can already obtain information on the behavior of the nonlinearity at the origin and at infinity from condition \eqref{qg}. The cases (a), (b), and (c) from Proposition \ref{prop:behaviorinfty} and the classification of diffusions in \S\ref{sec:5.1} imply the following important consequence.

\begin{corollary}\label{cor:boundedsets}
Given any entropy $g(\xi)$ such that 
$$
\tilde{\mathcal{Q}}:=\bigcup_{\xi>\xi_{\min}} \mathcal{Q}_{(f'(\xi),f''(\xi))}
$$
is bounded, then the only admissible diffusions are covered by the following two cases:
\begin{itemize}
    \item $P$ is regular non-degenerate; or
    \item $P$ is degenerate and has either linear, sublinear, or saturating behavior at infinity. 
\end{itemize}
\end{corollary}

Corollary \ref{cor:boundedsets} again follows by taking the limit $\xi\to\infty$ in condition \eqref{qg} since $\varphi(\xi)\to\infty$ as $\xi\to\infty$ to deduce
\begin{equation*}
 (\alpha(r),\beta(r))\in\tilde{\mathcal{Q}} \qquad \mbox{for all } r>0  .
\end{equation*}
Case (d) from Proposition \ref{prop:behaviorinfty} is only possible for entropies $g(\xi)$ with unbounded $\tilde{\mathcal{Q}}$. For instance, $P(u)\simeq u^m$, $m>1$, or $P(u)=e^u-1$ as $u\to \infty$ are included here.

We finally give a general result for entropies satisfying natural bounds on the derivatives of $f$.
\begin{theorem}\label{fconvex}
  Let $g\in C^3(\R)$ be such that the entropy function $f$,  
  defined by \eqref{entrofgen} satisfies: $f$ is convex and $f'\ge0$ is bounded from above. Then all admissible diffusions satisfy 
  $$
  (\alpha(r),\beta(r))\in\mathcal{Q}_{(A,0)} \qquad \mbox{for all } r>0  \,,
$$
with $A:=\sup_{\xi>\xi_{\min}} f'(\xi)\geq 0$. \\
\end{theorem}
\begin{proof}
Due to the assumptions on the entropy function $f$, i.e.\ $f'$ is bounded and monotonously non-decreasing, there exists a sequence $\xi_n\to \infty$ such that $(f'(\xi_n),f''(\xi_n))\to (A,0)$ as $n\to \infty$. Then, condition \eqref{qg} implies
\begin{equation*}
 (\alpha(r),\beta(r))\in\mathcal{Q}_{(f'(\xi_n),f''(\xi_n))} \qquad \mbox{for all } 0<r<\varphi(\xi_n)  . 
\end{equation*}
By taking $n\to\infty$ the conclusion of the theorem follows since $\varphi(\xi)\to\infty$ as $\xi\to\infty$ and by using the upper semicontinuity of the sets $Q_{(a,b)}$ w.r.t.\  the parameters $a$ and $b$.
\end{proof}

We can draw similar consequences to Corollary \ref{cor:boundedsets} on the set of possible admissible nonlinearities in case that $\mathcal{Q}_{(A,0)}$ is bounded. We show in Lemma \ref{Q-prop2} this is the case for $A>0$.

%%%%%%

Now, let us concentrate on some properties of these sets depending on the parameters $(a,b)$. Let us consider the one parameter family of parabolas $b=la^2$, $l\in\R$, which include the separation lines in Lemma \ref{lem-elip} and the red lines in Figure \ref{fig:sketchsets}. Now we define the sets
$$
\mathcal{Q}_{a}^l :=\mathcal{Q}_{(a, l a^2)}\,
$$
for $a\ge 0$, $l\in\R$. The indices of the sets $\mathcal{Q}_{a}^l$ for $l$ fixed are parabolas corresponding to the red and the black curves in Figure \ref{fig:sketchsets}.

\begin{lemma}\label{Qldecre}
For any fixed $l\in\R$, $\mathcal{Q}_{a}^l$ is a decreasing family of sets with respect to increasing $a$.
\end{lemma}
\begin{proof}
Notice that from \eqref{quadric-cond} we rewrite $T\leq 0$ as
\begin{align}\label{quadric-cond2bis}
  \frac1{a^2}  T(\alpha,\beta,a,l a^2)=&\left[\left(1+8\frac{2-d}{d}\right) -\frac{8(1-d)}{d} l\right]\alpha^2 - \left[\left(2+8\frac{1-d}{d}\right)+8l\right]  \alpha\beta \nonumber\\
  &+ \beta^2-\frac8{a}\left(\beta+\frac{1-d}{d}\alpha\right)\le0\,.
\end{align}
The claim is now a consequence of \eqref{quadric-cond2bis} together with the McCann's condition $\beta+\frac{1-d}{d}\alpha\ge 0$.
\end{proof}

Notice that by varying $a$, the type of quadric that defines $\mathcal{Q}_{a}^l$ does not change due to \eqref{quadric-cond2bis}.

\subsection{Admissible nonlinearities for the $p$-entropies.}\label{sec:admiss-sets-p}

Let us recall that we introduced in \eqref{psi-p} the term ``$p$-entropies'' for the linear equation \eqref{linFP} with $D=1$. They could be scaled for $D\ne1$ according to \eqref{entrof}. Their generating function $g_p(\xi)$ is defined in \eqref{pentropy} and $f'_p(\xi)\equiv p-1$ is given in Example \ref{p-entropy}. In this section we shall generalize this notion to nonlinear equations:

\begin{definition}\label{def:p-ent}
For a nonlinear diffusion equation \eqref{GFP}, the entropy functional $\H_g(u|u_\infty)$ (from Definitions \ref{genrelentropy} or \ref{genrelentropy2}) with the generating function 
\[
g_p(\xi)=\frac{p\big(e^{(p-1)\xi}-1\big)}{p-1},
\]
implying $f'_p(\xi)\equiv p-1$ is called \emph{$p$-entropy}.
\end{definition}

Let us note that these $p$-entropies do not coincide (except for $p=1$) with the family of entropies $\H_{m,p}$ from Proposition \ref{prop:genm} for $P(u)=u^m$, $m>1$. \\

Since Definition \ref{def:p-ent} depends via $\xi$ also on the nonlinear function $\phi(u)$, $p$-entropies, for $p\in [1,2]$ fixed, share the same generating function $g_p$, but the actual form of the functional $\H_{g_p}$ will still depend on the considered nonlinearity $P$. 
But let us now consider the following ``perturbation'' of the linear diffusion equation: Let the nonlinearity $P_2(u)$ coincide with $P_1(u):\equiv u$ on some interval $[0,u_0]$ (as it is the case in Example \ref{exnon2} with $\tau_o=1$). Moreover, assume that the mass of the initial condition, $M>0$ is so small that $u_{\infty,1}(x)\in[0,u_0],\,x\in\R^d$. Then, $u_{\infty,1}=u_{\infty,2}$, and $\H_{g,1}(u|u_{\infty,1})=\H_{g,2}(u|u_{\infty,2})$ for all functions with $u(x)\in[0,u_0],\,x\in\R^d$ (see Definition \ref{genrelentropy}). 
Concerning perturbations of the linear function $P(u)$ we shall give in \S \ref{Sec:stability} a more general stability result of relative entropies $\H_g[P]$ w.r.t.\ $P$ and for $g$ fixed. This motivates to consider the $p$-entropies from the linear case also for close-by nonlinearities. Let us write out explicitly the relative entropy for $p=2$ and one particular nonlinearity $P$:

\begin{example}
Let $P(u):=u-2\sqrt u+4\log(1+\frac12 \sqrt u)$, giving $\alpha(u)=1-\frac2{\sqrt u} +4\frac{\log(1+\frac12 \sqrt u)}{u}$, $\beta(u)=1-\frac{1}{2+\sqrt{u}}$, $\phi(u)=\log(u+2\sqrt u)$, and $u_\infty(x)=\big(\sqrt{1+e^{C-V(x)}}-1\big)^2$. Then we obtain
$$
  \H_{g_2}(u|u_\infty)=\int_{\R^d} G_2(u,u_\infty)(x)\,dx,
$$
with $G_2(a,b)=\frac{1}{b+2\sqrt b}\big(a^2+\frac{8}{3}a^{\frac32}-b^2-\frac{8}{3}b^{\frac32}\big)-2(a-b)$ obtained from \eqref{eq:7} with $g_2$ from Definition \ref{def:p-ent}. Using Corollary \ref{p-entropy2} below, one can verify that this nonlinearity $P$ is admissible for the 2-entropy (and hence all $p$-entropies by Corollary \ref{p-entropy-all} below) and for all dimensions $d\in\N$, since the nonlinearity curve $(\alpha(u),\beta(u))$, $u\ge0$ connects the point $(\frac12,\frac12)$ to the point $(1,1)$ along a curve that is graph of a monotone concave function $\beta=\beta(\alpha)$, and thus lies inside $\mathcal{Q}_{(1,0)}$. In 1D the latter set is depicted in Figure \ref{remark511}.
\qed  
\end{example}

% \bigskip

A direct consequence of Remark \ref{rm:pentroparam} and condition \eqref{qg2} is the following characterization of admissible nonlinearities.

\begin{corollary}\label{p-entropy2}
The $p$-entropy, $1\leq p\leq 2$, is admissible 
if and only if the nonlinearity curve $(\alpha(u),\beta(u))$, $u\ge 0$, lies in the set $\mathcal{Q}_{(p-1,0)}$.
\end{corollary}

Notice that the sets in condition \eqref{qg} are all identical, i.e. $\mathcal{Q}_{(f'(\xi),f''(\xi))}=\mathcal{Q}_{(p-1,0)}$,  for the $p$-entropy. We can also answer a similar question for the union of all $p$-entropies.

\begin{corollary}\label{p-entropy-all}
All $p$-entropies, $1\leq p\leq 2$, are simultaneously admissible 
if and only if the nonlinearity curve $(\alpha(u),\beta(u))$, $u\ge 0$, lies in the set $\mathcal{Q}_{(1,0)}$. 
\end{corollary}

\begin{proof}
By Corollary \ref{p-entropy2} above, the nonlinearity curve should lie inside
\[ \bigcap_{1\leq p\leq 2}\mathcal{Q}_{(p-1,0)}=\mathcal{Q}_{(1,0)}, \]
where the equality follows by Lemma \ref{Qldecre} with $l=0$. 
\end{proof}

The previous result can be reformulated as follows: All $p$-entropies, $1\leq p\leq 2$, are admissible 
if and only if the $2$-entropy is admissible.

\begin{figure}[htbp]
\begin{center}
\includegraphics[width=10cm]{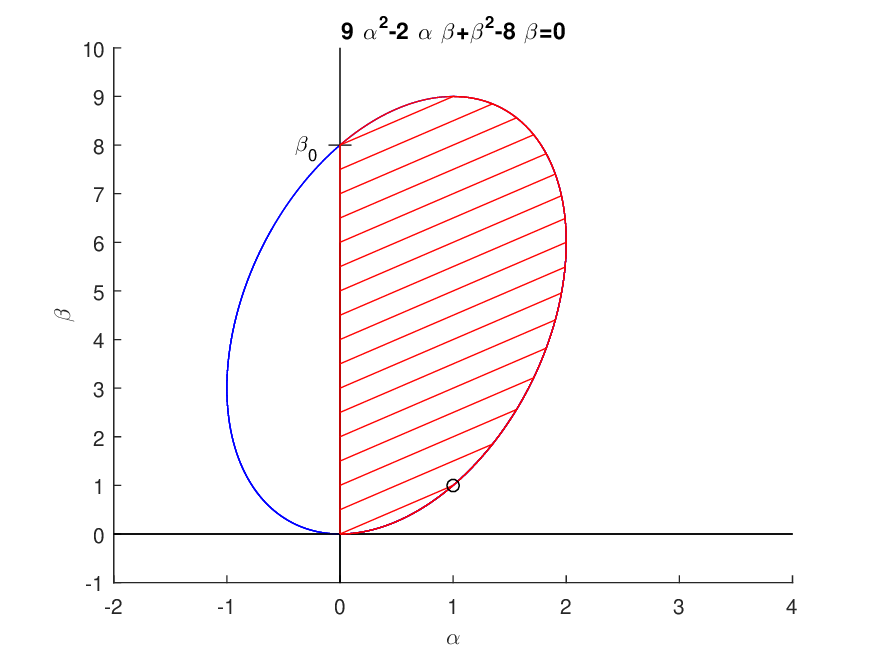}
\end{center}
%\vspace{-0.3cm}
\caption{\label{remark511} {\footnotesize The admissible set $\mathcal{Q}_{(1,0)}$ in 1D from Remark \ref{1dappl} is a truncated ellipse. The point $(1,1)$ is marked with a black circle.}}
\end{figure}

\begin{remark}\label{1dappl}
To illustrate the usefulness of the previous result, notice that we have proved that all $p$-entropies are admissible  
in one dimension if and only if the nonlinearity curve $(\alpha(u),\beta(u))$, $u\ge 0$, lies inside the ellipse
$$
9\alpha^2 -2\alpha\beta+\beta^2-8\beta\leq 0\,,
$$
intersected with the first quadrant, see Figure \ref{remark511}. Equivalently, this ellipse condition can be written as
$$
9P(u)^2-2uP(u)P'(u)+u^2P'(u)^2-8u^2P'(u)\leq 0\,,
$$
which is a 0-homogeneous differential inequality for the nonlinearity $P(u)$. 

This result also holds in two dimensions with the only change that the set in Figure \ref{remark511} becomes the inside of a parabola truncated with the first quadrant, remaining bounded. The same holds in higher dimensions with the parabola turning into a branch of a hyperbola.
\end{remark}

A direct application of Corollary \ref{cor:boundedsets} to the set of $p$-entropies gives the following characterization of their admissible nonlinearities.

\begin{corollary}
Given the $p$-entropy $g_p(\xi)$, then its admissible set $\mathcal{Q}_{g_p}$ is given by the ellipse $\mathcal{Q}_{(p-1,0)}$ for any $1< p\leq 2$, and the only admissible diffusions are the ones whose graph $(\alpha(u),\beta(u))$, $u\ge 0$, lies in the set $\mathcal{Q}_{(p-1,0)}$.
Therefore, the admissible diffusions are either regular non-degenerate or degenerate at the origin and have either linear, sublinear or saturating behavior at infinity. 
\end{corollary}

Let us illustrate the previous results by giving some examples of admissible nonlinearity curves. 

\begin{example}\label{exnon1}
\emph{Connection between two limiting linear behaviors of $P(u)$.} Given 
\begin{equation}\label{par-segmP}
P(u)=u\left( \frac{1-\tau_o}{1+\tau_1 u^{1-\lambda}}+\tau_o \right)\,,
\end{equation}
with $\tau_o,\tau_1\in (0,\infty)$ and $\lambda\in [0,1)$, then the reader can easily check that
$\alpha(0)=\beta(0)=1$ and $\alpha(\infty)=\beta(\infty)=\tau_o$. More precisely, $P(u)\simeq u$ for $u\to 0+$ and $P(u)\simeq \tau_o u$ for $u\to\infty$. Moreover, the nonlinearity curve in non-parametric form is given by 
\begin{equation}\label{par-segm}
  \beta(\alpha)=\frac{1-\lambda}{1-\tau_o}(\alpha-\tau_o)^2 + \lambda (\alpha-\tau_o) + \tau_o\,,
\end{equation}
with $\alpha$ between 1 and $\tau_o$. Note that this segment of parabola joins the points $(1,1)$ and $(\tau_o,\tau_o)$ on the diagonal. If $\tau_o>1$, $\beta(\alpha)$ lies above the diagonal, and otherwise below, see Figure \ref{example1}. 

It is straightforward to check that $\lambda$ is the slope of this nonlinearity curve at $\alpha=\tau_o$. Therefore, the sharp range of $\lambda$ for making the $p$-entropy admissible can be obtained from Corollary \ref{p-entropy2}. For instance, taking $\lambda$ close enough to 1, we can ensure that the nonlinearity curve lies inside the truncated quadrics $\mathcal{Q}_{(p-1,0)}$ if $\tau_o<1$ for $1\leq p\leq 2$ and if $\tau_o>1$ for $1\leq p \le 1+\frac1{\tau_o}<2$.  
(cf. Figure \ref{example1P}). Hence all $p$-entropies $g_p$ are admissible for $\tau_o<1$.
\end{example}

\begin{figure}[htbp]
\begin{center}
\includegraphics[width=10cm]{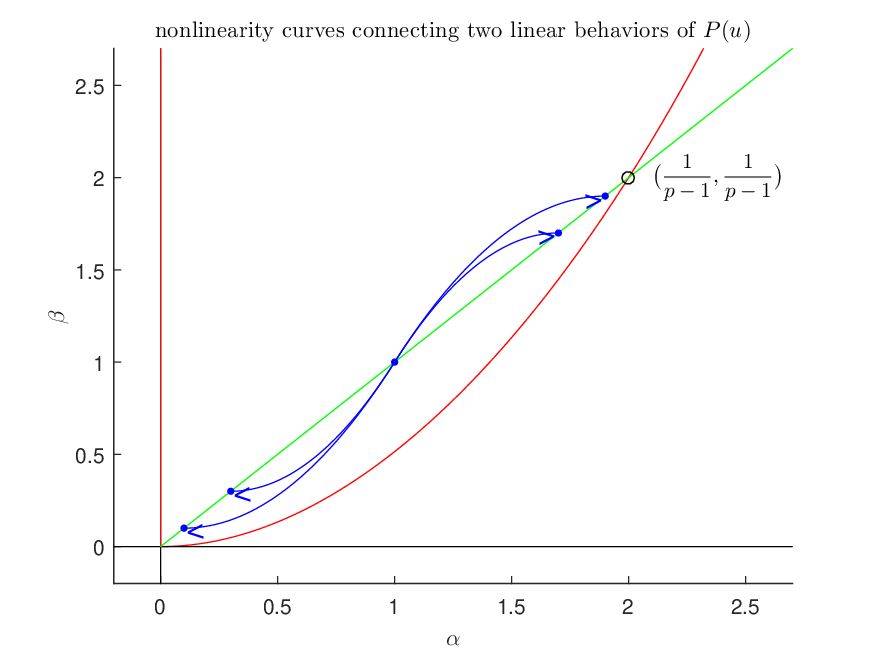}
\end{center}
%\vspace{-0.3cm}
\caption{\label{example1} {\footnotesize Illustration of Example \ref{exnon1}. The four nonlinearity curves (in blue) connect in the $\alpha-\beta$--plane the point $(1,1)$ to the points $(\tau_o,\tau_o)$ with $\tau_o=0.1,\,0.3,\,1.7,\,1.9$ (the maximum allowed value of $\tau_o$ is $\frac{1}{p-1}$). Their orientation w.r.t. $u$ is indicated by arrows. These curves are the parabola segments from \eqref{par-segm} with the choice $\lambda=0$. Since all starting and end points lie on the diagonal (in green), the limiting behavior of $P(u)$ as $u\to0+$ and $u\to\infty$ is linear. All four nonlinearity curves are admissible for the $p$-entropy with $p=\frac32$ and $d=1$, as they lie inside the truncated ellipse $\Q_{(0.5,0)}$; $\partial\Q_{(0.5,0)}$ is plotted in red. [colors only online]}}
\end{figure}

\begin{figure}[htbp]
\begin{center}
\includegraphics[width=10cm]{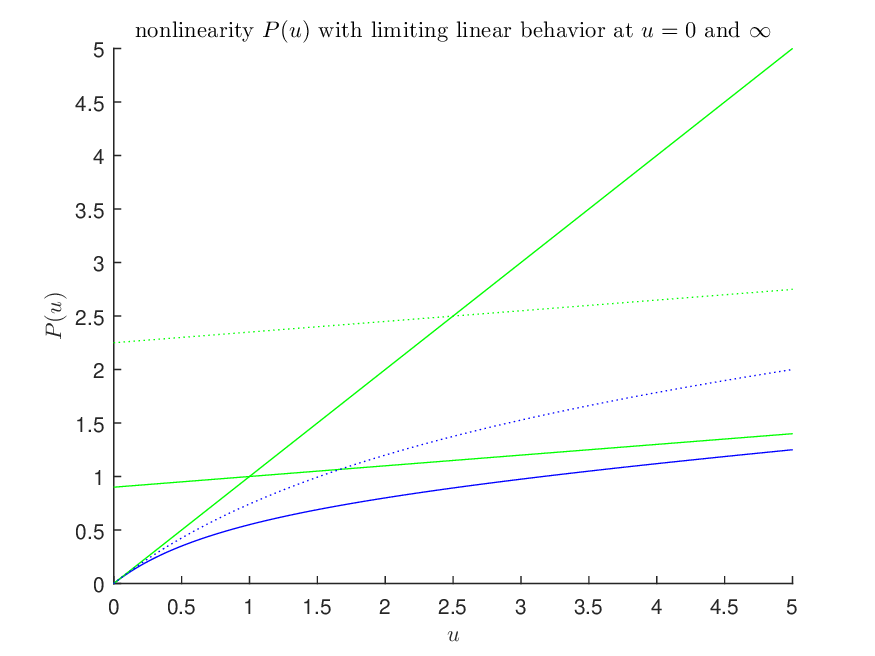}
\end{center}
%\vspace{-0.3cm}
\caption{\label{example1P} {\footnotesize Illustration of Example \ref{exnon1}. For the lowest curve in Figure \ref{example1} with $\tau_0=0.1$ we give two examples of corresponding nonlinearities $P(u)$ from \eqref{par-segmP} (both in blue; solid curve for $\tau_1=1$ and dotted curve for $\tau_1=0.4$). Note that their $\beta(\alpha)$-curves coincide for all $\tau_1\in(0,\infty)$. The two green solid lines and the green dotted line are the asymptotes at $u=0$ ($P_{asymp}=u$) and at $u=\infty$ ($P_{asymp}=\frac{1-\tau_0}{\tau_1}+\tau_o u$). [colors only online]}}
\end{figure}

\begin{example}\label{exnon2}
\emph{Connection between linear and sublinear limiting behaviors of $P(u)$.} Given 
\begin{equation}\label{lin-sublin}
P(u)=\left\{\begin{array}{cl} 
\tau_o u & \mbox{for } 0\leq u\leq 1\\[2mm]
\tau_o\frac{u^{\lambda}-1}{\lambda}+\tau_o & \mbox{for } u\geq 1
\end{array}\right. ,  
\end{equation}
with $\tau_o\in (0,\infty)$ and $\lambda\in (0,1)$, then the reader can easily check that
$\alpha(0)=\beta(0)=\tau_o$,  $\alpha(\infty)=\beta(\infty)=0$, and $\mu(u)>0$. 
More precisely, $P(u)\simeq \tau_o u$ for $u\to 0+$ and $P(u)\simeq \tfrac{\tau_o}{\lambda} u^\lambda$ for $u\to\infty$.
Moreover, the nonlinearity curve in non-parametric form is given by 
$$
\alpha(\beta)=\frac{\beta}{\lambda}-\frac{1-\lambda}{\lambda}\tau_o^{-\tfrac{\lambda}{1-\lambda}} \beta^{\tfrac1{1-\lambda}}\,,
$$
with $\alpha\in[0,\tau_o]$. From this formula, the reader can check that $\frac{d\beta}{d\alpha}(0+)=\lambda$ and $\frac{d\beta}{d\alpha}(\tau_o-)=\infty$ as predicted in \S\ref{sec:5.1}, and moreover $\beta(\alpha)$ is increasing and convex. As a consequence, this curve joins the point $(\tau_o,\tau_o)$ to the origin and it lies below the diagonal. 

Since the nonlinearity curve $\beta(\alpha)$ is convex, it lies above the ray with slope $\lambda$. Therefore, by choosing $\tau_o$ small enough, depending on $1<p\le 2$,  and $\lambda \in [\frac{d-1}{d},1)$ we can ensure that the nonlinearity curve lies inside the truncated quadric $\mathcal{Q}_{(p-1,0)}$. Note that $\frac{d-1}{d}$ is the slope of the tangent line to this quadric at the origin. In this case, the $p$-entropy is admissible. 

Finally we compare (for simplicity just for $d=1$) this nonlinearity with $\tau_o=1$ to the linear diffusion with $D=1$: The point $(\alpha,\,\beta)=(1,1)$ lies both on the nonlinearity curve and on $\partial\mathcal{Q}_{(1,0)}$, cf.\ Figure \ref{remark511}.  But since the nonlinearity curve satisfies $\frac{d\beta}{d\alpha}(1-)=\infty$, for any $\lambda\in(0,1)$, it cannot be fully included in $\mathcal{Q}_{(1,0)}$. Corollary \ref{p-entropy-all} thus inplies that the nonlinearity \eqref{lin-sublin} does \emph{not} admit all $p$-entropies, $1\le p\le2$, even though $P$ coincides with the linear diffusion function for $0\le u\le1$.
\end{example}

\begin{example}\label{exnon3}
Let us also point out that Example \ref{exnon1} with $\tau_o=0$ has an analogous limiting behavior as Example \ref{exnon2}. Actually, the nonlinearity becomes
$$
P(u)=\frac{u}{1+\tau_1 u^{1-\lambda}}\,,
$$
with $\tau_1\in (0,\infty)$ and $\lambda\in [\frac{d-1}{d},1)$. Since the slope of $\beta(\alpha)$ at $\alpha=0$ is $\lambda$, similar arguments as above imply that these curves will lie inside  $\mathcal{Q}_{(p-1,0)}$ for $\lambda$ close to 1.
\end{example}

%%%%%%%%%%%%%%%%%%%%%%%%%%%%%%%%%%%%%%%%%%%%%%%%%%%%%%%%%%%%%%%%%%%%%%%%%%%%%%%%

\subsection{Admissible nonlinearities for all entropies of the linear diffusion case.}\label{sec:admiss-sets-for alllin}

We recall that the admissible entropies for linear diffusion equations are characterized by Definition \ref{d:2.1} or, equivalently, by condition \eqref{f-inequ}, see Lemma \ref{f-solutions} for the precise statements.

In generalization of \S\ref{sec:admiss-sets-p} and Definition \ref{def:p-ent} we shall now refer to the \emph{admissible entropies for linear diffusion equations} by fixing their generating function $g$ from \eqref{entrof}. All of those $g$ will then be used in the entropy functional $\H_g(u|u_\infty)$. Note that these definitions depend via $\xi$ also on the nonlinear function $\phi(u)$ defined in \eqref{eq:P2phi}. Hence, a linear and nonlinear diffusion equation may share the same generating function $g$, but the actual form of the functional $\H_g$ will still depend on the considered nonlinearity $P$. 

\begin{theorem}\label{thmlinentropy}
All admissible entropies for linear diffusions are admissible for the equation \eqref{GFP} if and only if the nonlinearity curve $(\alpha(u),\beta(u))$, $u\ge 0$ lies in the set
$$
\mathcal{Q}_{lin}:= \mathcal{Q}_{(1,0)}\cap \{0\le \alpha\le 1\}.
$$
\end{theorem}

\begin{proof}
We divide the proof into two steps.

{\bf Step 1.-} We shall first show that the nonlinearities should lie in the set
$$
\bigcap_{(a,b)\in\mathcal{S}_1\cup \mathcal{S}_2} \mathcal{Q}_{(a,b)},
$$
with
$\mathcal{S}_1=\left\{(a,b): 0<a<1 \mbox{ and } b= a (a-1) \right\}$ and  $\mathcal{S}_2=\left\{(1,b): b\ge 0 \right\}$.

We first realize that any admissible entropy for the linear diffusion case is represented by a curve passing through a point $(a,b)\in \mathcal{S}$ with
$$
\mathcal{S}=\left\{(a,b): 0<a<1 \mbox{ and } b\ge  a (a-1) \right\}\cup\{(0,0),(1,0)\}
$$
according to Lemma \ref{f-solutions}. Notice that the entropy curve could be just a single point, see Example \ref{p-entropy}. In fact, one can construct admissible entropies for the linear diffusion equation passing through any of those points. 

Moreover, any admissible entropy $(f'(\xi),f''(\xi))$ for the linear diffusion equation is globally defined, that is $\xi_{\min}=-\infty$. In fact, any reparameterization given by $(f'(\xi-\xi_o),f''(\xi-\xi_o))$ for a fixed $\xi_o \in \R$ is also an admissible entropy for the linear diffusion equation. Therefore, any nonlinearity for which all admissible entropies of the linear diffussion equation are admissible should satisfy \eqref{qg}, i.e.
$$
 (\alpha(r),\beta(r))\in\mathcal{Q}_{(f'(\xi-\xi_o),f''(\xi-\xi_o))} \qquad \mbox{for all } 0<r<\varphi(\xi)  ,
$$
for all $\xi,\xi_o\in\R$. By rewriting this condition as
$$
 (\alpha(r),\beta(r))\in\mathcal{Q}_{(f'(\xi),f''(\xi))} \qquad \mbox{for all } 0<r<\varphi(\xi+\xi_0)  ,
$$
for all $\xi,\xi_o\in\R$, and taking the limit $\xi_o\to\infty$, we conclude that the nonlinearity should satisfy
$$
 (\alpha(r),\beta(r))\in\mathcal{Q}_{(f'(\xi),f''(\xi))} \qquad \mbox{for all } r>0 \mbox{ and } \xi\in\R.
$$
Therefore, the nonlinearity curve $(\alpha(r),\beta(r))$ lies in
$$
\bigcap_{\xi\in\R} \mathcal{Q}_{(f'(\xi),f''(\xi))}
$$
for all admissible entropies $(f'(\xi),f''(\xi))$ of the linear diffusion equation. As a consequence, the nonlinearity curve $(\alpha(r),\beta(r))$ lies in
$$
\bigcap_{(a,b)\in\mathcal{S}} \mathcal{Q}_{(a,b)} 
$$
which is the blue shaded region in Figure \ref{fig:sketchsets}.

The statement of Step 1 follows now taking into account Lemma \ref{Qldecre}. This is due to the fact that every point in $\mathcal{S}_1\cup \mathcal{S}_2$ is the right end point of the intersection of the set $\mathcal{S}$ with a parabola passing through the origin of the form $(a,l a^2)$ for some $l \in \R$, see Figure \ref{fig:sketchsets}, and the sets $\mathcal{Q}_{(a,l a^2)}$ are decreasing in $a$ for every $l\in\R$. Notice that
$$
\mathcal{S}=\left\{(a,l a^2): 0\leq a < 1, l\in \R \mbox{ and } l a\ge  a-1 \right\}\cup\{(1,0)\}.
$$

{\bf Step 2.-} We now analyze the monotonicity of the sets $\mathcal{Q}_{(a,b)}$ along the parameterized sets $\mathcal{S}_1$ and $\mathcal{S}_2$. We start with the last one: notice that from \eqref{quadric-cond} we rewrite $T(\alpha,\beta,a,b)\leq 0$ as
\begin{align*}%\label{qabincreasing}
  \frac1{a^2} T(\alpha,\beta,a,b)=&\left(1+8\frac{2-d}{d}\right) \alpha^2 - \left(2+8\frac{1-d}{d}\right) \alpha\beta + \beta^2 \nonumber\\
  &-8\left(\beta+\frac{1-d}{d}\alpha\right)-\frac{8b\alpha}{a^2} \left(\beta+\frac{1-d}{d}\alpha\right) \le0\,.
\end{align*}
Since $\beta+\frac{1-d}{d}\alpha\ge 0$ due to McCann's condition and $\alpha \ge 0$, then the last term is decreasing with $b$ for any fixed $a>0$. Hence, the set $\mathcal{Q}_{(a,b)}$ is increasing in $b$ for any fixed $a>0$. Using this fact for $a=1$, we conclude
$$
\bigcap_{(a,b)\in\mathcal{S}_2} \mathcal{Q}_{(a,b)}=\mathcal{Q}_{(1,0)}.
$$
Analogously, we proceed with $\mathcal{S}_1$ to deduce that $T(\alpha,\beta,a,a(a-1))\leq 0$ is equivalent to
\begin{align}\label{quadric-cond2bis2}
   \frac1{a^2}T(\alpha,\beta,a,a(a-1))=&\left(1+8\frac{2-d}{d}\right) \alpha^2 - \left(2+8\frac{1-d}{d}\right) \alpha\beta  -8\left(\beta+\frac{1-d}{d}\alpha\right)\nonumber\\
  &+ \beta^2+8\tau(\alpha-1) \left(\beta+\frac{1-d}{d}\alpha\right) \le0\,,
\end{align}
where $\tau=\frac{1-a}a\ge 0$, $0<a\le 1$. Since $\beta+\frac{1-d}{d}\alpha\ge 0$ due to McCann's condition, then the last term is decreasing with $\tau$, and hence increasing with $a$, for $0\le \alpha \le 1$. Hence, the set $\mathcal{Q}_{(a,a(a-1))}\cap \{0\le \alpha\le 1\}$ is decreasing in $a$.
As a consequence, we obtain
\begin{equation}\label{aux1}
\left(\bigcap_{(a,b)\in\mathcal{S}_1} \mathcal{Q}_{(a,b)}\right)\cap \{0\le \alpha\le 1\}=\mathcal{Q}_{(1,0)}\cap \{0\le \alpha\le 1\}.
\end{equation}
Finally, the same argument for $\alpha> 1$ leads to the statement that the set $\mathcal{Q}_{(a,a(a-1))}\cap \{\alpha> 1\}$ is increasing in $a$. Observe that for every fixed $\alpha >1$ and $\beta\geq 0$ the last term in \eqref{quadric-cond2bis2} goes to infinity as $a\to 0^+$. Therefore, we deduce
\begin{equation}\label{aux2}
   \left(\bigcap_{(a,b)\in\mathcal{S}_1} \mathcal{Q}_{(a,b)}\right)\cap \{\alpha> 1\}=\emptyset. 
\end{equation}
Combining the statements \eqref{aux1} and \eqref{aux2} together with Step 1, we conclude the proof.
\end{proof}

Note that the point $(\alpha,\,\beta):=(1,1)\in\partial\mathcal Q_{lin}$, in all dimensions $d\in\N$. It corresponds to $P(u)\equiv u$ and, if fact, it is a corner point of $\mathcal Q_{lin}$; it is marked in Figure \ref{remark511}. Thus, any (small) perturbation of the linear diffusion $P(u)\equiv u$ (for small, intermediate, or large $u$) may move the nonlinearity curve out of $\mathcal Q_{lin}$. This would imply that some of the entropies of the linear diffusion case become inadmissible. More precisely we have:
\begin{corollary}
    Let the nonlinearity $P(u)\not\equiv u$ satisfy $P'(u_0)=P(u_0)/u_0$ for some $u_0>0$; for example, $P$ could coincide with the linear function $u$ on some non-trivial interval (as, e.g., in Example \ref{exnon2} with $\tau_o=1$). Then, it cannot happen that \emph{all} entropies of the linear diffusion case are admissible for $P$.
\end{corollary}
\begin{proof}
Due to the discussion in \S\ref{sec:5.1}, a continuation of the nonlinearity curve out of the point $(1,1)$ and above the diagonal $\alpha=\beta$ would lead to $\alpha>1$. But such points would lie outside of $\mathcal Q_{lin}$. Moreover, a continuation below the diagonal has a vertical tangent line at $(1,1)$ (see \S\ref{sec:5.1}), and this would lead the nonlinearity curve again out of $\mathcal Q_{lin}$. The statement then follows from Theorem \ref{thmlinentropy}.
\end{proof}

%%%%%%%%%%%%%%%%%%%%%%%%%%%%%%%%%%%%%%%%%%%%%%%%%%

\subsection{Stability of relative entropies $\H_g[P]$ w.r.t.\ $P$.}\label{Sec:stability}
In view of the examples in the previous two subsections \ref{sec:admiss-sets-p} and \ref{sec:admiss-sets-for alllin}, we can generalize the setting
by comparing relative entropies generated by the same function $g$, but pertaining to two different, non-degenerate nonlinearities $P_1$ and $P_2$ that both satisfy the integrability condition \eqref{int-cond}.

Let us also assume that the corresponding diffusion equations \eqref{GFP} have the same potential $V$ with $\inf_{\R^d} V=0$, and that initial conditions, and hence their steady states $u_{\infty,1}$, $u_{\infty,2}$ have the same mass $M>0$. We recall that these steady states satisfy
\begin{equation}\label{C-bar}
  \phi_j(u_{\infty,j})+V(x)=\bar C_j;\quad j=1,2.
\end{equation}
Under this setup we have the following stability estimates, here $u$ is a generic density function in $L^1_M(\R^d)$.
\begin{prop}
    Let the entropy generator $g$ be globally Lipschitz with constant $L$. Let $u\in L^1_M(\R^d)$. Moreover, let one of the following conditions hold:
    \begin{enumerate}
        \item[(a)] 
        \begin{align*}
          K_1&:=\|P_1'-P_2'\|_{L^\infty\big(0,\max(\sup u,\max u_{\infty,1})\big)} <\infty, \\
          \mbox{and}& \\
          K_2&:=\int_{\R^d} |u-u_{\infty,1}|\big(|\log u|+|\log u_{\infty,1}|\big)dx < \infty. 
        \end{align*}
        \item[(b)] There is a constant $K_3$ such that, for all $s\in(0,\max(\sup u,\max u_{\infty,1}))$,
        \begin{equation}\label{Delta-phi}
        |\phi_1(s)-\phi_2(s)| = \Big|\int_1^s \frac{P_1'(r)-P_2'(r)}{r} dr\Big| \le K_3.
        \end{equation}
    \end{enumerate}
    Then the following stability estimates hold:
    \begin{equation*}%\label{H-stab-est}
       |\H_{g,1}(u|u_{\infty,1}) - \H_{g,2}(u|u_{\infty,2})| \le
       \H_{g,2}(u_{\infty,1}|u_{\infty,2}) +\left\{ \begin{array}{rc} {\displaystyle L\Big( K_1 K_2 + 2M\,|\bar C_1-\bar C_2|\Big)}, & \mbox{in case (a)}\\[0.3cm]
 2LM(K_3+|\bar C_1-\bar C_2|), & \mbox{in case (b)}\end{array}\right., \nonumber
    \end{equation*}
for any two non-degenerate nonlinearities $P_1$ and $P_2$ satisfying \eqref{int-cond}.
\end{prop}
\begin{proof}
Using Definition \ref{genrelentropy}, we split the domain of integration and estimate 
\begin{align*}
    |\H_{g,1}(u|u_{\infty,1}) - &\, \H_{g,2}(u|u_{\infty,2})| \\ \le &\, \left|\int_{\R^d}\int_{u_{\infty,1}(x)}^{u(x)} \big[g\big(\phi_1(s)-\phi_1(u_{\infty,1}(x))\big) - g\big(\phi_2(s)-\phi_2(u_{\infty,2}(x))\big) \big]\,ds\,dx \right| \\
    & +\left|\int_{\R^d}\int_{u_{\infty,2}(x)}^{u_{\infty,1(x)}} g\big(\phi_2(s)-\phi_2(u_{\infty,2}(x))\big) \big]\,ds\,dx \right|\,,
\end{align*}
where the second term is equal to $\H_{g,2}(u_{\infty,1}|u_{\infty,2})$. It thus remains to estimate the first term of the r.h.s. We now use \eqref{C-bar} and the assumption that $g$ is Lipschitz to estimate it by
\begin{eqnarray}\label{error1}
    && \int_{\R^d}\left|\int_{u_{\infty,1}(x)}^{u(x)} \big|g\big(\phi_1(s)-\phi_1(u_{\infty,1}(x))\big) - g\big(\phi_2(s)-\phi_2(u_{\infty,2}(x))\big) \big|\,ds\right| dx \nonumber\\
    && \qquad \le L \int_{\R^d}\Big[ \,\Big|\int_{u_{\infty,1}(x)}^{u(x)} |\phi_1(s)-\phi_2(s)|\,ds\Big| + |\bar C_1-\bar C_2|\,|u(x)-u_{\infty,1}(x)| \Big]dx \,.
\end{eqnarray}
With \eqref{Delta-phi}, the stability estimate for case (b) is then immediate. In case (a) we  use for the first term in \eqref{error1} the following estimate
\begin{align*}
    \left|\int_{u_{\infty,1}(x)}^{u(x)} |\phi_1(s)-\phi_2(s)| ds\right|
    &= \left|\int_{u_{\infty,1}(x)}^{u(x)} \Big|\int_1^s \frac{P_1'(r)-P_2'(r)}{r} dr\Big| ds\right| \\
    &\le K_1\,\left|\int_{u_{\infty,1}(x)}^{u(x)} |\log s| ds\right| \\
    &\le K_1\big|u(x)-u_{\infty,1}(x)\big|\max\big(|\log u(x)|,|\log u_{\infty,1}(x)|\big)\,.
\end{align*}
\end{proof}

%%%%%%%%%%%%%%%%%%%%%%%%%%%%%%%%%%%%%%%%%%%%%%%%%%%%%%%%%%%%%%%%%%%%%%%%%%%%%%%%

\section{Improved decay estimates with weights}\label{sec-funcineq}
\setcounter{equation}{0}

In \S\ref{sec-quasilin} we characterized all admissible entropies for a given nonlinear diffusion. But even for power-law nonlinearities, the (analogs of the) $p$-entropies, i.e.,   $\mathcal{H}_{m,p}(u|\uinf)$ are not explicit, see Proposition \ref{prop:genm}. Hence it is not obvious, what is the additional information provided by knowing their exponential decay. In this section we shall illustrate this aspect by deriving decay estimates for explicit, weighted $L^1$-norms of $u-\uinf$, which are controlled by some of our new entropies.

\subsection{Generalized Csiszár-Kullback inequalities}
We first give a general result showing the control of the $L^1$-norm of the difference between a given function and the stationary $\uinf$ by our generalized relative entropies under suitable assumptions. Our result is valid for both the degenerate and non-degenerate diffusions. 
For the degenerate case let $B_\infty:=\mbox{\rm supp} (\uinf)$ be the (compact) support of the steady state $\uinf$ of \eqref{GFP} given by Definition \ref{genrelentropy2}. 
We also notice that our result supersedes the most general result in the literature for this type of inequalities \cite{AMTU2,CJMTU}, see Example \ref{ex:comp-CJMTU} below.

\begin{theorem}\label{gencz}
Assume that either $P$ is a regular non-degenerate diffusion satisfying $\inf_{r\ge0} P'(r)>0$ or a degenerate diffusion in which case we further assume that for some constant $n>0$, the map $r\mapsto P'(r)/r^n$ is continuous and positive on $(0,\infty)$, and is non-increasing on some interval $(0,r_0]$. Moreover, assume that the entropy function $g$ is such that $g'$ is positive on $(\xi_{\min},\infty)$.
Then there is a constant $C$ (depending on $P$, $V$, $M$, $g$, and possibly $n$ for degenerate diffusions) such that
    \begin{align}
        \label{eq:ourCK}
        \|u-\uinf\|_{L^1} \le C {\mathcal H}_g(u|\uinf)^{1/2}
    \end{align}
for all functions $u\in L^1_M(\R^d)$ such that ${\mathcal H}_g(u|\uinf)<\infty$.\\
\end{theorem}
\begin{proof}
The proof needs to improve over the Taylor expansion in \eqref{Taylorent}. We structure it in the following steps:

\

{\bf Step 1.-}
In the regular non-degenerate diffusion case, we obtain
at every $x$ with $0\leq u(x)<\uinf(x)$ by integrating by parts using that $g(0)=0$:
 \begin{align}
        G\big(u,u_\infty\big)(x)
        &= \int^{u(x)}_{\uinf(x)} g\big(\phi(s)-\phi(\uinf(x))\big) \, ds \nonumber\\
        &= - (\uinf(x)-u(x))\,g(0) +\int
        _{u(x)}^{\uinf(x)}(s-u(x))\,g'\big(\phi(s)-\phi(\uinf(x))\big)\phi'(s)ds  \nonumber \\
        &\ge \int_{(\uinf(x)+u(x))/2}^{u_\infty(x)}(s-u(x))\,g'\big(\phi(s)-\phi(\uinf(x))\big)\phi'(s)ds \nonumber\\
        &\ge {\frac38} \left(\uinf(x)-u(x)\right)^2\inf_{\frac{\uinf(x)}2<s<\uinf(x)} g'\big(\phi(s)-\phi(\uinf(x))\big)\: \inf_{\frac{\uinf(x)}2<s<\uinf(x)}\phi'(s)\nonumber \\
        &=: {\frac38}\left(\uinf(x)-u(x)\right)^2 \, I_1\, I_2\,, \label{G-est-nondeg}
    \end{align}
where $I_1$ and $I_2$ are the two infima in the penultimate line.

In the degenerate diffusion case, recall that $\phi(u_\infty(x))+V(x)=\bar C$ for all $x\in B_\infty$. It then follows at every $x$ with $0\leq u(x) < \uinf(x)$ --- hence such an $x$ is in $B_\infty$ --- that
    \begin{align*}
        \tilde G\big(u(x),\uinf(x);x\big)
        &= \int^{u(x)}_{\uinf(x)} 
        g\big(\phi(s)+V(x)-\bar C\big)ds  \\
        &= \int^{u(x)}_{\uinf(x)} g\big(\phi(s)-\phi(\uinf(x))\big) \, ds 
        \geq {\frac38}\left(\uinf(x)-u(x)\right)^2 \, I_1\, I_2\,,
    \end{align*}
    proceeding similarly as for \eqref{G-est-nondeg}.

\

{\bf Step 2.-} The first infimum $I_1$ is positive, uniformly in $x\in B_\infty$. To see this, define $U:=\max u_\infty>0$, implying $\phi(U)=\bar C$. In both cases (regular non-degenerate diffusion and degenerate diffusion), $P'$ is continuous on $[0,U]$ and $\phi$ is monotonically increasing.    
In the degenerate case we recall that $\xi_{\min}=\phi(0)-\bar C=-\int_0^U P'(r)/r\,dr>-\infty$.
Hence we estimate for all $s\in(\uinf(x)/2,\uinf(x))$: 
\begin{align}\label{d-phi-estimate}
        0 &\ge \phi(s) - \phi(\uinf(x)) = - \int_s^{\uinf(x)}\frac{P'(r)\,dr}r 
        \ge - \int_{\uinf(x)/2}^{\uinf(x)}\frac{P'(r)\,dr}{r} \\
    &\ge
    \begin{cases}
        - \int_{0}^{U/2}\frac{P'(r)\,dr}{r}, \quad & u_\infty(x)\in[0,U/2]\,,\\
        - \int_{U/4}^{U}\frac{P'(r)\,dr}{r}, \quad & u_\infty(x)\in(U/2,U]
    \end{cases}  \nonumber \\
    &\ge \delta - \int_0^U \frac{P'(r)\,dr}{r}  = \delta+\xi_{\min} , \nonumber 
\end{align}
where $\delta:=\min\Big(\int_{0}^{U/4}\frac{P'(r)\,dr}{r},\,\int_{U/2}^{U}\frac{P'(r)\,dr}{r}\Big)>0$ is independent of $x\in B_\infty$. 
In the non-degenerate case we estimate the last integral in \eqref{d-phi-estimate} as:
$$
  - \int_{\uinf(x)/2}^{\uinf(x)}\frac{P'(r)\,dr}{r} 
        \ge -\|P'\|_{L^\infty(0,\infty)} \int_{\uinf(x)/2}^{\uinf(x)}\frac{dr}{r} =-\|P'\|_{L^\infty(0,\infty)} \log 2.
$$
In the degenerate case $g'$ is continuous and positive on $[\delta+\xi_{\min},0]$, and in the non-degenerate case on $[-\|P'\|_{L^\infty(0,\infty)} \log 2,0]$. Hence, in both cases there is a positive lower bound on $g'(\phi(s)-\phi(\uinf(x)))$ for all $s\in(\uinf(x)/2,\uinf(x))$. 

\

{\bf Step 3.-}
The second infimum $I_2$ is also positive: notice first that in the case of regular non-degenerate diffusions, we have
    \begin{align}\label{phi-est}
        \phi'(s) = \frac{P'(s)}{s} 
        \ge \frac{P'_\text{min}}{\uinf(x)} 
        \geq \frac{P'_\text{min}}{P'_{\text{max}}}  \frac{P'(\uinf(x))}{\uinf(x)}:= c\phi'(\uinf(x))
    \end{align}
for all $s\in(\uinf(x)/2,\uinf(x))$ with 
\[
    0<P'_\text{min} :=\min_{0\leq s\leq U} P'(s) \le P'_{\text{max}}:=\max_{0\leq s\leq U} P'(s). 
\]
By the assumed uniform positivity and continuity of $P'(r)/r^n$, 
we conclude that there is some $\bar c>0$ such that 
\[ 
    \frac{P'(s)}{s^n} \ge \bar c \frac{P'(r)}{r^n} 
\]
for all $0<s<r\le U$. Consequently,  we have for all $s\in(\uinf(x)/2,\uinf(x))$:
    \begin{align*}
        \phi'(s) = \frac{P'(s)}{s} 
        = s^{n-1}\frac{P'(s)}{s^n} 
        \ge \bar c s^{n-1}\frac{P'(\uinf(x))}{\uinf(x)^n}
        = \bar c\left(\frac{s}{\uinf(x)}\right)^{n-1}\frac{P'(\uinf(x))}{\uinf(x)}
        \ge \tilde c \phi'(\uinf(x))\,
    \end{align*}
with $\tilde{c}=\bar c\min(1,2^{1-n})$ achieving a similar estimate as in \eqref{phi-est} for the regular non-degenerate case.

\

{\bf Step 4.-} We now conclude by collecting the estimates of $I_1$ and $I_2$, that there is a constant $K$, independent of $u$, such that
 \begin{align*}
        \phi'(\uinf(x))(\uinf(x)-u(x))^2 \le K^2  G(u(x),\uinf(x))
    \end{align*}
in the regular non-degenerate diffusion case, and 
    \begin{align*}
        \phi'(\uinf(x))(\uinf(x)-u(x))^2 \le K^2 \tilde G(u(x),\uinf(x);x)
    \end{align*}
in the degenerate diffusion case, 
holds for all $x\in\R^d$ with $0\leq u(x)<\uinf(x)$. Now, both for non-degenerate and degenerate diffusion, recalling that $u$ and $u_\infty$ are non-negative and of the same mass $M$:
    \begin{align*}
        \|u-\uinf\|_{L^1} 
        &= 2\int_{\{u<\uinf\}} (\uinf-u)(x)\;dx \\
        & \le 2\left(\int_{\{u<\uinf\}}\phi'(\uinf)(u-\uinf)^2(x)\;dx\right)^{1/2}\left(\int_{u<\uinf}\frac{dx}{\phi'(\uinf)}\right)^{1/2} \\
        &\le 2K\left(\int_{B_\infty}\frac{dx}{\phi'(\uinf)}\right)^{1/2}\mathcal H_g(u|\uinf)^{1/2}
    \end{align*} 
leading to the desired estimate.

\
    
{\bf Step 5.-} It remains to show that 
    $$
    \int_{B_\infty}\frac{dx}{\phi'(\uinf)}<\infty\,.
    $$
In the regular non-degenerate diffusion case, notice first that $\supp(u_\infty)$ is $\R^d$ and we directly estimate the above integral as
    \begin{align*}
        \int_{\R^d}\frac{dx}{\phi'(u_\infty)}
        &= \int_{\R^d}\frac{\uinf}{P'(u_\infty)}dx\leq \frac1{P'_{\text{min}}}\int_{\R^d}\uinf dx=\frac{M}{P'_{\text{min}}}.
    \end{align*}
    In the degenerate diffusion case, define $\bar V:=\bar C-\phi(0+)=-\ximin$, so that $B_\infty=\{V\le\bar V\}$, and let $\rho:=\phi^{-1}(\bar C-\bar V/2)>0$. We have:
    \begin{align*}
        \int_{B_\infty}\frac{dx}{\phi'(u_\infty(x))}
        &= \int_{\{0<V<\bar V/2\}}\frac{dx}{\phi'\big(\phi^{-1}(\bar C-V(x))\big)}
        +\int_{\{\bar V/2< V<\bar V\}}\frac{dx}{\phi'\big(\phi^{-1}(\bar C-V(x))\big)}\\
        &\le \frac{\big|\{0<V<\bar V/2\}\big|}{\inf_{\rho<r< U}\phi'(r)}
        +\int_{\{\bar V/2< V<\bar V \}}\frac{dx}{\phi'\big(\phi^{-1}(\bar C-V(x))\big)}.
    \end{align*}
    Since $\phi'(r)=P'(r)/r$ has a positive lower bound on $[\rho, U]$ by hypothesis, the first expression is finite, and it suffices to estimate the remaining integral.

    For evaluation of the second integral, we pass to radial coordinates $x=x_0+r\theta$, where $\theta\in{\mathbb S}^{d-1}$ and $r\ge0$, and $x_0$ is the minimal point of $V$. Let $0<\underline r(\theta)<\overline r(\theta)$ such that
    \[ \{x\in\Rd\ |\ \bar V/2< V(x) < \bar V\} = \{ x_0+r\theta\ |\ \theta\in\mathbb{S}^{d-1},\ \underline r(\theta)<r<\overline r(\theta) \}. \]
    Since $r\mapsto V(x_0+r\theta)$ is a $\lambda$-convex function for each fixed $\theta$, with minimum zero at $r=0$, we have that $V(x_0+r\theta)\ge \lambda/2\,r^2$, and in particular we have that $\overline r(\theta)\le R:=(2\bar V/\lambda)^{1/2}$. Since further $\partial_r V(x_0+r\theta)\ge\lambda r$ for each $r>0$, and
    \begin{align*}
        \partial_r\phi^{-1}(\bar C-V(x_0+r\theta)) = \frac{-\partial_rV(x_0+r\theta)}{\phi'\big(\phi^{-1}(\bar C-V(x_0+r\theta))\big)},
    \end{align*}
    we conclude that
    \begin{align*}
        \int_{\{ \bar V/2< V<\bar V\}}\frac{dx}{\phi'\big(\phi^{-1}(\bar C-V(x))\big)}
        &= \int_{{\mathbb S}^{d-1}} \int_{\underline r(\theta)}^{\overline r(\theta)} \frac{r^{d-1} dr}{\phi'\big(\phi^{-1}(\bar C-V(x_0+r\theta))\big)} d\theta\\
        &\le -\frac1\lambda\int_{{\mathbb S}^{d-1}} \int_{\underline r(\theta)}^{\overline r(\theta)} r^{d-2}\partial_r\phi^{-1}(\bar C-V(x_0+r\theta)) dr d\theta \\
        &\le R^{d-2}|\mathbb{S}^{d-1}|\,\big[\phi^{-1}(\zeta)\big]_{\zeta=\phi(0+)}^{\zeta=\phi(\rho)}
        \le (2\bar V/\lambda)^{(d-2)/2}|\mathbb{S}^{d-1}|\,\rho,
    \end{align*}     
    where we have used that $\bar C-V(x_0+\overline r(\theta)\theta)=\bar C-\bar V=\phi(0+)$ by definition of $\bar V$, and $\bar C-V(x_0+\underline r(\theta)\theta)=\bar C-\bar V/2=\phi(\rho)$ by definition of $\rho$. 
    Thus, the integral is finite.
\end{proof}

A decay rate in $L^1(\R^d)$ can be obtained as a simple corollary of the previous theorem together with \eqref{mainconclusion2}:

\begin{corollary}
    Assume that the potential $V$ and the nonlinearity $P(u)$ satisfy {\bf (HV1')-(HV3'), (HP1)-(HP2)}, and {\bf (HPV)} and that the nolinearity $P$ satisfies the hypotheses of Theorem \ref{gencz}. Assume further that the initial data has finite relative entropy ${\mathcal H}_g(u_0|u_\infty)<\infty$, then
    \begin{align*}%\label{L1-decay}
        \|u(t)-\uinf\|_{L^1(\R^d)} \le C {\mathcal H}_g(u(t)|u_\infty)^{1/2} \leq C {\mathcal H}_g(u_0|u_\infty)^{1/2} \exp\left(-\lambda t\right),
    \end{align*}
    for all solutions $u$ of \eqref{eq:PDE0} with $\lambda$ from {\bf (HV2)}.
\end{corollary}
The previous generalization of the Csiszár-Kullback inequality shows the exponential convergence in $L^1_M$ of solutions to \eqref{eq:PDE0} superseding \cite[Theorem 31-32]{CJMTU}. 

\begin{example}\label{ex:comp-CJMTU}
In particular, our result improves the decay rates in \cite[Remark 34]{CJMTU} and \cite[Theorem 31-32]{CJMTU} for the degenerate power-law nonlinearities $P(r)=r^m$ with $m>2$ since it applies to the whole range $m>1$.
\end{example}

\subsection{Weighted moment estimates for degenerate diffusions}

In this subsection we shall derive functional inequalities for the moment control of some non-negative $u\in L^1(\R^d)$ in terms of the newly found entropies from \S\ref{sec-quasilin}.
\begin{prop}\label{prop:momentcontroldegenerate}
    Assume {\bf (HV1')-(HV3'), (HP1)-(HP2)}, and {\bf (HPV)}, and that $P$ is degenerate. 
    Given an increasing $C^3$ function $g: (\xi_{min},\infty)\to\R$ with $g(0)=0$, define the relative entropy ${\mathcal{H}}_g(u|\uinf)$ as in Definition \ref{genrelentropy2}. Then, for any $u\in L^1_M(\R^d)$,
    \begin{equation}\label{g-check}
        \int_{\R^d} U_g (x) u(x)\;dx \leq {\mathcal{H}}_g(u|\uinf),
        \quad\text{where}\quad
        U_g (x):=
        \begin{cases}
            g(\phi(0+)+V(x)-\bar C) & \text{for $x\notin B_\infty$}, \\
            0 & \text{for $x\in B_\infty$}.
        \end{cases}
    \end{equation}
\end{prop}
\begin{proof}
    Directly from the definition of $\tilde G$ in \eqref{eq:7b}, we have that
    \begin{align*}
        \tilde G(b,b;x) = 0, \quad
        \partial_a\tilde G(a,b;x) = g(\phi(a)+V(x)-\bar C),\quad
        \partial_a^2\tilde G(a,b;x) = g'(\phi(a)+V(x)-\bar C)\phi'(a).
    \end{align*}
    In particular, $\tilde G(a,b;x)$ is convex in $a$ since $g$ and $\phi$ are increasing, and so
    \begin{align*}
        {\mathcal{H}}_g(u|\uinf)  
        &= \int_{\R^d} \tilde G(u(x),\uinf(x);x)\;dx \\
        &\ge \int_{\R^d} \big[\tilde G(\uinf(x),\uinf(x);x) + \partial_a \tilde G(\uinf(x),\uinf(x);x) (u-\uinf)(x)\big]\;dx \\
        & = \int_{\R^d} g\big(\phi(\uinf(x))+V(x)-\bar C\big) (u-\uinf)(x)\;dx.
    \end{align*}
    Now split the last integral into the part inside $B_\infty$ --- where $\phi(\uinf(x))+V(x)=\bar C$ --- and outside of $B_\infty$:
    \begin{align*}
        {\mathcal{H}}_g(u|\uinf)  
        &\geq g(0) \int_{B_\infty}  (u-\uinf)(x)\;dx +  \int_{\R^d\setminus B_\infty} g(\phi(0+)+V(x)-\bar C) u(x)\;dx \\
        &=  \int_{\R^d} U_g (x) u(x)\;dx,
    \end{align*}
    where we have used that $g(0)=0$.
\end{proof}
\begin{example}
    For the standard entropy $g(\xi)=\xi$, \eqref{g-check} becomes
    \begin{align*}
        \int_{\R^d\setminus B_\infty} \big(V(x)-\bar V\big)u(x)\;dx \le {\mathcal{H}}_g(u|\uinf),
    \end{align*}
    where $\bar V:=\bar C-\phi(0+)=-\ximin$ is the value of $V$ on the edge of $\uinf$'s support, i.e., on $\partial B_\infty$.
\end{example}
\begin{theorem}\label{th:x-weigths}
    Under the hypotheses of  Proposition \ref{prop:momentcontroldegenerate} and Theorem \ref{gencz}, we suppose additionally that there are $q>0$ and $c>0$ such that $g(\xi)\ge c\xi^q$ for all sufficiently large $\xi$. Then, there is a constant $A$ such that for all $u\in L^1_M(\Rd)$:
    \begin{align*}
        \int_{\R^d} \big(1+|x|^2)^q \,|u-\uinf|(x)\;dx \le A\big[\mathcal{H}_g(u|\uinf)^{1/2}+\mathcal{H}_g(u|\uinf)\big].
    \end{align*}
\end{theorem}
\begin{proof}
    By $\lambda$-convexity of $V$, there is a constant $b$ such that $V(x)\ge \frac\lambda4|x|^2-b$. By definition of $\check g$ and the assumed growth behaviour of $g$, it follows further that $U_g(x)\ge a|x|^{2q}-b'$ for some constants $a>0$ and $b'$. Therefore, with yet another constant $a'>0$,
    \begin{align*}
        \big(1+|x|^2\big)^q 
        \le a' \big(1+U_g(x)\big) \quad \text{for all $x\in\Rd$}.
    \end{align*}
    It then follows that
    \begin{align*}
        \int_{\Rd} \big(1+|x|^2)^q \,|u-\uinf|(x)\;dx 
        \le a'\left(\int_{\Rd} |u-u_\infty|\,dx + \int_{\Rd} U_g(x) u(x)\,dx\right)\,.
    \end{align*}
    Finally use \eqref{eq:ourCK} and \eqref{g-check}.
\end{proof}
\begin{example}
    For the $m$-homogeneous nonlinearities with $m>1$, and for a corresponding $p$-entropy ${\mathcal H}_{m,p}$, we obtain from \eqref{eq:gm} that the exponent
    \begin{align*}
        q := 1+ \frac{p-1}{(m-1)\kappa_m} 
    \end{align*}
    is a possible choice in the assumptions of Theorem \ref{th:x-weigths}. More explicitly, we have
    \begin{align*}
        \int_{\R^d} \big(1+|x|^2)^{1+\frac{p-1}{(m-1)\kappa_m}} \,|u(x)-\uinf(x)|\;dx \le A_{m,p}\big[\mathcal{H}_{m,p}(u|\uinf)^{1/2}+\mathcal{H}_{m,p}(u|\uinf)\big]
    \end{align*}
    with suitable constants $A_{m,p}$, see  Proposition \ref{prop:genm} for properties of $\mathcal{H}_{m,p}(u|\uinf)$.
\end{example}

\subsection{Generalized Log-Sobolev inequalities}
Recall from our central result \eqref{mainconclusion3} that if $g$ generates an admissible entropy $\mathcal H_g$, then it is related to $J_g$ via
\begin{align}
    \label{eq:fncineq}
    \mathcal H_g(u|u_\infty) \le \frac1{2\lambda} J_g(u),
\end{align}
where the dissipation amounts to 
\begin{align*}
    J_g(u) = \int_{\R^d}u g'(\xi) |\nabla\xi|^2\,dx.
\end{align*}
Notice that \eqref{eq:fncineq} only holds for non-negative functions $u\in L^1(\R^d)$ that have the same total mass as $u_\infty$.
\begin{example}
    This follows up on the example from Section \ref{sct:xmp-pme2} on the quadratic nonlinearity $P(r)=r^2$. Recall that $u_\infty=(\frac{\bar C-V}{2}+1)_+$ and $\xi=2(u-1)+V -\bar C$ in that case. We have shown that for any $p\in[1,2]$, the function $g_p$ with
    \begin{align*}
        g_p'(z) = \left(1+\frac{p-1}\kappa\right)\left(\frac{2+z+\bar C}{2+\bar C}\right)^{\frac{p-1}\kappa}
    \end{align*}
    generates an admissible entropy $\mathcal H_p$, given in \eqref{H_p}. In this particular situation, setting $q:=2+\frac{p-1}\kappa$ for simplicity, the functional inequality \eqref{eq:fncineq} amounts to
    \begin{align*}
         \frac q{q-1}\int_{\R^d}\big[\big(2u+V)^q-\big(2u_\infty+V)^q\big]\,dx 
        \le \frac{4}\lambda \int_{\R^d}u\big|\nabla\big[(2u+V)^{q/2}\big]\big|^2\,dx.
    \end{align*}
    Notice that for $p=1$ (when $q=2$), one obtains
    \begin{align}
        \label{eq:funny}
        \int_{\R^d}\left(u^2+Vu\right)\,dx  - \int_{\R^d}\left(u_\infty^2+Vu_\infty\right)\,dx 
        \le \frac1{2\lambda} \int_{\R^d}u \big|\nabla(2u+V)\big|^2\,dx,
    \end{align}
    which is classical, see e.g. \cite{CJMTU}. Moreover, \eqref{eq:funny} is the special case for $m=2$ from the convex Sobolev inequalities associated to $P(u)=u^m$: 
    \begin{align*}
        \int_{\R^d}\left(\frac{u^m}{m-1}+Vu\right)\,dx - \int_{\R^d}\left(\frac{u_\infty^m}{m-1}+Vu_\infty\right)\,dx
        \le \frac1{2\lambda} \int_{\R^d} u\left|\nabla\left(\frac{mu^{m-1}}{m-1}+V\right)\right|^2\,dx.
    \end{align*}    
    For $p\in(1,2]$ and $m>1$ the resulting generalized Sobolev inequalities \eqref{eq:fncineq} are new. While the r.h.s.\ $J_g(u)$ is explicit, the l.h.s.\ $\H_{m,p}(u|u_\infty)$ is in general not explicit, but we already gave lower and upper estimates on it in \eqref{eq:orderent} inside Proposition \ref{prop:genm}.
\end{example}

%%%%%%%%%%%%%%%%%%%%%%%%%%%%%%%%%%%%%%%%%%%%%%%%%%%%%%%%%%%%%%%%%%%%%%%%%%%%%%%%
\section{Sharpness results}\label{sec:sharp}
\setcounter{equation}{0}

In this section we shall discuss non-trivial saturation in the functional inequality \eqref{mainconclusion3}, i.e.\ we analyze the possibility to find some function $\bar u\ne u_\infty$ with mass $M$ such that \eqref{mainconclusion3} becomes an equality. For reasons that we shall explain shortly, we limit ourselves to the case of the canonical entropy with $g(\xi)=\xi$, where the inequality \eqref{mainconclusion3} simplifies to
\begin{align}
    \label{eq:forsharpness}
    \intRd \big[\Phi(u)-\Phi(u_\infty) + V(u-u_\infty)\big]\,dx
    \le \frac1{2\lambda} \intRd u\big|\nabla(\phi(u)+V)\big|^2\,dx.
\end{align}
Finding a $\bar u\in L_M^+$ that saturates this inequality is equivalent to finding an initial condition $u(0)=u_0$ such that the decay estimate \eqref{mainconclusion2} is saturated for any $t\ge s\ge0$, or --- still equivalently, as we shall see --- the final estimate \eqref{finform1} is an equality at each time, with vanishing remainder $\mathcal R\equiv0$.

We shall fix a nonlinearity $P$ and mass $M$, and derive sufficient and necessary conditions on the potential $V$ (satisfying {\bf (HV2)} and the additional assumption $V\in C^2(\R^d)$) for saturation in \eqref{eq:forsharpness}. Below, we shall always assume the McCann condition \eqref{mccann}, which implies that the standard entropy on the left-hand side of \eqref{eq:forsharpness} --- generated by $g_1(\xi)=\xi$ --- is admissible, see Remark \ref{StandardEntropy}. To simplify the presentation, we shall further assume that $\beta=P'>0$ on $\R^+$, which is slightly more restrictive than {\bf (HP1)}.

Before carrying out this program, we briefly justify why the restriction to the standard entropy is reasonable. For linear Fokker-Planck equations of the form \eqref{linFP}, the question of saturation was analyzed in \S3.5 of \cite{AMTU}. There, it was found that optimality is possible only for logarithmic and quadratic entropies, with their prototypical generators $\psi_1$ and $\psi_2$ given in \eqref{psi-p}: these two generators turn the Bakry-\'Emery condition \eqref{e:2.11c} into an equality, and they correspond to the lower and upper bounds $y=f'=0$ and $1$, respectively, in the inequality \eqref{f-inequ}, see Lemma \ref{f-solutions}. While logarithmic entropies lead in \cite{AMTU} to non-negative optimal functions $\bar u$, the optimal functions for quadratic entropies (obtained as eigenfunctions of the linear Fokker-Planck operator in $L^2(\R^d;u_\infty)$) change sign --- which is outside of the setting considered here. It thus appears that only for the standard entropy, there is a chance of saturation in \eqref{mainconclusion3}.

We proceed in analogy to \S3.5 of \cite{AMTU}. Undoing the estimation in \eqref{eq:nonopt} we can rewrite \eqref{finform1} as 
\begin{eqnarray}\label{finform4}
  -\frac12\frac{d}{dt}J_g - \lambda J_g  
  &=&\int_{\R^d} ug'(\xi)\nabla\xi\cdot(\nabla^2 V-\lambda I_d) \cdot \nabla\xi \;dx
  + \int_{\R^d} u e^{f(\xi)} R(\alpha,\beta,\xi) \;dx \\
  &=&\int_{\{u>0\}} u\nabla\xi\cdot(\nabla^2 V-\lambda I_d) \cdot \nabla\xi \;dx
  + \int_{\{u>0\}} u \big(\alpha\|\nabla^2\xi\|^2 +(\beta-\alpha)(\Delta\xi)^2\big) \;dx \nonumber\\
  &=:& r(u(t)) =: r_1(u(t))+r_2(u(t)), \nonumber
\end{eqnarray}
where we used $g=g_1\equiv\xi$ and hence $f_1\equiv0$, and Remark \ref{StandardEntropy}. We emphasize that $r$ is always non-negative since both $r_1$ and $r_2$ are non-negative: the first because $V$ is $\lambda$-convex, and the second because we assume the McCann condition $\beta\ge\frac{d-1}d\alpha$, and $(\Delta\xi)^2\le d\|\nabla^2\xi\|^2$, see \eqref{CS} below.
Integrating \eqref{finform4} in time yields
\begin{equation}\label{error-int}
    \intRd \big[\Phi(u_0)-\Phi(u_\infty) + V(u_0-u_\infty)\big]\,dx
    = \frac1{2\lambda} \intRd u\big|\nabla(\phi(u_0)+V)\big|^2\,dx
  -\frac1{\lambda}\int_0^\infty r(u(t))\,dt,
\end{equation}
where $u(t)$ is the trajectory ``connecting" $u_0$ and $u_\infty$. By non-negativity of $r$, inequality \eqref{eq:forsharpness} saturates if and only if $r(u(t))=0$ along the entire trajectory.

We remark that, strictly speaking, the reasoning above applies a priori only to classical solutions, because the derivation of \eqref{finform4} involved a variety of integration by parts. If $u$ is only a weak solution (e.g.\ a compactly supported solution of a degenerate diffusion equation), one should proceed via approximation as discussed in the introduction, see \cite{otto,CJMTU} --- we shall not carry out that approximation here. We recall, as discussed in the introduction, that solutions for the class of equations of the form \eqref{GFP} are typically continuous functions in $(x,t)$ for all $t>0$ under the assumptions {\bf (HV1')-(HV3'), (HP1)-(HP2)}, and {\bf (HPV)}. Moreover, if the initial data is bounded, then solutions are uniformly bounded in time and space. This is assumed in the rest of the section.

As discussed above, saturation happens in \eqref{eq:forsharpness} if and only if $r_1(u(t))=0$ and $r_2(u(t))=0$ for a.e. $t>0$. This, in turn, happens if and only if the following three conditions hold on $(\supp u)^\circ\subset \R^d\times\R^+_t$, the interior of the $(x,t)$-support of $u$:
\begin{align}
    &\nabla\xi\cdot(\nabla^2 V-\lambda I_d) \cdot \nabla\xi=0\,, \label{sharp-cond1}\\
    & \nabla^2\xi = 0 \quad \mbox{on }\: \Omega_1:=\big\{(x,t)\in(\supp u)^\circ\,\big|\, \beta(u(x,t)) > \frac{d-1}{d} \alpha(u(x,t))\big\}\,,\label{sharp-cond2}\\
    & \nabla^2\xi = \sigma_2(x,t)I_d \quad \mbox{on }\: \Omega_2:=\big\{(x,t)\in(\supp u)^\circ\,\big|\, \beta(u(x,t)) =\frac{d-1}{d}  \alpha(u(x,t))\big\}\,,\label{sharp-cond3}
\end{align}
with an arbitrary scalar function $\sigma_2(x,t)$. For future reference we denote the time traces of $\Omega_1$ and $\Omega_2$ at any fixed time $t>0$ by $\Omega_1^t$, $\Omega_2^t\subset \R^d$, respectively.
While condition \eqref{sharp-cond1} is obvious from $r_1=0$, \eqref{sharp-cond2} and \eqref{sharp-cond3} need a justification in order to follow from $r_2=0$: We shall use the Cauchy-Schwarz inequality for symmetric matrices, 
\begin{equation}\label{CS}
  (\Delta \xi)^2=\big(\tr(I_d\,\nabla^2\xi)\big)^2 
  \le \|I_d\|^2 \|\nabla^2\xi\|^2 = d\|\nabla^2\xi\|^2\,,
\end{equation}
and hence $(\Delta \xi)^2=\theta\,d\|\nabla^2\xi\|^2$ with some scalar function $\theta=\theta(x,t)\in [0,1]$. Thus the remainder term \eqref{cjmtu}, i.e.\ the second factor in $r_2$, can be written as
\begin{equation}\label{R=0}
  R(\alpha,\beta,\xi) = [\alpha+(\beta-\alpha) d\theta] \,\|\nabla^2\xi\|^2 \ge\alpha(1-\theta) \|\nabla^2\xi\|^2\,,
\end{equation}
proceeding as in the proof of \cite[Theorem 11]{CJMTU}. In the last estimate we used the McCann condition \eqref{mccann}.
The remainder term in \eqref{R=0} is zero if $\nabla^2\xi=0$, i.e.\ condition \eqref{sharp-cond2}, or if $\beta =\frac{d-1}{d} \alpha$ along with $\theta=1$, which happens for equality in \eqref{CS}. The latter case corresponds to condition \eqref{sharp-cond3}.

Next we discuss the implications of the three above conditions: \eqref{sharp-cond1} will imply that $V$ is quadratic in at least one direction, see Lemma \ref{lem:9.1} below. \eqref{sharp-cond2} implies 
\begin{equation}\label{xi-form1}
   \xi(x,t)=\sigma_0+\sigma_1\cdot x,
\end{equation}
where the scalar function $\sigma_0=\sigma_0(t)$ and the vector function $\sigma_1=\sigma_1(t)$ are (for any fixed $t>0$) constant in $x$ on each connected component of $\Omega_1^t$. \eqref{sharp-cond3} implies 
\begin{equation}\label{xi-form2}
  \xi(x,t)=\sigma_0+\sigma_1\cdot x+\frac{\sigma_2}{2} |x|^2,
\end{equation}
and $\sigma_0=\sigma_0(t)$, $\sigma_1=\sigma_1(t)$, and $\sigma_2=\sigma_2(t)$ are (for any fixed $t>0$) constant in $x$ on each connected component of $\Omega_2^t$. Moreover, continuity of $u$, and hence of $\xi$, implies that $\sigma_j$, $j=0,1,2$, are continuous in $t$. In particular, $\sigma_j$ are independent on $\Omega_1^t$ and $\Omega_2^t$.

While the support of the initial condition, $\supp u_0$, may consist of several disconnected components, $u_\infty$ has a connected support due to the convexity of $V$. 
For the saturation analysis of \eqref{mainconclusion3} we shall consider here only initial data $u_0$ for which the positivity set $\{u_0>0\}$ is connected to simplify technicalities. We shall assume in the sequel that the positivity set $\{u(t)>0\}$ is connected for all $t\ge0$. This is motivated by the fact that an initially disconnected support of $u(.,t)$ would have to merge along the flow anyhow, as $u(t)$ converges to $u_\infty$ when $t\to\infty$. This qualitative behavior of the solutions, proven for the case of the porous medium equation with quadratic confinement \cite[\S18]{VazquezPME}, is not known up to our knowledge in this generality, although expected. 

Let us now discuss condition \eqref{sharp-cond3}  in some more detail: Since $\beta>0$ on $\R^+$, condition \eqref{sharp-cond3} is only relevant for $d\ge2$. Let McCann's equality in \eqref{sharp-cond3} now hold on a (maximal) interval $[u_1,u_2]$ of $u$-values. If $u_1<u_2$,\footnote{Actually it is possible that the McCann condition holds on $\R^+$, but as an equality only at one point, i.e.\ with $u_1=u_2>0$: E.g.\ consider $P(u)=1+(u-1)/2-(u-1)^2/8+(u-1)^3$ locally around $u=1$, for $d=2$.} then \eqref{finform3} implies that 
\begin{equation}\label{P-fast-diff}
  P(u)=p_0 u^{\frac{d-1}{d}},
\end{equation}
with some $p_0>0$, holds on the maximal interval $[u_1,u_2]$, i.e.\ the fast-diffusion equation with the minimal admissible exponent (cf.\ Proposition \ref{prop:genm}(b)). 
If the sets $\Omega_1^t$ and $\Omega_2^t$ touch for some fixed $t>0$, then $\partial\Omega_1^t\cap\partial\Omega_2^t$ is a level set of $u(t)$ with values either $u_1\geq 0$ or $u_2$. 

Let us distinguish some cases. If $\partial\Omega_1^t\cap\partial\Omega_2^t$ consists of isolated points (due to local extrema of $u(.,t)$), such points must belong to the set $\Omega_2^t$, and \eqref{xi-form1}, i.e. $\sigma_2(t)=0$, holds also there by continuity of $\xi(.,t)$. 
Otherwise, there is at least one accumulation point of $\partial\Omega_1^t\cap\partial\Omega_2^t$. Then the functions $\xi(.,t)$ from \eqref{xi-form1} and \eqref{xi-form2} coincide on this boundary set by continuity. From the analiticity of $\xi(.,t)$ hence, we conclude that the coefficients $\sigma_0(t)$, $\sigma_1(t)$, and $\sigma_2(t)$ in \eqref{xi-form1}, \eqref{xi-form2} must coincide. 
Hence, $\sigma_2(t)=0$ also in such a component of $\Omega_2^t$. Furthermore in this case, if $u_1>0$, continuity and finite mass of $u(.,t)$ imply that $u(.,t)$ must also take values in $[0,u_1)$ ``close to'' $\Omega_2^t$. Hence any component of $\Omega_2^t$ must touch $\Omega_1^t$, and hence $\sigma_2(t)=0$. 

By contrast, the quadratic term in \eqref{xi-form2} can only be present at time $t$ if \eqref{P-fast-diff} holds on some (non-trivial) interval $[0,u_2]$. Furthermore in this case, $u_2\ge \max u(.,t)$,  where the maximum is taken over the connected component of $\Omega_2^t$ with $\sigma_2(t)\ne0$. 
As we shall see in the proof of Lemma \ref{lem:9.2}, $\sigma_2(t)$ cannot become zero in finite time. Thus, when considering now the whole trajectory, $u_2$ must even satisfy $u_2\ge \sup_{\R^d\times\R^+} u(x,t)$. In this case, the evolution equation for the considered initial condition $u_0$ is purely the limiting fast-diffusion equation with confinement. Notice that $\supp u(t)=\R^d$ for $t>0$ in the case of the quadratic confinement, see for instance \cite{BV06} and the references therein. For general confinement potentials, the positivity of the solution is not known up to our knowledge, although expected.

Recall our assumption that the positivity set $\{u(t,\cdot)>0\}$ is connected for every $t\geq 0$. As a conclusion of the discussion above, saturation in \eqref{mainconclusion3} can only hold if either \eqref{xi-form1} or both \eqref{xi-form2} with $\sigma_2\ne0$ and \eqref{P-fast-diff} holds for all $t>0$, i.e.\ on $\{u>0\}$. This means that either $\Omega_1=\{u>0\}$, or $\Omega_2=\{u>0\}$ and \eqref{P-fast-diff} for all values attained by $u$.
Hence we shall now analyze the compatibility of the two solution forms \eqref{xi-form1} and \eqref{xi-form2} with the evolution equation \eqref{GFP} and with the terminal condition $u(t=\infty)=u_\infty$. Let us briefly anticipate the result: In the former case, saturation of the inequality \eqref{mainconclusion3} can hold only for the function $\bar u$ being a translate of $u_\infty$. In the latter case, also scaled versions of $u_\infty$ may saturate \eqref{mainconclusion3}.

We recall from {\bf (HV2)} that $V$ is assumed to attain its minimal value zero at $x=0$.
Let $B_\infty$ be the interior of $u_\infty$'s support. Note that $B_\infty$ is convex and contains the origin. $B_\infty$ is bounded for degenerate diffusion, and is $\R^d$ for non-degenerate diffusions. 

\subsection{Case 1: Strict McCann condition}
\begin{lemma}\label{lem:9.1}
    Assume that $P$ satisfies the strict McCann condition on some interval $(0,U]$, with $U\in(0,\infty]$. Assume the support of the initial datum $u_0$ is connected and that $u_0\neq u_\infty$.  If the corresponding solution $u$ is such that the sharpness condition $r(u)\equiv 0$ holds in \eqref{error-int} and that $u\le U$, then there are a unit vector $\textgoth{e}\in\R^d$ and a distance $r_0>0$ such that
  \begin{itemize}
  \item[(a)] $u_0$ is a translate of $u_\infty$ by $r_0\textgoth{e}$,
    \begin{align}
      \label{eq:u0opt}
      u_0(x)=u_\infty(x-r_0\textgoth{e})\quad \mbox{for } x\in\R^d;
    \end{align}
  \item[(b)] the solution $u$ is a translation of $u_0$ into $u_\infty$,
    \begin{align}
      \label{eq:uopt}
      u(x,t) = u_\infty(x-e^{-\lambda t}r_0\textgoth{e})\quad \mbox{for } x\in\R^d;
    \end{align}
  \item[(c)] the potential $V$ satisfies 
    \begin{align}
      \label{eq:Vopt}
      V(x) = \frac\lambda2(\textgoth{e}\cdot x)^2 
      + \tilde V\big(x-(\textgoth{e}\cdot x)\textgoth{e}\big)
      \quad \text{for all}\quad x\in B:=\bigcup_{0<r<r_0}(B_\infty+r\textgoth{e}).
    \end{align}
    Here $\tilde V$ is an arbitrary, $\lambda$-convex function of the $\textgoth{e}$-orthogonal component of $x$.
 \end{itemize}  
\end{lemma}
We remark that this set  $B$ is just the linear interpolation between the two ``endpoint sets", i.e.\ $\supp u_0$ and $\supp u_\infty$.\\

\begin{proof}
  Due to the strict McCann condition \eqref{sharp-cond2} for all $0<u\leq U$, we have for all $t>0$ and all $x$ in the support of $u(\cdot,t)$ that
  \begin{align}
    \label{eq:add-1}
    \xi(x,t) = \sigma_0(t)+\sigma_1(t)\cdot x\,.
  \end{align}
  The evolution equation $\partial_tu=\nabla\cdot(u\nabla\xi)$ now implies that
  \begin{equation*}%\label{eq:9.1}
    \partial_tu = \sigma_1(t)\cdot\nabla u \quad \mbox{on }\Omega_1=(\supp u)^\circ.  
  \end{equation*}
Thus, $u$ satisfies this transport equation along the time-dependent vector field $\sigma_1$ which, for each $t$ fixed, is constant w.r.t. $x\in\Omega_1^t$. 
Since $u$ converges to $u_\infty$, it thus holds, for all $t\ge0$ and $x\in\R^d$:
  \begin{align*}
    u(x,t) = u_\infty(x-\gamma_t), \quad\text{with}\quad \gamma_t=\int_t^\infty\sigma_1(s)ds,
  \end{align*}
  and in particular, $u_0$ is of the form \eqref{eq:u0opt} with $r_0\textgoth{e}=\gamma_0$.
  By hypothesis, $u_0\neq u_\infty$, and thus $\gamma_0\neq0$.
  Note that integrability of $\sigma_1$ at infinity follows from our a priori information that $u$ converges to $u_\infty$, and thus $\gamma_t\to0$ as $t\to\infty$.
  
  For brevity, introduce $\Xi[u](x) := \phi(u(x))+V(x)-\bar C$,
  so that $\xi(x,t) = \Xi[u(t)](x)$.
  Since $\Xi[u_\infty]\equiv0$ on $B_\infty$, we obtain
  \begin{align}\label{sigma01}
    \sigma_0(t)+\sigma_1(t)\cdot x
    = \xi(x,t)
    &= \Xi[u(t)](x)
      = \phi(u_\infty(x-\gamma_t))+V(x)-\bar C \\
    & = \Xi[u_\infty](x-\gamma_t)-\big(V(x-\gamma_t)-\bar C\big)+\big(V(x)-\bar C\big)
      = V(x) - V(x-\gamma_t), \nonumber
  \end{align}
  for all $x\in B_\infty+\gamma_t$. Thus $V$ on $B_\infty+\gamma_t$ is identical to $V$ on $B_\infty$, up to an affine correction, and in particular, for all $x\in B_\infty+\gamma_t$,
  \begin{align}
    \label{eq:Vhomogen}
    \nabla^2V(x) = \nabla^2V(x-\gamma_t).
  \end{align}
  Differentiation of \eqref{sigma01} in $t$ yields
  \begin{align*}
    \dot\sigma_0(t) + \dot\sigma_1(t)\cdot x
    = -\sigma_1(t)\cdot\nabla V(x-\gamma_t),
  \end{align*}
  and a subsequent differentiation in $x$ yields
  \begin{align}
    \label{eq:add-777}
    \dot\sigma_1(t) = -\nabla^2V(x-\gamma_t)\sigma_1(t) = -\nabla^2V(x) \sigma_1(t),
  \end{align}
  where the last equality follows from \eqref{eq:Vhomogen}. Recalling that $\nabla\xi(x,t)=\sigma_1(t)$ because of the special form \eqref{eq:add-1} of $\xi$, we can conclude by means of condition \eqref{sharp-cond1} that, for all $x\in B_\infty+\gamma_t$,
  \begin{align}
    \label{eq:add-13}
    \sigma_1(t)\cdot\nabla^2V(x)\cdot\sigma_1(t) = \lambda|\sigma_1(t)|^2.
  \end{align}
  Since $\nabla^2V(x)$ is a symmetric matrix bounded below by $\lambda I_d$, it follows that $\sigma_1(t)$ is an eigenvector for the eigenvalue $\lambda$, i.e.,
  \begin{align*}
    %\label{eq:add-93}
    \nabla^2V(x)\sigma_1(t) = \lambda\sigma_1(t).
  \end{align*}
  Substitute this into \eqref{eq:add-777} to conclude that $\dot\sigma_1(t) = -\lambda\sigma_1(t)$, and consequently,
  \begin{align*}
    \sigma_1(t) = e^{-\lambda t}\sigma_1(0).
  \end{align*}
  This shows that $\gamma_t=e^{-\lambda t}\gamma_0$, proving \eqref{eq:uopt}. In particular, $\gamma_t=r_0e^{-\lambda t}\textgoth{e}$ always points in the same direction $\textgoth{e}$, which implies that
  \begin{align*}
    B:=\bigcup_{t>0}\operatorname{supp}u(\cdot,t) = \bigcup_{0<r<r_0}\big(B_\infty+r\textgoth{e}\big).
  \end{align*}
  The relation \eqref{eq:add-13} implies for all $x\in B$ that
  \begin{align*}
    \textgoth{e}\cdot\nabla^2V(x)\cdot\textgoth{e} = \lambda.
  \end{align*}
  Note that $\textgoth{e}$ is an eigenvector associated to $\lambda$ of the symmetric matrix $\nabla^2V(x)$.
  Since $\nabla^2V(x)\ge\lambda I_d$, it further follows that $\textgoth{e}'\cdot\nabla^2V(x)\cdot \textgoth{e}=0$ at every $x\in B$, for all vectors $\textgoth{e}'\in\R^d$ that are orthogonal to $\textgoth{e}$. This implies that $V$ is indeed of the form \eqref{eq:Vopt}.
\end{proof}
We remark that, in Lemma \ref{lem:9.1}, the potential $V$ may also be quadratic in more than one direction.

\subsection{Case 2: Equality in the McCann condition}
\begin{lemma}\label{lem:9.2}
  Assume the marginal case $P(r)=p_0r^{1-1/d}$ with $d\ge2$, and that the support of the initial datum $u_0$ is connected and that $u_0\neq u_\infty$.
  If the solution $u$ 
  satisfies the sharpness condition $r(u)\equiv0$, 
  then there are a unit vector $\textgoth{e}\in\R^d$, a distance $r_0>0$, and a scaling factor $s_0\in\R^+
  $ such that
  \begin{itemize}
  \item[(a)] $u_0$ is a scaled translate of $u_\infty$,
    \begin{align}
      \label{eq:u0opt2}
      u_0(x)=s_0^{-d}u_\infty\left(\frac{x-r_0\textgoth{e}}{s_0}\right)\quad \mbox{for } x\in\R^d;
    \end{align}
  \item[(b)] the solution $u$ is a scaling and translation of $u_0$ into $u_\infty$,
    \begin{align}
      \label{eq:uopt2}
      u(x,t) = (1+(s_0-1)e^{-\lambda t})^{-d}u_\infty\left(\frac{x-r_0e^{-\lambda t}\textgoth{e}}{1+(s_0-1)e^{-\lambda t}}\right)\quad \mbox{for } x\in\R^d.
    \end{align}
  \item[(c)] If $s_0\neq 1$ the potential $V$ is perfectly quadratic,
    \begin{align}
      \label{eq:Vopt2}
      V(x)=\frac\lambda2|x|^2,
      \qquad \mbox{and} \qquad u_\infty(x) = \left(\frac{\bar C}{p_0(1-d)}+1+\frac{\lambda}{2(d-1)p_0}|x|^2\right)^{-d}\quad \mbox{for } x\in\R^d,
    \end{align}
    otherwise, i.e.\ for $s_0=1$, $V$ satisfies \eqref{eq:Vopt} for all $x\in\R^d$.
 \end{itemize}  
\end{lemma}

\begin{remark}
The form of the nonlinearity $P$ in Lemma \ref{lem:9.2} can be slightly generalized: It would be enough that $P(r)=p_0r^{1-1/d}$ holds only of some interval $[0,U]$ with $U\ge u_\infty^M:=\max_{x\in\R^d} u_\infty(x)$. In this case the scaling factor would have to satisfy $s_0\ge \Big(\frac{u_\infty^M}{U}\Big)^{1/d}$.
\end{remark}
\begin{proof}[Proof of Lemma \ref{lem:9.2}]
  First notice that, for the given $P$, we have accordingly
  \begin{align*}
    \phi(r) = (d-1)p_0(1-r^{-1/d}).
  \end{align*} 
  For any $t>0$, the  condition \eqref{sharp-cond3} implies that
  \begin{align*}
    %\label{eq:add-4}
    \xi(x,t) = \sigma_0(t)+\sigma_1(t)\cdot x + \frac{\sigma_2(t)}2|x|^2,       \quad \mbox{on }\Omega_2=(\supp u)^\circ.
  \end{align*}
Here, $\sigma_j$, $j=0,1,2$, are constant in $(x,t)\in\Omega_2$ being connected. Since $\xi(x,t):=\phi(u(x,t))+V(x)-\bar C<\infty$ on $(x,t)\in\Omega_2$ and recalling \eqref{def-varphi}, we hence deduce
\begin{equation}\label{eq:u-form2aux}
  u(x,t)=\varphi\big(\xi(x,t)-V(x)\big)
  \,, 
\end{equation}
for all $(x,t)\in\Omega_2$. Notice that for $P(r)=r^m$, $0<m<1$, the function $\varphi(z)$ defined in \eqref{def-varphi} is positive, see \eqref{phi-inv}.
Taking into account that solutions to \eqref{eq:PDE0} are assumed to be continuous for $t>0$, see Section 2, then \eqref{eq:u-form2aux} holds on $x\in\R^d$ and $t>0$ or equivalently $\Omega_2=\R^d\times (0,\infty)$, otherwise there would be a jump discontinuity at the boundary of $\Omega_2^t$ for some $t>0$.

  We start by identifying the possible potentials $V$. First, assume that $\sigma_2(t)=0$ for all $t\geq 0$. We can apply the same proof as in Lemma \ref{lem:9.1}, and thus the potential satisfies \eqref{eq:Vopt}. Moreover, the conclusions about the solution $u$ in \eqref{eq:u0opt} and \eqref{eq:uopt} carry over, coinciding with \eqref{eq:u0opt2} and \eqref{eq:uopt2}  for $s_0=1$.

  Now, assume that there exists $t_0>0$, such that $\sigma_2(t_0)\neq 0$ and define $x_0:=-\sigma_1(t_0)/\sigma_2(t_0)$. We now show that the potential $V$ has the simple shape \eqref{eq:Vopt2}. Condition \eqref{sharp-cond1} implies that
  \begin{align}
    \label{eq:add-15}
    [\sigma_1(t_0)+\sigma_2(t_0)x]\cdot\nabla^2V(x) \cdot[\sigma_1(t_0)+\sigma_2(t_0)x] = \lambda|\sigma_1(t_0)+\sigma_2(t_0)x|^2\, \mbox{ for all } x\in \Rd.
  \end{align}
  Let $\textgoth{e}\in\R^d$ be an arbitrary unit vector, and consider for $s\in\R$ the point $x=x_0+s\textgoth{e}$. Then $\sigma_1(t_0)+\sigma_2(t_0)(x_0+s\textgoth{e})=s\sigma_2(t_0)\textgoth{e}$, and we obtain from \eqref{eq:add-15} that
  \begin{align*}
\textgoth{e}\cdot\nabla^2V(x_0+s\textgoth{e})\cdot \textgoth{e} = \lambda
\quad \forall\,s\in\R.
  \end{align*}
  Since $\textgoth{e}$ is an arbitrary unit vector, this means that $V$ is an exact parabola with coefficient $\lambda$ on each line through $x_0$. By smoothness of $V$, all these parabolas have the same value and continuous first derivatives at $s=0$ with respect to $\textgoth{e}$. Since we assumed that $V$ takes its minimum at $x=0$, it follows that $V(x)=\frac\lambda2|x|^2$ due to the convexity assumption on $V$ in {\bf (HV2)}, see all full details in Lemma \ref{lem-parabolic} in the appendix.

Now, we want to identify the solution $u$ for the case in which $\sigma_2(t)$ is not identically zero. Since $\xi:=\phi(u)+V(x)-\bar C$ on $\R^d$ and recalling \eqref{def-varphi}, we hence deduce
\begin{equation}\label{eq:u-form2a}
  u(x,t)=\varphi\big(\xi-V(x)\big)
  =\varphi\left(\sigma_0(t)+\sigma_1(t)\cdot x+(\sigma_2(t)-\lambda)\frac{|x|^2}{2}\right)\,.
\end{equation}
Next we shall determine the time evolution of the $\sigma_j$'s such that $u$ satisfies the evolution equation $\partial_tu=\nabla\cdot(u\nabla\xi)$ or equivalently
  \begin{align}
    \label{eq:add-14a}
    \partial_tu = \nabla\cdot\big(u\,(\sigma_1+\sigma_2x)\big)      \quad \mbox{on }\R^d\times (0,\infty).
  \end{align}
Plugging \eqref{eq:u-form2a} into \eqref{eq:add-14a}, we get
\begin{eqnarray}\label{diff-eqa}
 0&=& \partial_t u-\nabla\cdot(u\nabla\xi) = \partial_t u-\nabla u\cdot(\sigma_1+\sigma_2 x)-d\sigma_2 u \nonumber\\
 &=& \varphi'\left(\sigma_0(t)+\sigma_1(t)\cdot x+(\sigma_2(t)-\lambda)\frac{|x|^2}{2}\right) \\
 &&\times \Big(\dot\sigma_0(t)+\dot\sigma_1(t)\cdot x+ \dot\sigma_2(t)\frac{|x|^2}{2}
 -(\sigma_1(t)+(\sigma_2(t)-\lambda) x)\cdot(\sigma_1(t)+\sigma_2(t) x)\Big)-d\sigma_2(t) u, \nonumber
\end{eqnarray}
where $\dot\sigma_j$ denotes the time derivative. We also recall from \eqref{phi-inv} that $\varphi'(z)>0$ for $u=\varphi(z)>0$ and use \eqref{power-constants} to obtain
$$
  \varphi(z)=\Big(\frac{z+\bar C}{p_0(1-d)}+1\Big)^{-d}
  =-\frac{1}{d} \varphi'(z) \,\big(z+\bar C+p_0(1-d)\big)\,.
$$
Thus we conclude from \eqref{diff-eqa}, using $\varphi'>0$, that
\begin{equation*} %\label{diff-eq1}    
    \dot \sigma_0+\dot\sigma_1\cdot x+\dot\sigma_2\frac{|x|^2}{2}-(\sigma_1+(\sigma_2-\lambda) x)\cdot(\sigma_1+\sigma_2 x)
    +\sigma_2\Big(\sigma_0+\sigma_1\cdot x+(\sigma_2-\lambda)\frac{|x|^2}{2} +\bar C+p_0(1-d)\Big) =0\,
\end{equation*}
holds for all $x\in\R^d$ and $t> 0$.
The coefficients of the different $x$-powers then yield the following ODEs:
\begin{eqnarray*}
    && \dot\sigma_0 =|\sigma_1|^2
    -\sigma_2\big(\sigma_0
    +\bar C+p_0(1-d)\big), \\
    && \dot\sigma_1=\sigma_1(\sigma_2-\lambda), \\
    && \dot\sigma_2=\sigma_2(\sigma_2-\lambda)\,.
\end{eqnarray*}
Let us point out that a posteriori $\sigma_2(t)\neq 0$ for all $t>0$ since $\sigma_2(t_0)\neq 0$.

Disregarding the solution $\sigma_2\equiv0$ (which is ruled out in this case), the general solution of this ODE system is
\begin{eqnarray}
    && \sigma_0(t) =c_0(e^{-\lambda t}-c_2)-\frac{\big(c_3(c_2e^{\lambda t}-1)+\lambda|c_1|^2\Big)^2}{2c_2\lambda|c_1|^2(c_2e^{\lambda t}-1)e^{\lambda t}}, \label{sigma0}\\
    && \sigma_1(t)=\frac{\lambda c_1}{1-c_2 e^{\lambda t}}, \nonumber \\ % \label{sigma1}\\
    && \sigma_2(t)= \frac{\lambda}{1-c_2 e^{\lambda t}}\,, \label{sigma2}
\end{eqnarray}
with arbitrary constants $c_0,\,c_2\in\R\setminus\{0\}$, $c_1\in\R^d$, and $c_3:=\bar C+p_0(1-d)$. This analysis shows that $\sigma_j$, $j=0,1,2$, are given by \eqref{sigma0}-\eqref{sigma2}. Note that the solutions to the ODE system \eqref{sigma0}-\eqref{sigma2} are well defined up to $t=0$, thus the initial data must be of the form \eqref{eq:u-form2a}.

Since $u(t)\stackrel{t\to\infty}{\longrightarrow} u_\infty$ in \eqref{eq:u-form2a}, then the condition 
\begin{equation}\label{u-at-infty-a}
    u(x,t=\infty)=u_\infty(x)=\varphi(-V(x)),  
\end{equation}
implies that $\sigma_j(t)\stackrel{t\to\infty}{\longrightarrow}0,\,j=0,1,2$. This implies
\begin{equation*} %\label{c_j}
   c_0=-\frac{c_3^2}{2\lambda c_2|c_1|^2},\qquad \mbox{and hence } \quad \sigma_0(t)=\frac{e^{-\lambda t}}{c_2} \Big[\frac12 |c_1|^2 \frac{\lambda}{1-c_2 e^{\lambda t}}-c_3\Big]\,,
\end{equation*}
with $c_1\in\R^d$, $c_2\in(-\infty,0)\cup(1,\infty)$.
We also see that $\sigma_2(t)<\lambda$, and $\sigma_2(t)$ cannot become zero in finite time, as mentioned earlier. 

We finally want to rewrite the solution $u$ in terms of $u_\infty$. This could be done by working on the explicit, although involved, formula for $u$ given in \eqref{eq:u-form2a}, however we will find it in an alternative and shorter manner. With this objective, we first realize that if $\theta(t)$ and $\gamma(t)$ are defined by
\begin{align*}
    %\label{eq:add-18}
    \theta(t) = \int_t^\infty\sigma_2(s)\,d s,
    \quad
    \gamma(t)= \int_t^\infty e^{\theta(t)-\theta(s)}\sigma_1(s)\,d s,
  \end{align*}
then all solutions of the transport equation \eqref{eq:add-14a} satisfying \eqref{u-at-infty-a} can be represented as
\begin{align}
    \label{eq:add-92}
    u(x,t) = e^{-d\theta(t)}u_\infty\big(e^{-\theta(t)}(x-\gamma(t))\big).
\end{align}
This can easily be checked by direct inspection. Thus, we conclude that
\begin{align*}
    %\label{eq:add-18f}
    \theta(t) = \log\left(1-\frac{1}{c_2}e^{-\lambda t}\right)
    \qquad \mbox{and} \qquad
    \gamma(t)= -\frac{c_1}{c_2}e^{-\lambda t}.
  \end{align*}
Plugging these formulas into \eqref{eq:add-92}, we get
$$
u(x,t) = (1+(s_0-1)e^{-\lambda t})^{-d}u_\infty\left(\frac{x-r_0e^{-\lambda t}\textgoth{e}}{1+(s_0-1)e^{-\lambda t}}\right)
$$
with $s_0=1-\frac{1}{c_2}$ and $r_0\textgoth{e}=-\frac{c_1}{c_2}$.
\end{proof}

\begin{remark}\
\begin{enumerate}
\item[(a)] We note that the instantaneous positivity on the solution is known in the case of quadratic confinement \cite{BV06} combined with a time-dependent rescaling. This positivity property, although expected for general confinement potentials, is not present in the literature, so we do not assume it a priori in the statement.
\item[(b)] For the case $s_0=1$ in Lemma \ref{lem:9.2}, the solution $u$ is simply translated in the direction of $\textgoth{e}$, so parallel directions do not influence each other. This explains that the result is weaker than for $s_0\neq 1$, and thus, the potential $V(x)$ is only identified in the $\textgoth{e}$ direction.
\end{enumerate}
\end{remark}

\subsection{Summary of sharpness results}
In the following result we summarize the above discussion. Part (a) is the analog of \cite[Theorem 3.11]{AMTU} for linear Fokker-Planck equations with the logarithmic entropy.

\begin{theorem}\label{th:sharp}
Let $d\ge1$, and let $V\in C^2(\R^d)$ satisfy {\bf (HV1)},  {\bf (HV2)} with $\lambda>0$ as the largest possible constant, and assume {\bf (HPV)}. Let $P$ satisfy {\bf (HP1)},  {\bf (HP2)} as well as the McCann condition \eqref{mccann}, and $P'>0$ on $\R^+$. Let mass $M>0$ be fixed, and define $u_\infty^M:=\max_{x\in\R^d} u_\infty(x)$. Then, saturation of the convex Sobolev inequality \eqref{mainconclusion3} with $g=g_1$ holds if and only if the following two conditions (i) and (ii) hold -- under the (simplifying) assumption that the support of $\bar u$ is connected. Depending on $P$ these two conditions are formulated separately:
\begin{enumerate}
    \item[(a)] If $\beta(u)>\frac{d-1}{d}\alpha(u)$ on $(0,U]$ with some $U\ge u_\infty^M$:
    \begin{enumerate}
        \item[(i)] $\bar u(x)=u_\infty(x-r_0\textgoth{e})$ for some unit vector $\textgoth{e}\in\R^d$ and some $r_0>0$.
        \item[(ii)] $V$ satisfies \eqref{eq:Vopt}.
    \end{enumerate}
    \item[(b)] If $\beta(u)=\frac{d-1}{d}\alpha(u)$ on $[0,U]$ for some $U\ge u_\infty^M$ (and hence $d\ge2$):
     Either the conditions (a-i) and (a-ii) hold, or the following two (with $s_0\ne1$):
    \begin{enumerate}
        \item[(i)] $\bar u(x)=s_0^{-d}u_\infty\left(\frac{x-r_0\textgoth{e}}{s_0}\right)$ with some unit vector $\textgoth{e}\in\R^d$, an $r_0\ge0$, and some $s_0\ge \Big(\frac{u_\infty^M}{U}\Big)^{1/d}$. 
        \item[(ii)] $V(x)=\frac{\lambda}{2}|x|^2$ holds on $\R^d$.
    \end{enumerate}
\end{enumerate}
\end{theorem}

\begin{proof}
Let $u=u(t)$ be the trajectory with the initial condition $u(0)=\bar u\in L^1_M$, with $\bar u$ appearing in \eqref{mainconclusion3}. 

\noindent
\underline{Case (a):} For the backward direction assume that the conditions (a-i) and (a-ii) hold. Then the proof of Lemma \ref{lem:9.1} shows that $u(x,t)$ takes the form \eqref{eq:uopt} and $\xi(x,t)$ is of the form \eqref{xi-form1} with the vector $\sigma_1=\nabla \xi$ being aligned with the $\textgoth{e}$-direction. Then \eqref{sharp-cond1} and \eqref{sharp-cond2} hold on $(\supp u)^\circ$ and $r(u)=0$ on $\R^+_t$ follows. Hence \eqref{mainconclusion3} saturates.

Next we consider the forward direction. As established before, saturation implies that the conditions \eqref{sharp-cond1}, \eqref{sharp-cond2}, and \eqref{sharp-cond3} hold on $(\supp u)^\circ$, or equivalently the conditions \eqref{sharp-cond1} as well as  \eqref{xi-form1} on $\Omega_1$,  and  \eqref{xi-form2} on $\Omega_2$. The above discussion on the case when $\Omega_1^t$ touches $\Omega_2^t$ implies that \eqref{xi-form1} must actually hold on all of $(\supp u)^\circ$. Then Lemma \ref{lem:9.1} gives the result.\\
    
\noindent
\underline{Case (b):} The backward direction is analogous to before: Here $u(x,t)$ takes the form \eqref{eq:uopt2} and $\xi(x,t)$ is of the form \eqref{xi-form2}. Then \eqref{sharp-cond1} and \eqref{sharp-cond3} hold on $(\supp u)^\circ$ and $r(u)=0$ on $\R^+_t$ follows. Hence \eqref{mainconclusion3} saturates.

For the forward direction we recall from the proof of Lemma \ref{lem:9.2} that either $\sigma_2\equiv0$ on $(\supp u)^\circ$ and hence condition \eqref{xi-form1} holds or $\sigma_2\ne0$ everywhere on $(\supp u)^\circ$  and hence condition \eqref{xi-form2} holds. In the former case Lemma \ref{lem:9.1} gives the result, in the latter case Lemma \ref{lem:9.2}.
\end{proof}

Saturation of the convex Sobolev inequality in the previous theorem has to be understood for functions that satisfy the assumptions of Lemma \ref{lem:9.1} and \ref{lem:9.2} and the validity of the derivation of the identity \eqref{finform4}.

In the following example we shall verify for the porous medium equations that the ``optimal functions'' $\bar u$ from Theorem \ref{th:sharp} indeed saturate the functional inequality \eqref{mainconclusion3} for the standard entropy. This verification is particularly useful for degenerate diffusions since the derivation of \eqref{finform4} was based on formal computations.

\begin{example}
    Let $P$ satisfy the conditions of Theorem \ref{th:sharp}. 
\begin{enumerate}
    \item[(a)]     
    Assume that (possibly after a rotation) the coordinates are such that $V(x)=\frac{\lambda}{2}x_1^2 +W(x_2,...,x_d)$.  
    Moreover let $\bar u(x):=u_\infty(x-\bar x)$ for some $\bar x=r_0 e_1$, with $r_0\in\R$ and $e_1\in\R^d$ the first unit vector. 
From \eqref{stat1a} and \eqref{def-varphi} we have
$$
  u_\infty(x)=\varphi(-V(x))=\bar \phi^{-1}(\bar C-V(x)).
$$  
For $\bar u$ and $x\in \supp(\bar u)$ we have 
$$
  \xi:=\phi(u_\infty(x-\bar x))+V(x)-\bar C
  =-V(x-\bar x)+V(x) = \frac{\lambda}{2}(2r_0x_1-r_0^2),
$$
and hence $\nabla \xi=\lambda \bar x$.
From Example \ref{sam1} and Proposition \ref{key} (in the non-degenerate case) or Example \ref{Ex:canon-entropy}  (in the degenerate case) as well as \eqref{tderent2} we have
\begin{eqnarray*}
  \mathcal H_{g_1}(\bar u|u_\infty) &=& 
  \int_{\R^d} \big(\Phi(\bar u)-\Phi(\uinf)
+ V\,(\bar u-\uinf)\big)(x)\;dx \\
  &=&  \int_{\R^d} [V(x+\bar x)-V(x)] u_\infty(x)\,dx =\frac{\lambda}{2} r_0^2 M\,, \\
  J_{g_1}(\bar u)&=&\int_{\bar u>0} \bar u|\nabla\xi|^2\,dx=\lambda^2 r_0^2 M\,,
\end{eqnarray*}
where we used for $\mathcal H_{g_1}$ the  
symmetry $u_\infty(-x_1,x_2,...,x_d)=u_\infty(x)$ 
(inherited from $V(x)$). This verifies equality in \eqref{mainconclusion3}.
    \item[(b)] Assume that $V(x)=\frac{\lambda}{2} |x|^2$. Then the equality follows by direct inspection.
\end{enumerate}
Summing up we note that $r(u(t))\equiv 0$ in both cases. Hence, the functional inequality \eqref{mainconclusion3} saturates.
\end{example}

%%%%%%%%%%%%%%%%%%%%%%%%%%%%%%%%%%%%%%%%%%%%%%%%%%%%%%%%%%%%%%%%%%%%%%%%%%%%%%%%
\appendix

\section{A result on the shape of potentials}\label{sec-appendix-aux}
\setcounter{equation}{0}

\begin{lemma}\label{lem-parabolic}
Let a function $V\in C^2(\R^d)$ satisfy {\bf (HV2)}. For one fixed $x_0\in\R^d$, let $V$ also satisfy
\begin{equation}\label{line-cond}
  \textgoth{e}\cdot\nabla^2V(x_0+s\textgoth{e})\cdot \textgoth{e} = \lambda
\quad \forall\,\textgoth{e}\in \SSS^{d-1},\:\forall\, s\in\R,
\end{equation}
where $\SSS^{d-1}$ denotes the $d-1$ dimensional sphere. Then, 
\begin{equation*} %\label{V-quadr}
  V(x)=\frac\lambda2 |x|^2 \quad \mbox{on } \Rd.
\end{equation*}
\end{lemma}
\begin{proof}
Due to assumption \eqref{line-cond}, the function
$$
  W(x):= V(x) - \frac\lambda2 |x|^2 + (\lambda x_0- \nabla V(x_0) )\cdot (x -x_0)
$$
satisfies 
$$
  \textgoth{e} \cdot \nabla^2 W(x_0 + s\textgoth{e}) \cdot \textgoth{e} = 0 \quad \forall\,\textgoth{e}\in \SSS^{d-1},\:\forall\, s\in\R,
  \quad \mbox{and } \nabla W(x_0)=0. 
$$
Integration along each ray through $x_0$ shows that  $W$ is constant along each ray. As they all intersect, continuity implies this is a unique constant on all $\R^d$.

It thus follows that $V(x) = \frac\lambda2 |x|^2 + a\cdot x + b$. Condition {\bf (HV2)} implies $a=b=0.$
\end{proof}

\section{List of symbols}\label{sec-appendix-symbols}
\begin{itemize}
\item $P:\setR^+\to\setR^+$ is a given increasing pressure function. $P$ is \emph{degenerate} if $\phi(0+)>-\infty$,
  and is \emph{non-degenerate} if $\phi(0+)=-\infty$\,.
\item $V:\Rd\to\R$ is a given $\lambda$-convex potential, with $\lambda>0$.
\item We define auxiliary functions $\alpha,\beta,\phi,\Phi:\setR^+\to\setR$ by
  \begin{align*}
    \alpha(r) = \frac{P(r)}r, \quad \beta(r)=P'(r), \quad
    \phi(u) = \int_1^u\frac{P'(r)}r\,dr,
    \quad
    \Phi(u) = \int_0^u\phi(r)\,dr.
  \end{align*}
\item $\overline{\phi}^{-1}$ is the generalized inverse of $\phi$,
  \begin{align*}
    \overline{\phi}^{-1}(z) =
    \begin{cases}
      0 & \text{if $z\le\phi(0+)$}, \\
      \infty & \text{if $z\ge\phi(\infty)$}, \\
      \phi^{-1}(z) & \text{otherwise}.
    \end{cases}
  \end{align*}
\item $\varphi(z) := \overline{\phi}^{-1}(z+\tilde C)$, with
  \begin{align*}
    \tilde C =
    \begin{cases}
      \bar C & \text{if $P$ degenerate}, \\
      0 & \text{if $P$ non-degenerate}.
    \end{cases}
  \end{align*}
\item The steady state is
  \begin{align*}
    u_\infty(x) = \overline{\phi}^{-1}(\bar C-V(x))=\varphi(-V(x)),
  \end{align*}
  where $\bar C$ is such that $u_\infty$ is of prescribed mass $M>0$. 
\item $\xi$ and $\tilde\xi$ are pressure variables,
  \begin{align*}
    \xi & = \phi(u)-\phi(u_\infty) \quad \text{if $P$ is non-degenerate}, \\
    \tilde\xi &= \phi(u)+V \quad \text{in general}.
  \end{align*}
  Note that $\tilde\xi - \xi = V-\phi(u_\infty) = \bar C$ if $P$ is non-degenerate,
  and that the evolution equation becomes
  \begin{align*}
    \partial_tu = \nabla\cdot(u\nabla\tilde\xi) .
  \end{align*}
\item We define
  \begin{align*}
    \mu := \beta-\frac{d-1}d \alpha,\quad
    \kappa := 1+\frac{\beta-\alpha}{8\alpha}\frac{d(\beta-\alpha)-8\alpha}{d(\beta-\alpha)+\alpha},
  \end{align*}
  and note that $\mu>0$ implies
  \begin{align*}
    \kappa \ge \bar\kappa_d = 1-\frac1{2d}.
  \end{align*}
\item Relative entropies under consideration are
  \begin{align*}
    \mathcal{H}_g(u|u_\infty) &= \int G(u(x),u_\infty(x))\,dx
                                \quad \text{if $P$ is non-degenerate}, \\
    \tilde{\mathcal{H}}_g(u|u_\infty) &= \int\tilde G(u(x),u_\infty(x);x)\,dx
                                \quad \text{if $P$ is non-degenerate},    
  \end{align*}
  where 
  \begin{align*}
    G(a,b) &= \int_b^a g(\phi(s)-\phi(b)) \, ds,\\
    \tilde G(a,b;x) &= \int_b^a g(\phi(s)+V(x)-\bar C) \, ds. 
  \end{align*}  
\item The different versions of entropy generators are related as
  \begin{align*}
    g' = e^f, \quad y = f'.
  \end{align*}
\end{itemize}

%================== ACKNOWLEDGEMENTS =======================================
\bigskip
\subsection*{Acknowledgements}
AA was partially supported by the Austrian Science Fund (FWF) project \href{https://doi.org/10.55776/F65}{10.55776/F65}.
JAC was supported by the Advanced Grant Nonlocal-CPD (Nonlocal PDEs for Complex Particle Dynamics: Phase Transitions, Patterns and Synchronization) of the European Research Council Executive Agency (ERC) under the European Union Horizon 2020 research and innovation programme (grant agreement No. 883363), and partially supported by the EPSRC EP/V051121/1 and by the ``Maria de Maeztu'' Excellence Unit IMAG, reference CEX2020-001105-M, funded by MCIN\slash AEI \slash10.13039\slash501100011033\slash. We acknowledge the hospitality of Isaac Newton Institute during the program on kinetic equations, CIRM-Marseille (on ``Aggregation-Diffusion Equations \& Collective Behavior'') and of ESI, Vienna where part of this work was done.

%================== BIBLIOGRAPHY ======================================
\bibliography{genentropies}
\bibliographystyle{abbrv}

\end{document}